\documentclass{article}
\usepackage[utf8]{inputenc}
\usepackage{xcolor}
\usepackage{geometry}
\usepackage{amsmath, amssymb, amsfonts, amsthm}
\usepackage{indentfirst}
\usepackage{cancel}
\usepackage{bm}
\usepackage{enumitem}
\usepackage{setspace}
\usepackage{hyperref}
\usepackage[all]{xy}
\usepackage{tikz}
\usepackage{diagbox}
\usepackage{array}
\usepackage{tabularx}
\usepackage{mathtools}
\usepackage{dsfont}
\usepackage{amsmath,calligra,mathrsfs}
\usepackage{adjustbox}
\usepackage{authblk}
\usepackage[textsize=tiny]{todonotes}
\usepackage[capitalize]{cleveref}
\usepackage{lipsum}
\usetikzlibrary{cd}
\newtheoremstyle{case}{}{}{}{}{}{:}{ }{}
\theoremstyle{case}
\newtheorem{case}{Case}
\usepackage{forloop}

\usepackage{tocloft}
\newcommand*{\sheafhom}{\mathscr{H}\kern -.5pt om}
\newcommand*{\sheafext}{\mathscr{E}\kern -.5pt xt}
\newcommand*{\sheafend}{\mathscr{E}\kern -.5pt nd}

\newcommand{\Ext}{\mathrm{Ext}}

\newcommand{\Hom}{\mathrm{Hom}}
\newcommand{\rhom}{\mathrm{RHom}}

\newcommand{\Syz}{\mathrm{Syz}}
\newcommand{\Ker}{\mathrm{Ker}}

\DeclareMathOperator{\Fil}{Fil}
\DeclareMathOperator{\gr}{gr}

\hypersetup{
	colorlinks=true,
	linkcolor=blue,
	linkbordercolor={0 0 1}}

\theoremstyle{plain}
\newtheorem{thm}{Theorem}[section] 
\newtheorem{lemma}[thm]{Lemma} 
\newtheorem{Corollary}[thm]{Corollary} 
\newtheorem{proposition}[thm]{Proposition} 

\theoremstyle{definition}
\newtheorem{construction}[thm]{Construction} 
\newtheorem{definition}[thm]{Definition} 
\newtheorem{notation}[thm]{Notation}

\newtheorem{remark}[thm]{Remark} 

\newcommand{\blank}{\underline{\;\;}}

\usepackage{listings}
\lstdefinelanguage{Julia}%
{morekeywords={abstract,break,case,catch,const,continue,do,else,elseif,%
		end,export,false,for,function,immutable,import,importall,if,in,%
		macro,module,otherwise,quote,return,switch,true,try,type,typealias,%
		using,while},%
	sensitive=true,%
	morecomment=[l]\#,%
	morecomment=[n]{\#=}{=\#},%
	morestring=[s]{"}{"},%
	morestring=[m]{'}{'},%
}[keywords,comments,strings]%
\usepackage[doi=true,isbn=false,url=false,eprint=false,maxbibnames = 50]{biblatex}

\usepackage{xurl}
\usepackage[T1]{fontenc}
\usepackage{csquotes}
\usepackage[american]{babel}

 \newbibmacro*{bbx:parunit}{%
	   \ifbibliography
	     {\setunit{\bibpagerefpunct}\newblock
		      \usebibmacro{pageref}%
		      \clearlist{pageref}%
		      \setunit{\adddot\par\nobreak}}
	     {}}

 \renewbibmacro*{doi+eprint+url}{%
	   \usebibmacro{bbx:parunit}
	   \iftoggle{bbx:doi}
	     {\printfield{doi}}
	     {}%
	   \iftoggle{bbx:eprint}
	     {\usebibmacro{eprint}}
	     {}%
	   \iftoggle{bbx:url}
	     {\usebibmacro{url+urldate}}
	     {}}

\title{Full Exceptional Sequence for a Fine Quiver Moduli Space}
\author{Svetlana Makarova, Junyu Meng}

\begin{document}
	\maketitle
	\begin{abstract}
		We consider the fine quiver moduli space of representations of the 3-Kronecker quiver
		of dimension vector $(2,3)$, which is a blow down of the Hilbert scheme of 3 points on $\mathds{P}^2$.  A short description of its geometry and Chow ring is given. Then we exhibit an exceptional sequence for the derived category by understanding a $\mathds{P}^1$-bundle over it and using Teleman Quantization. The fullness of the exceptional sequence is proved by using a covering argument and computations of mutations.  
	\end{abstract}

	\tableofcontents

	\newpage

	\section{Introduction}
	Quiver moduli spaces were first constructed by King in \cite{king} to geometrically deal with the problem of classification of representations of finite dimensional algebras. Using the GIT technique introduced by Mumford, the fine moduli space of stable representations of a quiver  with a fixed dimension vector can be constructed. Moreover, a universal family of representations can be constructed over the fine moduli space, uniquely up to a twist.
	Under some mild assumptions, a lot of aspects of quiver moduli space have been understood in recent works, for example, the Picard group and more generally the Chow ring, desingularizations, the rationality problem, framing construction, the Brauer group, the Euler characteristic, the cohomology of endomorphism bundles of the universal bundles and so on. The interested readers are invited to check \cite{Fanoquivermoduli,HNquiver,belmans2023vector,belmans2023chow,Reincke} for more detailed information.
	In the paper \cite{Fanoquivermoduli}, a systematic way of constructing Fano varieties in the framework of quiver moduli spaces is proposed, and it gives many interesting examples. In low dimension, apart from the classical flag varieties, there exists also some interesting nontrivial examples which are Fano quiver moduli spaces.
	
	
	In this paper, we consider the moduli space of representations of the 3-Kronecker quiver $\begin{tikzcd}
		\bullet \arrow[r,bend left] \arrow[r,bend right] \arrow[r] & \bullet
	\end{tikzcd}$ of dimension vector $(2,3)$. This moduli space, which we denote by $Y$, is a smooth prime Fano 6-fold of index 3 and provides ``the smallest'' quiver moduli space that is not a Grassmannian. The variety $Y$ and its higher dimensional analogue have aroused interest before, e.g. in the works \cite{EPS, Franzen}. Furthermore, $Y$ has some other descriptions:
	
	$\bullet$ It is a blow down of the Hilbert scheme of 3 points on $\mathds{P}^2$ by \cite{IlievManivel}; 
	
	$\bullet$ It is isomorphic to the variety of trisecant planes of the Segre embedded $\mathds{P}(W)$ in the Segre embedding $\mathds{P}(W)\hookrightarrow\mathds{P}(S^2W)$ by \cite{IlievManivel};
	
	$\bullet$ It is isomorphic to the height-zero moduli space $M_{\mathds{P}^2}(4,-1,3)$ of stable sheaves on $\mathds{P}^2$ with $(r,c_1,c_2)=(4,-1,3)$ by \cite{drezet}.
	
	Apart from the interesting descriptions of the quiver moduli space $Y$ itself, we also have a concrete description of the universal quiver representation $\mathcal{U}_1\rightarrow\mathcal{U}_2\otimes W$ over $Y$:
	In fact, there exist embeddings of $Y$ into $\mathrm{Gr}(2,6)$, $\mathrm{Gr}(3,8)$ and the universal bundles $\mathcal{U}_1$, $\mathcal{U}_2$ are pullbacks of universal subbundles over the two Grassmannians respectively.

	
	This paper is mainly concerned with the derived category of $Y$.
	There is a resolution of diagonal for fine quiver moduli spaces (see \cref{resolution} and \cref{general dercat}) that uses Schur modules of (duals of) the universal bundles.
	Motivated by this, we aim to find a full exceptional sequence consisting of such kind of Schur modules.
	\begin{thm}[\cref{theorem: exceptional collection on Y}, \cref{proposition: mutating slU1}, \cref{fullnessthm}]
		\label{main theorem}
		Let $Y$ be as above, i.e. the moduli space of representations of the $3$-Kronecker quiver with dimension vector $(2,3)$.
		Let $\mathcal O_Y (1) := \det \mathcal U_1^\vee$. The following collection of objects in $D^b(Y)$ is a strong full Lefschetz collection:
		\begin{align*}
			\left\langle 
			\mathcal{O}_Y, \mathcal{U}_2^*, \mathcal{U}_1^*, \mathcal{U}_2(1), \mathfrak{sl}(\mathcal{U}_1)(1), \
			\mathcal{O}_Y(1), \mathcal{U}_2^*(1), \mathcal{U}_1^*(1), \mathcal{U}_2(2),  \
			\mathcal{O}_Y(2), \mathcal{U}_2^*(2), \mathcal{U}_1^*(2), \mathcal{U}_2(3)\right\rangle.
		\end{align*}
	\end{thm}
	\paragraph{Outline of the Proof.}The vanishing of the involved cohomology comes from two important facts.
	First, using Teleman Quantization, which we explain in \cref{subsection: Teleman quantization}, one can prove vanishing of higher cohomology groups of many bundles; and using Hirzebruch-Riemann-Roch and the code in \cref{appendix: Sage code}, one can compute the Euler characteristics.
	Secondly, it is an important observation, first suggested by Alexander Kuznetsov, that the $\mathds{P}^1$-bundle $X=\mathds{P}_Y(\mathcal{U}_1)$ can be identified with the blow up of $\mathds{P}^7$ along $\mathds{P}^2\times \mathds{P}^2$.
	We explain this in \cref{subsection: interpretation of X as a blowup}, and point out that
	the $\mathds{P}^1$-bundle $X$ is another fine quiver moduli space, which comes from the framing construction.
	We then describe the universal bundles over $X$ by relating them with commonly seen sheaves over the blow up of $\mathds{P}^7$ along $\mathds{P}^2\times \mathds{P}^2$.
	Since the derived pullback along a $\mathds{P}^1$-bundle map is fully faithful, we can calculate the cohomology of bundles over $Y$ by calculating the cohomology of the pullback bundles over $X$.
	
	To streamline some calculations, we make use of the fact that $Y$ is a prehomogeneous variety under the action of $SL(W)$, where $W$ is a three-dimensional vector space freely generated by the three arrows in the 3-Kronecker quiver $\begin{tikzcd}
		\bullet \arrow[r,bend left] \arrow[r,bend right] \arrow[r] & \bullet
	\end{tikzcd}$.
	Then we can realize the spaces of morphisms between exceptional bundles as $SL(W)$-modules.
	
	To prove the fullness of the exceptional sequence above, we need to make use of another description of the variety $Y$, namely that it is the zero locus of a general global section of $Q^*(1)$ over $\mathrm{Gr}(2,\mathds{C}^8)$, where $Q$ is the universal quotient bundle. For a general subspace $\mathds{C}^7\subset\mathds{C}^8$, we consider the intersection of $Y$ with $\mathrm{Gr}(2,\mathds{C}^7)$ and denote this intersection by $J$.
	In fact, $J$ is a hyperplane section of the $\mathrm{G}_2$-Grassmannian $\mathrm{G}_2\mathrm{Gr}(2,7)$,
	and its derived category has been described in \cite{hyperplaneDer}.
	When the subspace $\mathds{C}^7$ varies, the set of such smooth $J$ covers $Y$. The standard covering argument shows that we only need to prove two extra bundles belong to the triangulated subcategory generated by the above exceptional sequence.
	
	We prove that the two extra bundles really belong to the triangulated subcategory generated by our exceptional sequence by showing that these two bundles can be obtained by mutating some exceptional bundles across several other exceptional bundles in the exceptional collection. This is done by looking at the mutation process step by step, calculating the cohomology using Teleman Quantization and the interpretation of
	the $\mathds{P}^1$-bundle $X$ and using a small trick by first passing to a smaller subvariety as follows.

	
	There is a subvariety of $Y$ which is isomorphic to the blow up of $\mathds{P}^2$ at three points, which we denote by $Bl(\mathds{P}^2)$. Over $Bl(\mathds{P}^2)$, the restrictions of the universal representation morphism and some other interesting bundles split into direct sums and are somewhat `simple' in the sense that any endomorphism of a restriction bundle is of constant rank.
	Moreover, a subgroup of $SL(W)$ which is isomorphic to the alternating group $A_3$ acts on $Bl(\mathds{P}^2)$ and endows $Bl(\mathds{P}^2)$ with some basic symmetry.
	When considering the $SL(W)$-action on $Y$, one finds that the subvariety $Bl(\mathds{P}^2)$ contains some points in the biggest open orbit and some points in the minimal closed orbits.
	
	The step-by-step inspection of the mutation process is finally done by understanding the restriction of the mutation process to $Bl(\mathds{P}^2)$ and then recovering the original mutation process over $Y$ using the good properties of $Bl(\mathds{P}^2)$ mentioned above.  
	
	
	\paragraph{Structure of the article.}
	In \cref{section: quiver moduli spaces}, we recall some basic facts about fine quiver moduli spaces, fix some notations and introduce Teleman Quantization, which will enable us to show some vanishing of higher cohomology on our GIT quotient space $Y$.
	In \cref{section: describing Y}, we recall that $Y$ is a smooth prime Fano 6-fold of index 3, collect many of its previously known descriptions and analyze the $SL(W)$-action orbits.
	An explicit description of the Chow ring is given.
	At the end of the section, we explore another quiver moduli space $X=\mathds{P}_Y(\mathcal{U}_1)$ closely related to $Y$.
	After all these preparations, we exhibit the exceptional collection \eqref{1 exc} in \cref{section: derived category} and prove that this exceptional sequence is full.
	We start the section with a sequence of calculations that show exceptionality, and in \cref{subsection: mutations} together with \cref{subsection: Completion of mutations}, we conduct a sequence of mutations that yield more exceptional collections.
	In \cref{subsection: proving fullness}, we use the covering argument to prove fullness of the exceptional collection in \cref{main theorem}.
	Finally, in \cref{appendix: Sage code}, we present Julia code that the reader is advised to use to calculate certain Euler characteristics.

	\paragraph{Related works.}
	The geometry of the moduli space $Y$ has been studied before in \cite{drezet, EPS, IlievManivel, Franzen, belmans2023chow}.
	
	In an upcoming paper \cite{Quantum} by the second author, the famous Dubrovin's Conjecture is verified for $Y$, which says that the derived category of a smooth Fano variety admits a full exceptional collection if and only if the big quantum cohomology ring is generically semisimple. Moreover, by doing essentially one more mutation, an $\mathrm{Aut}(Y)=(PGL(W)\rtimes\mathds{Z}_2)$-invariant full Lefschetz collection is derived, which verifies the refined Dubrovin's Conjecture proposed by Kuznetsov and Smirnov in \cite[Conjecture 1.3]{Kuznetsov_Smirnov_2021}.

		\paragraph{Acknowledgements.}
		The authors want to thank Alexander Kuznetsov and Pieter Belmans for helpful discussions.
		S.M. would like to express gratitude to Alexander Efimov for supervising the work on proving the exceptionality of the sequence as an undergraduate and then masters thesis in 2014--2017;
		S.M. owes a great deal of mathematical knowledge and early research experience to Alexander's supervision.
		S.M. would also like to thank Davesh Maulik for useful discussions and interest in the project.
		J.M. wants to thank his advisors Thomas Dedieu and Laurent Manivel for leading him to this interesting topic and for their useful suggestions and revisions of the article. J.M. also want to thak Gianni Petrella for his help with the software QuiverTools.

	\section{Quiver Moduli Spaces}
	\label{section: quiver moduli spaces}

	\subsection{Classical Construction}
	\label{subsection: classical quiver moduli}
	
	In this subsection, we recall several basic facts about quiver moduli space, including the construction and universal representations; we refer to \cite{Reincke} for further reading.
	
	We assume that $Q=(Q_0,Q_1)$ is an acyclic quiver (i.e. without oriented cycles) throughout this article. Usually we use $i \in Q_0$ to denote a vertex in the quiver and $[a:i\rightarrow j]\in Q_1$ to denote an arrow in the quiver.  When considering the moduli problem of representations of $Q$, we fix a dimension vector $\underline{d}\in\mathds{N}^{Q_0}$. The component of the dimension vector $\underline{d}\in\mathds{N}^{Q_0}$ corresponding to $i\in Q_0$ is denoted by $d_i$.
	We let $R_{\underline{d}}$ denote $\prod_{a:i\rightarrow j}\Hom(\mathds{C}^{d_i},\mathds{C}^{d_j})$ and let $GL(\underline{d})$ denote $\prod_{i\in Q_0}GL(\mathds{C}^{d_i})$. The algebraic group $GL(\underline{d})$ acts on $R_{\underline{d}}$ by conjugation.
	The subtorus $((\alpha.\mathrm{Id}_{d_i})_{i\in Q_0})_{\alpha\in \mathds{C}^*}\subset GL(\underline{d})$ acts trivially on $R_{\underline{d}}$ and we have an induced action of $G(\underline{d})=GL(\underline{d})/((\alpha.\mathrm{Id}_{d_i})_{i\in Q_0})_{\alpha\in \mathds{C}^*}$ on $R_{\underline{d}}$.
	
	To construct the moduli space by the GIT machine, we fix a stability parameter $\theta$, i.e. an element in $\mathds{Z}^{Q_0}$ such that $\sum_{i\in Q_0} d_i.\theta_i=0$. We associate to $\theta$ a character $\chi_{\theta}$ of the algebraic group $GL(\underline{d})$ 
	by putting $\chi_{\theta}((g_i)_{i\in Q_0})=\prod_{i\in Q_0}\det(g_i)^{-\theta_i}$,
	and note that it factors through $G(\underline{d})$.
	Naturally, $\chi_{\theta}$ can be viewed as a $G(\underline{d})$-linearization of the trivial line bundle over $R_{\underline{d}}$, and one can construct the GIT quotient using this ample linearization.
	
	The GIT (semi)stability with respect to the above linearization is proved in \cite{king} to be equivalent to the following (semi)stability:	
	\begin{definition}
		A representation $V$ of the quiver $Q$ of dimension vector $\underline{d}$ is called ($\theta$-)semistable if for any subrepresentation $W\subset V$, we have  $\sum_{i\in Q_0}\dim(W_i) \cdot \theta_i\leq0$, and is ($\theta$-)stable if the equality is achieved for exactly two subrepresentations, namely $0$ and $V$.
	\end{definition}
	
	We assume that $\underline{d}$ is coprime to $\theta$, i.e. for any subdimension vector $0<\underline{e}< \underline{d}$,
	we have $\sum_{i\in Q_0}e_i.\theta_i\neq 0$. This implies that $\underline{d}$ is an indivisible vector and that stability with respect to the above linearization coincides with semistability.
	
	Under this assumption, the GIT quotient space $M$ is 
	a smooth projective variety, and the quotient map $R_{\underline{d}}^{\theta-stable}\rightarrow M$ is a principal $G(\underline{d})$-bundle map, where $R_{\underline{d}}^{\theta-stable}\subset R_{\underline{d}}$ is the open subscheme of (semi)stable representations. 
	
	We can construct the universal quiver representation over $M$ so that we can view $M$ as a fine moduli space of 
	stable representations of $Q$: We pick $\underline{a}\in\mathds{Z}^{Q_0}$ such that $\sum_{i\in Q_0}a_i.d_i=1$ and let $L(\underline{a})$ denote the trivial line bundle equipped with the linearization induced by $\chi_{\underline{a}}$. For each vertex $i\in Q_0$ we consider the trivial vector bundle $\mathds{C}^{d_i}$ together with $GL(\underline{d})=\prod_{i\in Q_0}GL(\mathds{C}^{d_i})$ acting on it via the matrix component in $ GL(\mathds{C}^{d_i})$. If we tensor this $GL(\underline{d})$-linearized vector bundle $\mathds{C}^{d_i}$ with $L(\underline{a})$, then the subtorus $((\alpha.\mathrm{Id}_{d_i})_{i\in Q_0})_{\alpha\in \mathds{C}^*}\subset GL(\underline{d})$ acts trivially on the resulting vector bundle and we have a $G(\underline{d})$-linearized vector bundle which can be descended to a vector bundle $\mathcal{U}_i$ over $M$. 
	
	An important observation is that the above construction depends on the
	choice of 
	$\underline{a}$, and and $\mathcal{U}_i$ differs by a twist for different choices of $\underline{a}$. On the other hand, one can easily see that there are some tautological $GL(\underline{d})$-equivariant morphisms between trivial bundles $\mathds{C}^{d_i},\mathds{C}^{d_j}$ corresponding to arrows $a:i\rightarrow j$ over the space $R_{\underline{d}}$ and these morphisms descend to $M$, giving rise to the tautological morphisms between universal bundles $\mathcal{U}_i$. These $\mathcal{U}_i$ together with the morphisms between them satisfy the universal property up to a twist.
	
	Another observation is that, the whole $GL(\underline{d})$ acts trivially on the line bundle $\bigotimes_{i\in Q_0}(\det(\mathds{C}^{d_i}\otimes L(\underline{a})))^{\otimes a_i}$ and thus by descent theory we know $\bigotimes_{i\in Q_0}(\det(\mathcal{U}_i))^{\otimes a_i}\cong \mathcal{O}_M$, i.e. $\sum_{i\in Q_0}a_i.\mathrm{c}_1(\mathcal{U}_i)=0$ in the Chow ring of $M$.

	\subsection{Harder-Narasimhan Stratification}

	The definition of (semi)stability given in the last subsection can be formulated as slope stability. We fix the stability parameter $\theta$ such that $\sum_{i\in Q_0}d_i.\theta_i=0$ as before. For any nonzero dimension vector $\underline{e}\in\mathds{N}^{Q_0}$, we define the slope to be $$\mu(\underline{e})=\frac{\sum_{i\in Q_0}e_i.\theta_i}{\sum_{i\in Q_0}e_i}.$$
	If $W$ is a quiver representation of dimension vector $\underline{e}$, we will sometimes denote $\mu(W)$ by $\mu(\underline{e})$. By assumption, we have $\mu(\underline{d})=0$ and the original (semi)stability can be reformulated as:
	\begin{definition}
		A representation $V$ of the quiver $Q$ of dimension vector $\underline{d}$ is called ($\theta$-)semistable iff for any subrepresentation $W\subset V$ whose dimension vector is denoted by $\underline{e}$, we have  $\mu(\underline{e})\leq\mu(\underline{d})$ and is ($\theta$-)stable if the above inequality is always strict for any nontrivial subrepresentation.
	\end{definition} 
	
	Just as in the discussion of slope (semi)stability for vector bundles over a fixed curve, we also have the Harder-Narasimhan filtration for representations of quivers.
	\begin{definition}
		A filtration $0=W^0\subset  W^1\subset\cdots\subset W^{l-1}\subset W^l=V$ of a quiver representation $V$ is called a Harder-Narasimhan (HN) filtration if the subquotient representations $W^s/W^{s-1}$ are semistable for $s=1,2,\cdots,l$ and $\mu(W^1/W^0)>\mu(W^2/W^1)>\cdots>\mu(W^l/W^{l-1})$.
	\end{definition}
	\begin{lemma}[\cite{Reincke}]
		Every quiver representation $V$ has a unique Harder-Narasimhan filtration.
	\end{lemma}
	As a result, we can associate the dimension vectors of the subquotient representations appearing in the unique Harder-Narasimhan filtration to the quiver representation $V$.
	Letting $\underline{d}^s$ be the dimension vector of $W^s/W^{s-1}$, we obtain a tuple of dimension vectors $\tau=(\underline{d}^1,\underline{d}^2,\cdots,\underline{d}^l)$, which is a Harder-Narasimhan type in the sense of the following definition. 
	\begin{definition}
		A tuple of dimension vectors $\tau=(\underline{d}^1,\underline{d}^2,\cdots,\underline{d}^l)$ is called a Harder-Narasimhan type if $\sum_{s=1,2,\cdots,l}\underline{d}^s=\underline{d}$, $\mu(\underline{d}^1)>\mu(\underline{d}^2)>\cdots>\mu(\underline{d}^l)$ and for each dimension vector $\underline{d}^s$, there exists at least one semistable representation of $Q$ with dimension vector $\underline{d}^s$.
	\end{definition}
	
	It will be helpful to construct a stratification of the space of representations $R_{\underline{d}}$ based on the Harder-Narasimhan types associated to quiver representations. For a Harder-Narasimhan type $\tau$, we consider 
	$$S_{\tau}=\{\text{representation } V\in R_{\underline{d}} \mid \text{the associated Harder-Narasimhan type of }V\text{ is equal to }\tau\}.$$
	For the Harder-Narasimhan type $\tau=(\underline{d})$, it is obvious that $S_{\tau}$ is exactly the semistable locus in $R_{\underline{d}}$.
	
	Recall that an element in $R_{\underline{d}}$ is viewed as a quiver representation with $V_i$ being $\mathds{C}^{d_i}$ and the morphisms between $V_i$ and $V_j$ are given by the corresponding matrix components of $R_{\underline{d}}=\prod_{a:i\rightarrow j}\Hom(\mathds{C}^{d_i},\mathds{C}^{d_j})$. For the Harder-Narasimhan type $\tau$, we fix a filtration $0=V_i^0\subset V_i^1\subset\cdots\subset V_i^l=V_i=\mathds{C}^{d_i}$of the underlying vector space $\mathds{C}^{d_i}$, such that $\dim(V_i^s/V_i^{s-1})=d_i^s$. We consider 
	$$\Sigma_{\tau} = 
	\left\{
	V\in R_{\underline{d}}
	\,\middle|\,
	\begin{array}{l}
		\forall s: \text{the underlying tuple of vector spaces of}     \\
		\text{the $s$-th step in the HN filtration is }
		(V^s_i)_{i\in Q_0}
	\end{array}
	\right\}.$$
	
	Obviously, $S_{\tau}=GL(\underline{d}).\Sigma_{\tau}$. It is proven in \cite{HNquiver} that $S_{\tau}$ is a locally closed subvariety, and we obtain a stratification of $R_{\underline{d}}$ indexed by the set of Harder-Narasimhan types.
	
	There are more delicate structures hidden in this stratification which we refrain from discussing here. We introduce another important subvariety of $R_{\underline{d}}$ to end this subsection:
	$$Z_{\tau}=\{\text{representation } V\in R_{\underline{d}} \mid V\in \Sigma_{\tau},\text{ $V$ splits as the direct sum of subquotients in the HN filtration}\}.$$

	\subsection{Teleman Quantization}
	\label{subsection: Teleman quantization}

	It will be helpful in the sequel that we can show the vanishing of the higher cohomology of some vector bundles constructed by descending $G(\underline{d})$-linearized vector bundles over $R_{\underline{d}}$. The powerful machine that we have in hand is Teleman Quantization.
	
	In a slightly more general setting, we consider a reductive group $G$ (for example, $GL(\underline{d})$) acting linearly on a vector space $R$ (for example, $R_{\underline{d}}$) together with a $G$-linearized trivial line bundle which allows us to do the GIT quotient construction $R^{\text{ semistable}}\rightarrow R//G$.
	
	The famous Hesselink stratification is a stratification of the space $R$ such that the only open stratum is the open subset of semistable points.
	The strata contained in the unstable (i.e. ``not semistable'') locus are given by an instability analysis as follows.
	We first fix a norm $\|*\|$ on the set of 1-parameter subgroup (1-ps) of $G$ by picking a maximal subtorus $T$ of $G$ and taking a Weyl group invariant inner product over $\mathrm{Y}(T)\otimes_{\mathds{Z}}\mathds{R}$, where $\mathrm{Y}(T)$ is the lattice of 1-ps of $T$,
	such that the inner product always takes integral values on $\mathrm{Y}(T)$. Then for any 1-ps $\lambda$ of $G$, some conjugate copy $\lambda'$ of it actually lies in $T$, and we define the norm of $\lambda$ to be the norm of $\lambda'$.
	
	We divide the usual Hilbert-Mumford weight $\mu(x,\lambda)$ by the norm of the 1-ps to get the so-called normalized Hilbert-Mumford weight $\bar{\mu}(x,\lambda)=\frac{\mu(x,\lambda)}{\|\lambda\|}$.
	
	The unstable locus in $R$ will be divided into several $G$-invariant strata by associating to each unstable point in $R$ a set of one parameter subgroups of $G$ which are the most responsible for the instability of the point.
	\begin{proposition}
		For each unstable point $x\in R$, the set of primitive one parameter subgroups $\lambda$ such that $\lim_{t\rightarrow0}\lambda(t).x$ exists, for which  $\bar{\mu}(x,\lambda)$ reaches the minimum among all the one parameter subgroups $\eta$ of $G$ such that $\lim_{t\rightarrow0}\eta(t).x$ exists is nonempty. All the 1-ps in this set form an orbit under the conjugate action of a parabolic subgroup $P(x)\subset G$.
	\end{proposition} 
	\begin{remark}
		In the above Proposition, it is enough to consider primitive one parameter subgroups, since $\bar{\mu}(x,\lambda)=\bar{\mu}(x,n\lambda), \forall n\in \mathds{N}^*$.
	\end{remark}
	
	We call the set mentioned in the above Proposition to be the \emph{set of optimal 1-ps}
	for the point $x\in R$. As stated in the Proposition, optimal 1-ps are conjugate to each other under the action of a parabolic subgroup $P(x)$, and if we want to construct $G$-invariant strata inside the unstable locus, we naturally consider $G$-conjugate classes of primitive 1-ps.
	\begin{definition}
		Let $[\lambda]$ be a conjugate class of primitive one parameter subgroups under the action of $G$. We define 
		$$S_{[\lambda]} = 
		\left\{ x \mid
		\text{the set of optimal primitive 1-ps for $x$ is contained in the conjugate class }[\lambda] \right\} . $$
	\end{definition}
	
	\begin{definition}
		\label{definition: Z_lambda}
		Let $\lambda$ be a primitive one parameter subgroup of $G$.
		We define $\Sigma_{\lambda}=\{x \mid \lambda \text{ is optimal for }x \}$ and $Z_{\lambda}=\{x \mid x \in \Sigma_{\lambda} \text{ and } x \text{ is fixed by }\lambda\}$.
	\end{definition}
	
	It is obvious that $S_{[\lambda]}=G.\Sigma_{\lambda}$. One can prove that only finitely many $G$-conjugate classes $[\lambda]$ of primitive 1-ps contain the optimal set of primitive 1-ps for some unstable $x\in R$ and for each such $[\lambda]$, $S_{[\lambda]}$ is a locally closed smooth subvariety of $R$. Moreover, these $S_{[\lambda]}$ together with the semistable locus give a stratification for the space $R$, which is called the Hesselink stratification. For more details of the above description and the deeper structure of the stratification, the reader is invited to consult \cite{Hoskins} and \cite{Kirwan}.
	
	The following version of Teleman Quantization Theorem is taken from \cite{HalpernL}, where a stronger statement is given:
	\begin{thm}[Teleman]
		Let G be a reductive group acting linearly on an affine space $R$ and assume that it acts freely on the semistable locus (with respect to a chosen linearization).  Assume that all the $Z_{\lambda}$ defined above are connected. For each primitive one parameter subgroup $\lambda$ which appears in the stratification, we define $\eta_{\lambda}$ to be the weight of action of $\lambda$ on the determinant of the conormal bundle restricted to $Z_{\lambda}$, i.e.
		$$\eta_{\lambda}=\mathrm{weight}_{\lambda} \left(
		(\det N^*_{S_{[\lambda]}|R})|_{Z_{\lambda}} \right) .$$
		Assume $F$ is a $G$-linearized vector bundle on $R$. Assume also that $F$ can descend to the GIT quotient space $R//G$ and let $F$ also denote the descended vector bundle. If for all $\lambda$ as above, the weights of the action of $\lambda$ on $F|_{Z_{\lambda}}$ are strictly less than $\eta_{\lambda}$, then we have the isomorphism $\mathrm{H}^k(R,F)^G\cong \mathrm{H}^k(R//G,F), \forall k\in \mathds{N}$.
		In particular, all the higher cohomology of the descended bundle vanish because $R$ is an affine space that ensures the vanishing of the left-hand side whenever $k>0$.\label{teleman}
	\end{thm}
	
	We will apply this Theorem to the setting of quiver moduli space with $G=GL(\underline{d})$ and $R=R_{\underline{d}}$.
	In the above discussion, we only care about semistability, so the 1-dimensional subtorus of $GL(\underline{d})$ which acts trivially on $R_{\underline{d}}$ does not matter.
	For the norm on the set of 1-ps of $GL(\underline{d})$, we just pick the standard maximal torus of $GL(\underline{d})$ and choose the standard Weyl group invariant inner product on the lattice of 1-ps of the maximal torus. 
	
	It was proven in \cite{Hoskins} that in this setting, the Hesselink stratification coincides with the Harder-Narasimhan stratification and the identification is given as follows: Denoting $GL(\underline{d})=\prod_{i\in Q_0}GL(\mathds{C}^{d_i})$, for a Harder-Narasimhan type $\tau=(\underline{d}^1,\underline{d}^2,\dots,\underline{d}^l)$, we consider the following (virtual) one parameter subgroup $(\lambda_i)_{i\in Q_0}\subset GL(\underline{d})$:
	$$\lambda_i(z)=\mathrm{diag}(z^{\mu_1},\dots,z^{\mu_1},z^{\mu_2},\dots,z^{\mu_2},\dots,z^{\mu_l},\dots,z^{\mu_l}), \ \forall z\in \mathds{C}^*,$$
	where $\mu_s=\mu(\underline{d}^s)$, for $s=1,2,\dots,l$, and the term $z^{\mu_s}$ appears in the above expression $d_i^s$ times.
	Notice that the above expression is only a virtual one parameter subgroup because the exponents $\mu_i$ can be rational numbers, but one can easily normalize it to be a primitive 1-ps by taking a positive multiple of this virtual 1-ps.
	
	Under the identification $\tau\mapsto \lambda\mapsto [\lambda]$ of the index set of the Harder-Narasimhan stratification with the index set of the Hesselink stratification given above, we have that $Z_{\tau}=Z_{\lambda}$ are connected, that $\Sigma_{\tau}=\Sigma_{\lambda}$, and thus $S_{\tau}=GL(\underline{d}).\Sigma_{\tau}=GL(\underline{d}).\Sigma_{\lambda}=S_{[\lambda]}$. This allows us to calculate $\eta_{\lambda}$ and check the assumptions in Teleman Quantization in the example that we are interested in.

	\section{The moduli space \texorpdfstring{$Y$}{Y}}
	\label{section: describing Y}

	\subsection{Geometric Constructions}
	\label{subsection: geometric constructions}
	
	Let $W$ be a 3-dimensional vector space. We view an element $r$
	in $\mathds{C}^2\otimes\mathds{C}^3\otimes W$ as a $2\times 3$ matrix with entries in $W$. We let $Y$ be the GIT quotient of $\mathds{P}(\mathds{C}^2\otimes\mathds{C}^3\otimes W)$ by $GL(2) \times GL(3) / \mathds G_m$, where the action of $(A,B)\in GL(2) \times GL(3) / \mathds G_m$ sends $r\in \mathds{C}^2\otimes\mathds{C}^3\otimes W$ to $ArB^{-1}$ and here $\mathds{G}_m$ appears as the subgroup of simultaneous scalar multiplications on $\mathds{C}^2$ and $\mathds{C}^3$.
	Equivalently, $Y$ is the moduli space of stable representations of the 3-Kronecker quiver $\begin{tikzcd}
		\bullet \arrow[r,bend left] \arrow[r,bend right] \arrow[r] & \bullet
	\end{tikzcd}$ of dimension vector $(2,3)$ and with stability parameter $(3,-2)$.
	
	\begin{remark}
		\label{remark: stability condition explained}
		We can observe that a representation $\rho : W^* \otimes \mathds C^2 \to \mathds C^3$ is stable if and only if $\rho$ is surjective, and for every $v \in \mathds C^2 \setminus 0$, the restriction $\rho_{W^* \otimes v}$ has rank at least $2$.
	\end{remark}
	
	Following the methods of \cite{EPS}, one can show that $Y$ is a smooth 6-dimensional Fano variety of Picard number 1 and index 3, and
	the authors of \cite{Fanoquivermoduli} show
	that it is a blow down of the Hilbert scheme of three points on $\mathds{P}^2$.
	Furthermore, $Y$ is embedded into $\operatorname{Gr}(3,S^2W)$ as a closed subscheme by the construction in \cite[\S2]{EPS}, which we now recall.
	
	\begin{construction}
		\label{construction: Y into Gr(3 S^2 W)}
		An element $[r]\in Y$ can be viewed as an orbit of $2\times 3 $ matrices with entries in $W$.
		By \cite[Lem. 1]{EPS}, linear independence of the three maximal minors of such a matrix is equivalent to the GIT-stability, and the vector subspace $H_r\subset S^2W$ generated by the three minors is independent of the choice of bases.
		The map from $Y$ to $\operatorname{Gr}(3,S^2W)$ sends $[r]\in Y$ to $H_r\subset S^2W$.
	\end{construction}
	
	We fix a basis of $W$ to be $W=\langle x,y,z\rangle$ from now on. We can analyze the orbits of the action of $GL(W)$ on $Y$. There are five orbits, and the Hasse diagram is as follows, where we mark orbits by $H_r$ instead of $r$:
	
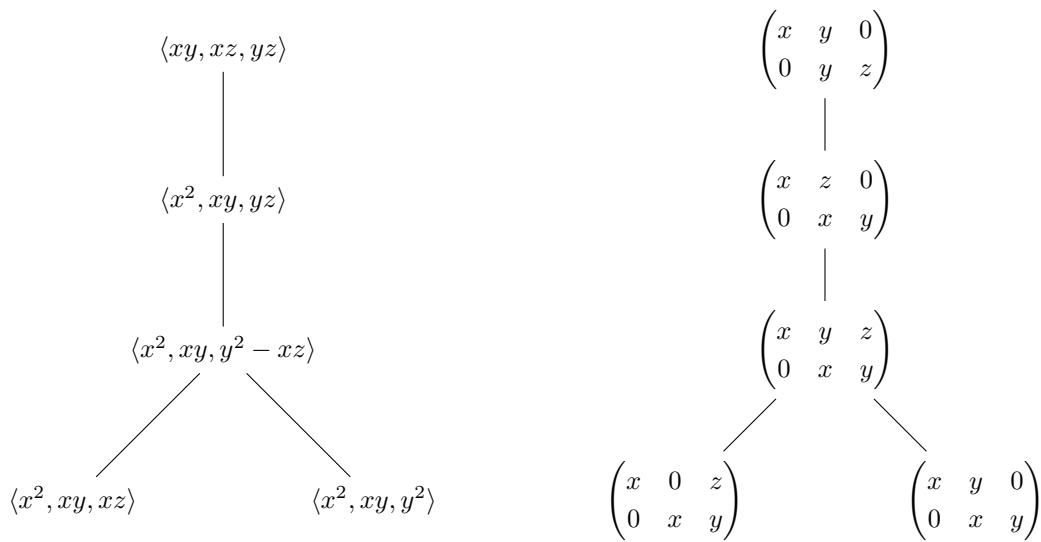
	\begin{tikzpicture}
		\node (max) at (0,14) {$\langle xy,xz,yz\rangle$};
		\node (a) at (0,12) {$\langle x^2,xy,yz\rangle$};
		\node (b) at (0,10) {$\langle x^2,xy,y^2-xz\rangle$};
		\node (c) at (-2,8) {$\langle x^2,xy,xz\rangle$};
		\node (d) at (2,8) {$\langle x^2,xy,y^2\rangle$};
		\draw (max) -- (a)--(b)--(c);
		\draw (b)--(d);
		\node (m) at (8,14) {$\begin{pmatrix}
				x & y& 0\\
				0&y &z
			\end{pmatrix}$};
		\node (a') at (8,12) {$\begin{pmatrix}
				x & z& 0\\
				0&x &y
			\end{pmatrix}$};
		\node (b') at (8,10) {$\begin{pmatrix}
				x & y& z\\
				0&x &y
			\end{pmatrix}$};+
		\node (c') at (6,8) {$\begin{pmatrix}
				x & 0& z\\
				0&x &y
			\end{pmatrix}$};
		\node (d') at (10,8) {$\begin{pmatrix}
				x & y& 0\\
				0&x &y
			\end{pmatrix}$};
		\draw (m) -- (a')--(b')--(c');
		\draw (b')--(d');
	\end{tikzpicture}
	
	Their dimensions are equal to $6,5,4,2,2$, respectively from top to bottom.
	
	The blow up of $Y$ along the orbit of elements of the form $\langle x^2,xy,xz\rangle$ is exactly the Hilbert scheme $\mathrm{Hilb}^3(\mathds{P}^2)$ of three points on $\mathds{P}^2$ and the exceptional divisor in $\mathrm{Hilb}^3(\mathds{P}^2)$ is the locus where the 3 points on $\mathds{P}^2$ are collinear. The Hilbert scheme description is connected with the embedding $Y\subset \mathrm{Gr}(3,S^2W)$ as follows: except the case where $[r]\in Y$ lies in the orbit of elements of the form $\langle x^2,xy,xz\rangle$, $H_r$ always defines a length 3 subscheme of $\mathds{P}(W^*)$ when we view $H_r\subset S^2W$ as a subspace of $\mathrm{H}^0(\mathds{P}(W^*),\mathcal{O}_{\mathds{P}(W^*)}(2))$. For example, when $H_r=\langle xy,yz,xz\rangle$ is viewed as a subspace of quadratic equations, the subscheme it defines on $\mathds{P}(W^*)$ is the union of three reduced points $(x=y=0),(y=z=0),(x=z=0)$. 
	
	We see that $Y$ is the closure of elements in $\mathds{P}(\wedge^3(S^2W))$ of the form $ xy\wedge xz\wedge yz$. Actually $Y$ is isomorphic to the variety of trisecant planes of the Segre embedded $\mathds{P}(W)$ in the Segre embedding $\mathds{P}(W)\hookrightarrow\mathds{P}(S^2W)$ by \cite{IlievManivel}: The variety of trisecant planes of the Segre embedded $\mathds{P}(W)$ is the closure of elements in $\mathrm{Gr}(3,S^2W)\subset \mathds{P}(\wedge^3(S^2W))$ of the form $\langle x^2,y^2,z^2\rangle$ and what we claimed just follows from the fact that there is an automorphism of $\wedge^3(S^2W)$ sending elements of the form $e^2\wedge f^2\wedge g^2$ to $ef\wedge fg\wedge eg$. 
	
	\begin{construction}
		\label{construction: Y into Gr(2 S^21 W)}
		There is a natural map from $Y$ to $\mathrm{Gr}(2,S^{2,1}W)$. Recall that $S^{2,1}W$ can be viewed as the kernel of the map $W\otimes S^2W\rightarrow S^3W$ sending $A \otimes BC$ to $ABC$.
		We define $\Syz(H_r)$ to be the kernel of $H_r\otimes W\rightarrow S^3W$.
		One can show that $\Syz(H_r)$ is two-dimensional except if $[r]$ is in the $SL(W)$-orbit of elements of the form $\langle x^2,xy,xz\rangle$.
		There is a canonical two dimensional subspace of $\Syz(H_r)$: if $r=\begin{pmatrix}
			A & B &C\\
			D & E& F
		\end{pmatrix}$,  where $A,B,C,D,E,F\in W$, then we have two linearly independent tensors $$A\otimes (BF-CE)-B\otimes(AF-CD)+C\otimes(AE-BD),$$
		$$D\otimes (BF-CE)-E\otimes(AF-CD)+F\otimes (AE-BD),$$
		both of which lie in the kernel of $W\otimes S^2W\rightarrow S^3W$, and hence can be viewed as elements of $S^{2,1}W$.
		The two-dimensional subspace $\Syz_r \subset \Syz(H_r)$ generated by these two tensors does not depend on the choice of $r$ in the orbit corresponding to $[r]$. We can thus define a morphism $Y\rightarrow \mathrm{Gr}(2,S^{2,1}W)$ by sending $[r]$ to $\Syz_r$.
	\end{construction}
	
	\begin{notation}
		\label{notation: Y and universal bundles}
		By the stability analysis conducted in \cite[Lem. 1, Lem. 2]{EPS}, $Y$ can be viewed as the moduli space of stable representations of the 3-Kronecker quiver
		$\begin{tikzcd}
			\bullet \arrow[r,bend left] \arrow[r,bend right] \arrow[r] & \bullet
		\end{tikzcd}$ of dimension vector $\underline{d}=(2,3)$, with the stability parameter $\theta=(3,-2)$.
		We fix the choice of a universal representation by letting $\underline{a}=(1,-1)$:
		Let $\tilde{\mathcal U}_1$, resp. $\tilde{\mathcal U}_2$, be the $2$-, resp. $3$-dimensional universal vector bundles on the representation space $R_{(2,3)}$ with the natural action of $GL(2)\times GL(3)$.
		Let $L(\underline{a})$ be the trivial line bundle on $R_{(2,3)}$ on which $GL(2)\times GL(3)$ acts with character $\chi_{\underline{a}}$.
		Then the 1-ps of scalars inside $GL(2)\times GL(3)$ acts trivially on $\tilde{\mathcal U}_1 \otimes L(\underline{a})$, resp. $\tilde{\mathcal U}_2 \otimes L(\underline{a})$,
		and therefore these vector bundles descend to vector bundles on $Y$, which we denote by $\mathcal U_1$, resp. $\mathcal U_2$.
		The structure of a quiver bundle $\tilde{\mathcal U}_1 \to \tilde{\mathcal U}_2 \otimes W$, with $W = \mathds C^3$, also descends to give a quiver bundle structure ${\mathcal U}_1 \to {\mathcal U}_2 \otimes W$.
		Moreover, based on the choice of $\underline{a}$, we have $c_1(\mathcal{U}_1)=c_1(\mathcal{U}_2)$, i.e. $\det(\mathcal{U}_1)\cong \det(\mathcal{U}_2)$, and this line bundle is a generator of the Picard group.
		We denote it by $\mathcal O_Y(-1)$.
	\end{notation}
	
	It turns out that the universal bundles $\mathcal U_1$ and $\mathcal U_2$ can be obtained as pullbacks of universal subbundles along embeddings into grassmannians.
	
	\begin{proposition}
		\label{proposition: universal representation as pullbacks from grassmannians}
		In the setting of \cref{notation: Y and universal bundles}, \cref{construction: Y into Gr(3 S^2 W)} and \cref{construction: Y into Gr(2 S^21 W)}, we have that $\mathcal U_2$ is isomorphic to the pullback of the universal subbundle along
		$Y \to \operatorname{Gr}(3, S^2 W)$, and $\mathcal U_1$ is isomorphic to the pullback of the universal subbundle along
		$Y \to \operatorname{Gr}(2, S^{2,1} W)$. The universal representation map $\mathcal{U}_1\rightarrow\mathcal{U}_2\otimes W$ is induced by the fiberwise inclusion $S^{2,1}W\hookrightarrow S^2W\otimes W$.
	\end{proposition}
	
	Moreover, we recall from \cite{belmans2023chow} another description of $Y$: $Y$ is the zero locus of a general section of $Q^*(1)$ on $\mathrm{Gr}(2,S^{2,1}W)$, where $Q$ is the universal quotient bundle on $\mathrm{Gr}(2,S^{2,1}W)$. 
	
	\begin{remark}
		\label{remark: S21 as slW}
		We view $S^{2,1}W$ as an irreducible $SL(W)$ representation and it is actually isomorphic to $\mathfrak{sl}(W)$ with $SL(W)$ conjugation action. We let $E_{ij}\in M_3(\mathds{C})$ denote the matrix whose $(i,j)$-entry is 1 with all the other entries being zero. The explicit isomorphism $S^{2,1}W\cong \mathfrak{sl}_3$ goes as follows:
		\begin{alignat*}{3}
			E_{13}\leftrightarrow x^2\otimes y-xy\otimes x \quad\quad&& E_{12}\leftrightarrow -x^2\otimes z+xz\otimes x\quad\quad&& E_{23}\leftrightarrow -y^2\otimes x+xy\otimes y\\[1em]
			E_{31}\leftrightarrow -z^2\otimes y+zy\otimes z\quad\quad&&E_{21}\leftrightarrow y^2\otimes z-zy\otimes y\quad\quad&& E_{32}\leftrightarrow z^2\otimes x-xz\otimes z
		\end{alignat*}	
		$$E_{22}-E_{11}\leftrightarrow yz\otimes x+xz\otimes y-2xy\otimes z\quad\quad E_{33}-E_{22}\leftrightarrow xz\otimes y+xy\otimes z-2yz\otimes x$$
		From now on, we will identify $S^{2,1}W$ with $\mathfrak{sl}(W)$.
	\end{remark}
	
	From the tautological exact sequence $0\rightarrow Q^*\rightarrow S^{2,1}W^*\rightarrow S^*\rightarrow 0$, we have a filtration for $\wedge^3(S^{2,1}W^*)$: $0\subset F_1=\wedge^3Q^*\subset F_2\subset F_3=\wedge^3(S^{2,1}W^*)$ such that $F_2/F_1\cong \wedge^2 Q^*\otimes S^*$ and $\wedge^3(S^{2,1}W^*)/F_2\cong Q^*\otimes\wedge^2S^*\cong Q^*(1)$.
	By the Borel-Bott-Weil Theorem, we know that $H^0(\mathrm{Gr}(2,S^{2,1}W), Q^*(1))\cong \wedge^3(S^{2,1}W^*)$.
	Let us consider the following alternating 3-form on $\mathfrak{sl}_3$: $(A,B,C)\in \mathfrak{sl}_3^{\oplus 3}\rightarrow \mathrm{Tr}(A[B,C])$.
	The stabilizer group in $GL(S^{2,1}W)$ of this 3-form is 8-dimensional, so the orbit of this 3-form is the open orbit in $\wedge^3(\mathfrak{sl}_3)^*$. Then the image of this 3-form under the vector bundle map $\wedge^3(\mathfrak{sl}_3)^*\rightarrow Q^*(1)=\wedge^3(\mathfrak{sl}_3)^*/F_2$ is a general section of $Q^*(1)$ and it vanishes at a point $\langle A,B\rangle\in \mathrm{Gr}(2, \mathfrak{sl}_3)$ iff $\mathrm{Tr}(C[A,B])=0$ for any $C\in \mathfrak{sl}_3$, that is iff $AB=BA$. In view of this, the zero locus of this global section can be viewed as the variety of abelian planes contained in $\mathfrak{sl}_3$. 
	
	One can show that this zero locus is indeed $Y$ embedded in $\mathrm{Gr}(2, S^{2,1}W)$ in the way mentioned above: We consider the representative $\langle xy,xz,yz\rangle\in Y\subset \mathrm{Gr}(3,S^2W)$ of the open orbit. It is mapped to $\langle E_{22}-E_{11},E_{33}-E_{22}\rangle\in \mathrm{Gr}(2,\mathfrak{sl}(W))$ via the embedding $Y\hookrightarrow \mathrm{Gr}(2,\mathfrak{sl}(W))$, which is an abelian plane contained in $\mathfrak{sl}_3$ and we can thus conclude that Y is contained in the zero locus of the global section. Since the zero locus of such a general global section has the expected dimension 6, some calculations of the classical invariants as in \cite{belmans2023chow} will show that $Y$ is the only irreducible component of the zero locus.
	
	Moreover, one can show that after a suitable $GL(W)$-equivariant identification of $\mathds{P}(\wedge^3(S^2W))$ with the corresponding $GL(W)$-invariant linear subspace of $\mathds{P}(\wedge^2(S^{2,1}W))$, the variety $Y$ can be exhibited as the scheme-theoretical intersection of $\mathrm{Gr}(3,S^2W)$ with $\mathrm{Gr}(2,S^{2,1}W)$ in the ambient projective space $\mathds{P}(\wedge^2(S^{2,1}(W)))$.
	
	Finally, we mention that $Y$ is isomorphic to the height-zero moduli space $M_{\mathds{P}^2}(4,-1,3)$ of stable sheaves on $\mathds{P}^2$ with $(r,c_1,c_2)=(4,-1,3)$ by \cite{drezet}.
	

	\subsection{Setup of Teleman Quantization for \texorpdfstring{$Y$}{Y}}

	Using \cref{proposition: universal representation as pullbacks from grassmannians}, we identify the universal bundles $\mathcal U_2$ and $\mathcal U_1$ with pullbacks of the universal subbundles on grassmannians.
	In order to calculate the cohomology of vector bundles over $Y$, one of the powerful tools is Teleman Quantization which we recalled in \cref{teleman}.
	Here are some basic information to check the assumptions in Teleman Quantization and one can check \cite{belmans2023rigidity} for reference.
	Recall that $\tilde{\mathcal U}_1$ and $\tilde{\mathcal U}_2$ denote the universal equivariant bundles on the representation space $R_{(2,3)}$.
	\cref{table: weights of universal representations} summarizes the $\lambda$-weights of the universal bundles when restricted to the strata $Z_\lambda$, where $\lambda$ is a 1-ps and $Z_\lambda$ is as in \cref{definition: Z_lambda}.
	
	\begin{table}[ht]
		\begin{center}
			\begin{tabular}{|c|c|c|c|c|c|}
				\hline
				\raisebox{-1ex}{\rule{0pt}{4ex}}
				$(\underline{d_1},\underline{d_2},\dots,\underline{d_l})$&1-ps $\lambda$ & \ Weights of $\tilde{\mathcal{U}}_1$ \ & \ Weights of $\tilde{\mathcal{U}}_2$ \ & \ Weight of $\det \tilde{\mathcal U}_1$ \ &$\eta_{\lambda}$ \\
				\hline
				(1,1),(1,2) &
				\scalebox{.7}{\raisebox{-4.5ex}{\rule{0pt}{10.5ex}}
					$\begin{pmatrix}
						t^3 & 0\\
						0& t^{-2}
					\end{pmatrix},\begin{pmatrix}
						t^3 & 0&0\\
						0& t^{-2}&0\\
						0&0&t^{-2}
					\end{pmatrix}$}
				&(5,0)&(5,0,0)&5&15 \\ 
				\hline (2,2),(0,1) &
				\scalebox{.7}{\raisebox{-4.5ex}{\rule{0pt}{10.5ex}}$\begin{pmatrix}
						t & 0\\
						0& t
					\end{pmatrix},\begin{pmatrix}
						t & 0&0\\
						0& t&0\\
						0&0&t^{-4}
					\end{pmatrix}$}
				&(5,5)&(5,5,0)&10&20 \\ 
				\hline
				(2,1),(0,2) &
				\scalebox{.7}{\raisebox{-4.5ex}{\rule{0pt}{10.5ex}}$\begin{pmatrix}
						t^4 & 0\\
						0& t^4
					\end{pmatrix},\begin{pmatrix}
						t^4 & 0&0\\
						0& t^{-6}&0\\
						0&0&t^{-6}
					\end{pmatrix}$}
				&(20,20)&(20,10,10)&40&100 \\ 
				\hline
				(1,0),(1,3) &
				\scalebox{.7}{\raisebox{-4.5ex}{\rule{0pt}{10.5ex}}$\begin{pmatrix}
						t^{12} & 0\\
						0& t^{-3}
					\end{pmatrix},\begin{pmatrix}
						t^{-3} & 0&0\\
						0& t^{-3}&0\\
						0&0&t^{-3}
					\end{pmatrix}$}
				&(30,15)&(15,15,15)&45& \ 120 \ \\
				\hline
				(1,0),(1,2),(0,1) \ &
				\scalebox{.7}{\raisebox{-4.5ex}{\rule{0pt}{10.5ex}}$\begin{pmatrix}
						t^9 & 0\\
						0& t^{-1}
					\end{pmatrix},\begin{pmatrix}
						t^{-1} & 0&0\\
						0& t^{-1}&0\\
						0&0&t^{-6}
					\end{pmatrix}$}
				&(25,15)&(15,15,10)&40&100 \\ 
				\hline
				(1,0),(1,1),(0,2) &
				\scalebox{.7}{\raisebox{-4.5ex}{\rule{0pt}{10.5ex}}$\begin{pmatrix}
						t^6 & 0\\
						0& t
					\end{pmatrix},\begin{pmatrix}
						t & 0&0\\
						0& t^{-4}&0\\
						0&0&t^{-4}
					\end{pmatrix}$}
				&(20,15)&(15,10,10)&35&90 \\ 
				\hline
				(2,0),(0,3) &
				\scalebox{.7}{\raisebox{-4.5ex}{\rule{0pt}{10.5ex}}$\begin{pmatrix}
						t^3 & 0\\
						0& t^3
					\end{pmatrix},\begin{pmatrix}
						t^{-2} & 0&0\\
						0& t^{-2}&0\\
						0&0&t^{-2}
					\end{pmatrix}$}
				&(15,15)&(10,10,10)&30& 90 \\ \hline
			\end{tabular}
			\caption{Weights of universal representations on strata}
			\label{table: weights of universal representations}
		\end{center}
	\end{table}
	
	
	\begin{remark}
		\label{remark: higher cohg vanish}
		As an example of the application of Teleman Quantization, we notice that all the weights of $\mathcal{U}_1$, $\mathcal{U}_2$, $\mathcal{U}_1\otimes \mathcal{U}_2$, $\mathcal{O}_Y(-1)$, $\mathfrak{sl}(\mathcal{U}_1)$, $\mathcal{U}_1^*\otimes \mathcal{U}_1^*$ under the action of 1-ps $\lambda$ are strictly less than $\eta_{\lambda}$ for all $\lambda$ corresponding to the possible Harder-Narasimhan types. As a result of \cref{teleman}, we know that all the higher cohomology of the mentioned bundles over $Y$ vanish.
	\end{remark}

	\subsection{Chow Ring of the Quiver Moduli Space \texorpdfstring{$Y$}{Y}}

	In this section, we study the Chow ring of $Y$
	using the result about the tautological presentation of Chow rings of fine quiver moduli spaces.
	We use conventions in \cref{notation: Y and universal bundles} and recall \cref{proposition: universal representation as pullbacks from grassmannians}.
	With these identifications, we can see that the map $\mathcal{U}_1\rightarrow \mathcal{U}_2\otimes W$ over $[r]\in Y$ is the inclusion $\Syz_r\subset S^{2,1}W\rightarrow S^2W\otimes W$.
	
	
	By the result of \cite{Kingwalter}, we know that the Chow ring $A^*(Y)$ is generated by the Chern classes of
	$\mathcal{U}_2,\mathcal{U}_1$ as an $\mathds{Z}$-algebra, and is a free abelian group of finite rank.
	By \cite{Franzen}, if we work over $\mathds{Q}$, all the relations between these generators are just $c_1(\mathcal{U}_1)=c_1(\mathcal{U}_2)$ together with the relations coming from the so-called $\theta$-forbidden dimension vectors.
	We denote $c_i(\mathcal{U}_2^*)$ by $c_i$, and we denote $c_i(\mathcal{U}_1^*)$ by $d_i$.
	By \cite[Example 4.7]{Franzen}, we get a clear description of the Chow ring of $Y$, summarized in \cref{table: Chow ring of Y}.
	
	\begin{table}[ht]
		\begin{center}
			\begin{tabular}{|c|c|l|}
				\hline
				\ \ \raisebox{-1.1ex}{\rule{0pt}{4ex}}Group \ \    & \ \ $\mathbb Z$-basis \ \                  & \multicolumn{1}{>{\centering\arraybackslash}m{9cm}|}{Relations} \\ \hline
				\raisebox{-1.1ex}{\rule{0pt}{4ex}}$A^1(Y)$ & $c_1$                    & $\begin{array}{l}c_1 = d_1\end{array}$ \\ \hline
				\raisebox{-1.1ex}{\rule{0pt}{4ex}}$A^2(Y)$ & $c_1^2,\ c_2,\ d_2$      & \\ \hline
				\raisebox{-1.1ex}{\rule{0pt}{4ex}}$A^3(Y)$ & \ $c_1c_2,\ c_1d_2,\ c_3$ \ & $\begin{array}{l}c_1^3=4c_1d_2-3c_3\end{array}$ \\ \hline
				$A^4(Y)$ & $c_2^2,\ c_2d_2,\ d_2^2$ &\raisebox{-2.8ex}{\rule{0pt}{7ex}}$\begin{array}{ll}
					c_1^4=-3c_2^2+9c_2d_2+3d_2^2, & c_1^2c_2=3d_2^2+c_2d_2, \\
					c_1^2d_2=3d_2^2, & c_1c_3=c_2^2-3c_2d_2+3d_2^2\end{array}$ \\ \hline
				$A^5(Y)$ & $\frac{c_2c_3}{3}$       &\raisebox{-4.3ex}{\rule{0pt}{10ex}}$\begin{array}{lll}
					3c_3d_2-c_1^2c_3=\frac{c_2c_3}{3}, & c_3d_2=\frac{2}{3}c_2c_3, & c_1^2c_3=\frac{5}{3}c_2c_3, \\
					c_1^5=19c_2c_3, & c_1^3c_2=9c_2c_3, & c_1^3d_2=6c_2c_3, \\
					c_1c_2^2=\frac{14}{3}c_2c_3, & c_1d_2^2=2c_2c_3, & c_1c_2d_2=3c_2c_3 \end{array}$                       \\ \hline
				$A^6(Y)$ & $c_3^2$                  &\raisebox{-5.8ex}{\rule{0pt}{13ex}}$\begin{array}{llll}
					c_1^6=57c_3^2,& c_1^4c_2=27c_3^2,  &
					c_1^4d_2=18c_3^2,& c_1^3c_3=5c_3^2, \\
					c_1^2c_2^2=14c_3^2,& c_1^2d_2^2=6c_3^2, &
					c_1^2c_2d_2=9c_3^2,& c_1c_3d_2=2c_3^2, \\
					c_1c_2c_3=3c_3^2,& c_2^3=9c_3^2, &
					c_2^2d_2=5c_3^2,& c_2d_2^2=3c_3^2, \\
					d_2^3=2c_3^2
				\end{array}$\\ \hline
			\end{tabular}
			\caption{Chow ring of $Y$, where $c_i = c_i(\mathcal{U}_2^*)$ and $d_i = c_i(\mathcal{U}_1^*)$}
			\label{table: Chow ring of Y}
		\end{center}
	\end{table}
	
	Using the formula given in the article \cite{belmans2023chow}, we know that the point class is equal to $c_3^2$, and combined with \cref{table: Chow ring of Y}, this gives us the following result.
	\begin{proposition}
		\label{proposition: top intersections}
		We have the following list of top intersection numbers:
		$c_1^6=57, \ c_1^4c_2=27, c_1^4d_2=18, \ c_1^3c_3=5, \ c_1^2c_2^2=14, \ c_1^2d_2^2=6, c_1^2c_2d_2=9, \ c_1c_3d_2=2, \ c_1c_2c_3=3, \ c_2^3=9, \ c_2^2d_2=5, \ c_2d_2^2=3, \ c_3^2 = 1, \ d_2^3=2$.\label{intersection}
	\end{proposition}
	
	Now we determine the fundamental classes of the two smallest orbits.
	
	The degenerate locus of the morphism of vector bundles $\mathcal{E}\otimes \mathcal{F}\rightarrow (\mathcal{F}\otimes W)\otimes \mathcal{F}\rightarrow S^2\mathcal{F}\otimes W$ is set-theoretically the orbit of elements of the form $\langle x^2,xy,xz\rangle$. Since the degenerate locus has the correct codimension, we know the class of the degenerate locus is $-3c_2d_2+6d_2^2$. One can calculate that $(c_1^2. -3c_2d_2+6d_2^2)=9$. On the other hand, a direct calculation shows that the degree of the orbit of elements of the form $\langle x^2,xy,xz\rangle$ is also equal to 9. We conclude that this orbit is the scheme-theoretical degenerate locus of $\mathcal{E}\otimes \mathcal{F}\rightarrow (\mathcal{F}\otimes W)\otimes \mathcal{F}\rightarrow S^2\mathcal{F}\otimes W$ and its fundamental class is $-3c_2d_2+6d_2^2$. 
	
	Using the same argument, we know that the orbit of elements of the form $\langle x^2,xy,y^2\rangle$ is the scheme-theoretic degenerate locus of $\mathcal{E}\otimes \mathcal{F}^*\rightarrow W$ and its fundamental class is $3c_2d_2-3d_2^2$.
	
	Moreover, we also have some information about the Chern characters of some vector bundles and the Chern class of the tangent bundle of $Y$ in order to apply the Hirzebruch-Riemann-Roch Theorem.
	These results are summarized in \cref{table: Chern characters of bundles} and \cref{chern ch}.
	
	\begin{table}[ht]
		\tabcolsep 5pt
		\renewcommand{\arraystretch}{1.4}
		\arrayrulewidth 1pt
		\begin{tabular}{|c|c|c|c|c|c|c|c|c|c|c|c|c|c|}
			\hline \diagbox{Sheaf}{Ch char}& $[Y]$ &$c_1$&$c_1^2$&$c_2$&$d_2$&$c_1c_2$&$c_1d_2$&$c_3$&$c_2^2$&$c_2d_2$&$d_2^2$&$c_2c_3$&$c_3^2$\\
			\hline $\mathcal{U}_2$&3&$-1$&$\frac{1}{2}$&$-1$&0&$\frac{1}{2}$&-$\frac{2}{3}$&0&$\frac{1}{8}$&$-\frac{7}{24}$&$\frac{1}{8}$&$-\frac{1}{180}$&0\\
			\hline $\mathfrak{sl}(\mathcal{U}_1^*)$&3&0&1&0&$-4$&$0$&0&0&$-\frac{1}{4}$&$\frac{3}{4}$&$-\frac{5}{12}$&0&$\frac{1}{360}$\\
			\hline$\mathcal{O}$&1&0&0&0&0&0&0&0&0&0&0&0&0\\
			\hline $\mathcal{U}_2^*$&3&1&$\frac{1}{2}$&$-1$&0&$-\frac{1}{2}$&$\frac{2}{3}$&0&$\frac{1}{8}$&$-\frac{7}{24}$&$\frac{1}{8}$&$\frac{1}{180}$&0\\
			\hline$\mathcal{U}_1^*$&2&1&$\frac{1}{2}$&0&$-1$&$0$&$\frac{1}{6}$&$-\frac{1}{2}$&$-\frac{1}{8}$&$\frac{3}{8}$&$-\frac{7}{24}$&$-\frac{1}{120}$&$-\frac{1}{720}$\\
			\hline$\mathcal{U}_2(1)$&3&2&$1$&$-1$&0&$-\frac{1}{2}$&$\frac{4}{3}$&$-\frac{3}{2}$&$-\frac{1}{2}$&$\frac{19}{12}$&$-\frac{5}{4}$&$-\frac{11}{360}$&$-\frac{1}{240}$\\
			\hline$\mathcal{O}(1)$&1&1&$\frac{1}{2}$&0&0&$0$&$\frac{2}{3}$&$-\frac{1}{2}$&$-\frac{1}{8}$&$\frac{3}{8}$&$\frac{1}{8}$&$\frac{19}{120}$&$\frac{19}{240}$\\
			\hline 
			$\mathcal{U}_1^*\otimes\mathcal{U}_2(1)$&6&7&$\frac{11}{2}$&$-2$&$-3$&$-2$&$\frac{55}{6}$&$-\frac{21}{2}$&$-\frac{43}{8}$&$\frac{391}{24}$&$-\frac{83}{8}$&$\frac{89}{360}$&$\frac{53}{240}$ \\
			\hline
		\end{tabular}
		\caption{Chern characters of certain bundles on $Y$}
		\label{table: Chern characters of bundles}
	\end{table}
	
	\begin{proposition}
		\label{chern ch}
		We have the following results for fundamental classes of bundles on $Y$:
		\begin{align*}
			c(\mathrm{T}_Y) &= 1+3c_1+(3c_1^2+5d_2)-(9c_3-16c_1d_2)+(4d_2^2+27c_2d_2-9c_2^2)+17c_2c_3+13c_3^2, \\
			\mathrm{Todd}(Y) &= 1+\frac{3}{2}c_1+
			\left(c_1^2 + \frac{5}{12}d_2\right)+ 
			\left(\frac{3}{8}c_1^3 + \frac{5}{8}c_1d_2\right)+
			\left(- \frac{1}{4}c_2^2 + \frac{3}{4}c_2d_2 + \frac{553}{360}d_2^2\right)+
			\frac{77}{60}c_2c_3+c_3^2, \\
			\mathrm{ch}(\mathcal{U}_1^*\otimes \mathcal{U}_2(1)) &= -\mathrm{ch}( \mathcal{U}_2)+3\mathrm{ch}(\mathfrak{sl}(\mathcal{U}_1^*))+6\mathrm{ch}(\mathcal{O})-6\mathrm{ch}(\mathcal{U}_2^*)-9\mathrm{ch}(\mathcal{U}_1^*)+9\mathrm{ch}(\mathcal{U}_2(1))+3\mathrm{ch}( \mathcal{O}(1)).
		\end{align*}
	\end{proposition}

	\subsection{Interpretation of \texorpdfstring{$X := \mathds{P}_Y(\mathcal{U}_1)$}{X}}
	\label{subsection: interpretation of X as a blowup}

	In this subsection, we will present a description of $\mathds{P}_Y(\mathcal{U}_1)$ as a blow up of a projective space, where the pullbacks of the universal bundles admit a more explicit interpretation.
	
	\begin{notation}
		We use the following convention for projective bundles: given a vector bundle $E$, we write $\mathds{P}(E) := \mathrm{Proj}(\mathrm{Sym}^.(E^*))$.
	\end{notation}
	
	In the framing construction given in \cite{belmans2023vector}, we know that $\mathds{P}_Y(\mathcal{U}_1)$ is also a quiver moduli space for the quiver
	$\begin{tikzcd}
		0\arrow[r] & 1 \arrow[r,bend left] \arrow[r,bend right] \arrow[r] & 2
	\end{tikzcd}$
	of dimension vector $(1,2,3)$  corresponding to the vertices $0,1,2$ marked in the graph with respect to a suitable stability parameter. Obviously, this new quiver contains the 3-Kronecker quiver as a subquiver.
	A natural universal representation for this quiver moduli space is $\mathcal{O}_{\mathds{P}_Y(\mathcal{U}_1)}(-1)\rightarrow \pi^*\mathcal{U}_1\rightarrow \pi^*\mathcal{U}_2\otimes W$ (see the next paragraph), where $\pi: \mathds{P}_Y(\mathcal{U}_1)\rightarrow Y$ is the standard projection. However, in order to make the identifications in the sequel more explicit, we will work with the following universal representation over $\mathds{P}_Y(\mathcal{U}_1)$ by tensoring the above universal representation with $\mathcal{O}_{\mathds{P}_Y(\mathcal{U}_1)}(1)$:  $\mathcal{O}_{\mathds{P}_Y(\mathcal{U}_1)}\rightarrow \mathcal{E}_1 \rightarrow \mathcal{E}_2\otimes W$, where $\mathcal{E}_1=\pi^*\mathcal{U}_1\otimes \mathcal{O}_{\mathds{P}_Y(\mathcal{U}_1)}(1)$ and $\mathcal{E}_2=\pi^*\mathcal{U}_2\otimes \mathcal{O}_{\mathds{P}_Y(\mathcal{U}_1)}(1)$.
	
	In the framing construction, a representation $\mathds{C}\rightarrow \mathds{C}^{\oplus 2}\rightarrow \mathds{C}^{\oplus 3}\otimes W$ for the quiver $\begin{tikzcd}
		0\arrow[r] & 1 \arrow[r,bend left] \arrow[r,bend right] \arrow[r] & 2
	\end{tikzcd}$ is a stable representation if and only if $\mathds{C}\rightarrow \mathds{C}^{\oplus 2}$ is injective and $\mathds{C}^{\oplus 2}\rightarrow \mathds{C}^{\oplus 3}\otimes W$ is a stable representation for the 3-Kronecker quiver.
	This explains why $\mathcal{O}_{\mathds{P}_Y(\mathcal{U}_1)}$ is actually a subbundle of $\mathcal{E}_1$. Moreover, $\mathds{C}^{\oplus 2}\rightarrow \mathds{C}^{\oplus 3}\otimes W$ is a stable representation for the 3-Kronecker quiver iff the adjunction map $\mathds{C}^2\otimes W^*\rightarrow \mathds{C}^3$ is surjective and for any nonzero $v\in \mathds{C}^2$, the image of $\mathds{C}.v\otimes W^*$ via the adjunction map is at least two-dimensional.
	
	In many situations, we have an explicit description of a general point of a quiver moduli space. For the quiver moduli space $\mathds{P}_Y(\mathcal{U}_1)$, we can identify its general points with the general points of $\mathds{P}(W\otimes W^*/\mathds{C}.\mathrm{Id}_{W})$, and actually identify the whole $\mathds{P}_Y(\mathcal{U}_1)$ with the blow up of $\mathds{P}(W\otimes W^*/\mathds{C}.\mathrm{Id}_{W})$ along the image of the composition map $\mathds{P}(W)\times \mathds{P}(W^*)\rightarrow \mathds{P}(W\otimes W^*)\dashrightarrow \mathds{P}(W\otimes W^*/\mathds{C}.\mathrm{Id}_{W})$.
	
	\begin{proposition}
		$\mathds{P}_Y(\mathcal{U}_1)$ is isomorphic to the blow up of $\mathds{P}(W\otimes W^*/\mathds{C}.\mathrm{Id}_{W}) $ along $\mathds{P}(W)\times \mathds{P}(W^*)$.
	\end{proposition}
	
	\begin{proof}
		The idea of the identification is just as finding the Jordan normal form for representations of a vertex equipped with one arrow pointing to itself, we can also find a normal form for most representations of $\begin{tikzcd}
			0\arrow[r] & 1 \arrow[r,bend left] \arrow[r,bend right] \arrow[r] & 2
		\end{tikzcd}$ by choosing good basis elements for $\mathds{C},\mathds{C}^2,\mathds{C}^3$. The choice of a basis element of $\mathds{C}$ is almost trivial (only up to scalar) and we have assumed as above  that we can choose such a basis element globally over the moduli space of representations by tensoring the original universal representation $\mathcal{O}_{\mathds{P}_Y(\mathcal{U}_1)}(-1)\rightarrow \pi^*\mathcal{U}_1\rightarrow \pi^*\mathcal{U}_2\otimes W$ with $\mathcal{O}_{\mathds{P}_Y(\mathcal{U}_1)}(1)$ to get $\mathcal{O}_{\mathds{P}_Y(\mathcal{U}_1)}\rightarrow \mathcal{E}_1 \rightarrow \mathcal{E}_2\otimes W$, which justified our choice above. Consequently, we also have a basis element for $\mathds{C}^2$ via the injective map $\mathds{C}\hookrightarrow\mathds{C}^2$.
		\begin{case}
			For a general point of $\mathds{P}_Y(\mathcal{U}_1)$, the composition of the adjunction morphism $\mathcal{O}_{\mathds{P}_Y(\mathcal{U}_1)}\otimes W^*\hookrightarrow \mathcal{E}_1 \otimes W^*\rightarrow \mathcal{E}_2$ is of full rank in the fiber. For a general representation $\mathds{C}\otimes W^*\hookrightarrow \mathds{C}^{\oplus 2}\otimes W^*\rightarrow \mathds{C}^{\oplus 3}$ (in the adjunction form) corresponding to a general point of $\mathds{P}_Y(\mathcal{U}_1)$, we can identify $\mathds{C}^{\oplus 3}$ with $W^*=\mathds{C}\otimes W^*$, thus actually fixing a basis of $\mathds{C}^3$ since $W^*$ has three canonical basis elements which are the three arrows in the 3-Kronecker quiver. 
			
			Now we temporarily fix a noncanonical basis element $\xi$ in $\mathds{C}^2$ complementary to the inclusion $\mathds{C}\hookrightarrow\mathds{C}^2$. The only remaining data to determine is the composition map $\mathds{C}\cong\mathds{C}.\xi\subset\mathds{C}^2\rightarrow \mathds{C}^{\oplus 3}\otimes W\cong W^*\otimes W$. This last data is only determined up to scalar multiplications and additions of the composition map $\mathds{C}\hookrightarrow\mathds{C}^2\rightarrow W^*\otimes W$ sending $z\in \mathds{C}$ to $z.\mathrm{Id}_W$, since the choice of $\xi$ is noncanonical. Thus the actual parametrizing data corresponds to a general point in $\mathds{P}(W\otimes W^*/\mathds{C}.\mathrm{Id}_{W})$. Conversely, a morphism $g\in\Hom(\mathds{C},W\otimes W^*)$ gives a stable representation in the obvious reverse way iff it remains of rank at least two after the additions of arbitrary copies of the diagonal matrix, which comes from the stability requirement for the 3-Kronecker quiver representations. As a result, we can identify an open subset of $\mathds{P}_Y(\mathcal{U}_1)$ with $\mathds{P}(W\otimes W^*/\mathds{C}.\mathrm{Id}_{W}) \backslash (\mathds{P}(W)\times \mathds{P}(W^*))$ (Our notation is justified by the fact that the composition $\mathds{P}(W)\times \mathds{P}(W^*)\rightarrow \mathds{P}(W\otimes W^*)\dashrightarrow \mathds{P}(W\otimes W^*/\mathds{C}.\mathrm{Id}_{W})$ is an embedding).
		\end{case}
		
		\begin{case}
			When the composition of the maps $\mathds{C}\otimes W^*\hookrightarrow \mathds{C}^{\oplus 2}\otimes W^*\rightarrow\mathds{C}^{\oplus 3}$ is not surjective for a stable representation, we know the rank will be 2 because of the stability condition. Let $K\subset W^*$ be the one-dimensional kernel of the above map $W^*\cong \mathds{C}\otimes W^*\rightarrow \mathds{C}^{\oplus 3}$, and we can identify the image of $\mathds{C}\otimes W^*\rightarrow \mathds{C}^{\oplus 3}$ with $\mathrm{Im}\cong W^*/K$. The remaining data is determined via temporarily choosing a basis element $a$ of $\mathds{C}^{\oplus 2}$ complementary to $\mathds{C}\hookrightarrow\mathds{C}^{\oplus 2}$, and choosing a basis element $b$ of $\mathds{C}^3$ complementary to $\mathrm{Im}\cong W^*/K\subset \mathds{C}^{3}$ and choosing the morphism $f=f_1\oplus f_2:(\mathds{C}.a)\otimes W^*\rightarrow W^*/K\oplus \mathds{C}.b$, where $f_2$ is the component map $f_2:(\mathds{C}.a)\otimes W^*\rightarrow \mathds{C}.b$. Below is a diagram which indicates how the adjunction map $\mathds{C}\otimes W^*\hookrightarrow \mathds{C}^{2}\otimes W^*\rightarrow \mathds{C}^{3}$ looks like in terms of the basis elements we have chosen. 
			$$
			\xymatrix{
				\mathds{C}\otimes W^*  \ar@{^{(}->}[r]\ar@{=}[d] &\mathds{C}^2\otimes W^*\ar@{=}[d]\ar[r]&\mathds{C}^3\ar@{=}[d]\\
				\mathds{C}\otimes W^* \ar@{^{(}->}[r] &\mathds{C} \otimes W^* \ar@{}[d] |{\bigoplus}\ar[r]&  W^*/K\ar@{}[d] |{\bigoplus}\\
				&\mathds{C}.a\otimes W^*\ar[r]^{f_2}\ar[ur]^{f_1}&\mathds{C}.b
			}
			$$
			Based on the stability condition, $f_2$ is surjective. We consider the 2-dimensional subspace $L=\mathrm{Ker}(f_2)\subset W^*$. By definition of $L$, the image of $(\mathds{C}.a)\otimes L$ by $f=f_1\oplus f_2$ is contained in $W^*/K$. Thus the image of $\mathds{C}^2\otimes L$ is contained in $W^*/K\subset\mathds{C}^3$ via the adjunction representation map
			and consequently $L$ does not depend on the choice of $a$. 
			
			It is obvious that when $L$ is known, the choice of the basis element $b$ and the component map $f_2$ doesn't matter at all because we only care about quiver representations up to isomorphisms. Now we only need to concentrate on $f_1: \mathds{C}.a\otimes W^*\rightarrow W^*/K$. Since the adjunction map $\mathds{C}^2\otimes W^*\rightarrow \mathds{C}^3$ is surjective, there exists $x\in W^*$, such that $f(a\otimes x)=b$ and thus $f_1(a\otimes x)=0$. Of course we know $x\notin L$ and $ \mathds{C}.a\otimes W^*=(\mathds{C}.a\otimes \mathds{C}.x)\bigoplus (\mathds{C}.a\otimes L)$. Consequently we only need to concentrate on $f_1|_{\mathds{C}.a\otimes L}=f|_{\mathds{C}.a\otimes L}$. 
			
			Now the reader should be able to convince himself that the only information that need to be encoded are $K\subset W^*$, $L\subset W^*$ and the map $f_1|_{\mathds{C}.a\otimes L}: \ L\cong \mathds{C}.a\otimes L\rightarrow W^*/K$ (A small reminder is that the choice of $b$ and thus $x$ don't matter at all. Another reminder is that the decomposition $f=f_1\oplus f_2$ and thus $f_1$ depend on the choice of $b$. When you replace $b$ by some $b'$, $f_1$ changes but $f_1|_{\mathds{C}.a\otimes L}=f|_{\mathds{C}.a\otimes L}$ doesn't change). Since $a\in \mathds{C}^2$ isn't chosen canonically, we know the latter map $f_1|_{\mathds{C}.a\otimes L}$ is chosen up to additions of the canonical map $L\subset W^*\rightarrow W^*/K$. When we view $K$ as an element in $\mathds{P}(W^*)$ and view $L$ as an element in $\mathds{P}(W)$, all the above together shows that $\mathds{P}_Y(\mathcal{U}_1)$ is isomorphic to the blow up of $\mathds{P}(W\otimes W^*/\mathds{C}.\mathrm{Id}_{W}) $ at $\mathds{P}(W)\times \mathds{P}(W^*)$. \qedhere
		\end{case} 
	\end{proof}
	
	\begin{notation}Over the blow up of $\mathds{P}(W\otimes W^*/\mathds{C}.\mathrm{Id}_{W})=\mathds{P}^7 $ along $\mathds{P}(W)\times \mathds{P}(W^*)=\mathds{P}^2\times \mathds{P}^2$, we let $\mathcal{O}(\tau)$ denote the pullback of $\mathcal{O}_{\mathds{P}^7}(1)$ from the blow up map. We let $Z$ denote the image of $\mathds{P}(W)\times \mathds{P}(W^*)$ in $\mathds{P}(W\otimes W^*/\mathds{C}.\mathrm{Id}_{W}) $ and let $E$ be the exceptional divisor. We know that $E$ is the projectivized normal bundle over $Z$, and we let $\mathcal{O}_{E/Z}(1)$ denote the tautological line bundle of this projective bundle. For convenience, we denote $\mathds{P}_Y(\mathcal{U}_1)$ by $X$.
	\end{notation}
	
	We let $p$ be the bundle projection $E\rightarrow Z$ and let $j$ be the inclusion $E\hookrightarrow X$.
	
	\[\begin{tikzcd}
		E \arrow{d}{p} \arrow{r}{j}&X=\mathds{P}_Y(\mathcal{U}_1)\arrow{r}{\pi}\arrow{d}&Y\\
		Z=\mathds{P}(W)\times \mathds{P}(W^*)\arrow{r}&\mathds{P}(W\otimes W^*/\mathds{C}.\mathrm{Id}_{W})
	\end{tikzcd}
	\]
	
	\begin{proposition}
		\label{E_1 extension}
		Over $\mathds{P}_Y(\mathcal{U}_1)$, we have the following exact sequence:
		$$0\rightarrow \mathcal{O}\rightarrow \mathcal{E}_1\rightarrow \mathcal{O}(-\tau+E)\rightarrow 0.$$
	\end{proposition}
	
	\begin{proof}
		We have seen that $\mathcal{O}$ is a subbundle of $\mathcal{E}_1$. Since $\mathrm{Pic}(\mathds{P}_Y(\mathcal{U}_1))=\mathds{Z}.\mathcal{O}(\tau)+\mathds{Z}.\mathcal{O}(E)$. We only need to check what $\mathcal{E}_1/\mathcal{O}$ is after restricting it to $X-E$ and to fibers of the projective bundle $E\rightarrow Z$ separately. Based on the identifications given above, we easily see that the restriction of $\mathcal{E}_2$  to $X-E$ is the trivial vector bundle and $(\mathcal{E}_1/\mathcal{O})|_{X-E}=\mathcal{O}_{X-E}(-\tau)$. Similarly, we consider the fiber $E_{(L,K)}$ of projective bundle $E\rightarrow Z$ over the point $(L,K)\in Z$. We know that the restriction of $\mathcal{E}_2$ to $E_{(L,K)}$ has a rank two trivial subbundle $(W^*/K)\otimes \mathcal{O}_{E_{(L,K)}}$, and the restriction of $\mathcal{E}_1/\mathcal{O}$ to $E_{(L,K)}$ is equal to $\mathcal{O}_{E_{(L,K)}}(-1)=\mathcal{O}(E)|_{E_{(L,K)}}$.
	\end{proof}
	
	We know $Z\cong \mathds{P}(W)\times \mathds{P}(W^*)$ and will let $\mathcal{O}_Z(n,m)$ denote $\mathcal{O}_{\mathds{P}(W)}(n)\boxtimes\mathcal{O}_{\mathds{P}(W^*)}(m)$ and sometimes use the same notation for the pullback of $\mathcal{O}_{\mathds{P}(W)}(n)\boxtimes\mathcal{O}_{\mathds{P}(W^*)}(m)$ to the projective bundle $E$ over $Z$. In this notation, we know that $(\mathcal{E}_1/\mathcal{O})|_{E}$ is equal to $\mathcal{O}_Z(-1,-1)\otimes \mathcal{O}_{E/Z}(-1)$.
	
	\begin{proposition}
		We have the following exact sequence over $X$.
		$$0\rightarrow W^*\otimes \mathcal{O}\rightarrow \mathcal{E}_2\rightarrow \mathcal{O}_E(L)\rightarrow 0$$
		where $\mathcal{O}_E(L)=(j_*\mathcal{O}_Z(1,0))\otimes \mathcal{O}(-\tau+E)$. \label{E2}
	\end{proposition}
	
	\begin{proof}
		Obviously we have the generically injective, thus injective, map between two locally free sheaves $W^*\otimes \mathcal{O}$ and $\mathcal{E}_2$, and the quotient $Q$ is scheme-theoretically supported on $E$. In the identification given above, we know that the restriction of $Q$ to $E$ is actually a line bundle corresponding to $\mathds{C}.b$ and is equal to the cokernel of the map $Q_1^*\otimes(\mathcal{E}_1/\mathcal{O})|_{E}\hookrightarrow W^*\otimes(\mathcal{E}_1/\mathcal{O})|_{E}$ corresponding to the fiberwise map $L\otimes \mathds{C}.a\hookrightarrow W^*\otimes\mathds{C}.a$, where $Q_1^*$ is the dual of the pullback of the tautological quotient bundle over $\mathds{P}(W)$. An easy first Chern class calculation shows that $Q|_E$ is equal to $(\mathcal{E}_1/\mathcal{O})|_{E}\otimes \mathcal{O}_Z(1,0)=\mathcal{O}_Z(0,-1)\otimes \mathcal{O}_{E/Z}(-1)$.   
	\end{proof}
	
	\begin{Corollary}
		\label{2dual}
		The following sequence is exact:
		$$0\rightarrow \mathcal{E}_2^*\rightarrow W\otimes\mathcal{O}\rightarrow j_*\mathcal{O}_Z(0,1)\rightarrow 0.$$
	\end{Corollary}
	
	\begin{proof}
		We just need to dualize the last short exact sequence using Grothendieck-Verdier Duality:
		\begin{align*}
			\mathrm{R}\mathcal{H}om^.(j_*\mathcal{O}_Z(1,0)\otimes \mathcal{O}(-\tau+E),\mathcal{O})
			&=\mathrm{R}\mathcal{H}om^.(j_*\mathcal{O}_Z(0,-1),\mathcal{O}(-E))
			\\&=j_*(\mathrm{R}\mathcal{H}om^.(\mathcal{O}_Z(0,-1),j^*\mathcal{O}(-E)\otimes \omega_j[-1])=j_*\mathcal{O}_Z(0,1)[-1].
			\qedhere
		\end{align*}
	\end{proof}
	
	\begin{Corollary}
		\label{corollary: determinants of E_1 and E_2 and also tau=rel.O(1)}
		$\mathrm{det}(\mathcal{E}_1)=\mathcal{O}(-\tau+E)$, $\mathrm{det}(\mathcal{E}_2)=\mathcal{O}(E)$, $\mathcal{O}_X(\tau)=\mathcal{O}_{\mathds{P}_Y(\mathcal{U}_1)}(1)$;
		\label{dual of E1}
		In particular, $\mathcal E_1^* \cong \mathcal E_1(\tau - E)$.
	\end{Corollary}
	\begin{proof}
		For the third statement, recall that $\mathcal{O}_X(\tau)$ is the pullback of $\mathcal{O}_{\mathds{P}^7}(1)$ from the blow up map, $\mathcal{O}_{\mathds{P}_Y(\mathcal{U}_1)}(1)$ is the tautological line bundle of the projective bundle $\mathds{P}_Y(\mathcal{U}_1)$ over $Y$. The third statement comes from the fact that $\mathcal{E}_1=\pi^*\mathcal{U}_1\otimes \mathcal{O}_{\mathds{P}_Y(\mathcal{U}_1)}(1)$ and $\mathcal{E}_2=\pi^*\mathcal{U}_2\otimes \mathcal{O}_{\mathds{P}_Y(\mathcal{U}_1)}(1)$ and $\mathrm{det}(\mathcal{U}_1)=\mathrm{det}(\mathcal{U}_2)$.
	\end{proof}
	\begin{lemma}
		\label{-3} $\mathrm{R}p_*(\mathcal{O}_{E/Z}(-3)) \cong \mathcal{O}_Z(5,5)[-2]$.
	\end{lemma}
	\begin{proof}
		$E$ is the projectivization of the normal bundle $N_{Z|\mathds{P}^7}$. As a result, $\mathrm{R}p_*(\mathcal{O}_{E/Z}(-3))=\mathrm{det}(N_{Z|\mathds{P}^7})[-2]$.
		On the other hand, we have the exact sequence $0\rightarrow \mathcal{O}_{\mathds{P}^7}(1)|_Z=\mathcal{O}_Z(1,1)\rightarrow N_{Z|\mathds{P}^8}\rightarrow N_{Z|\mathds{P}^7}\rightarrow 0$.
		To calculate $N_{Z|\mathds{P}^8}$, observe that we have the following exact sequence: $0\rightarrow T_{\mathds{P}(W)}\oplus T_{\mathds{P}(W^*)} \rightarrow T_{\mathds{P}(W\otimes W^*)}\rightarrow T_{\mathds{P}(W)}\boxtimes T_{\mathds{P}(W^*)}\rightarrow 0$, so that $N_{Z|\mathds{P}^8}\cong T_{\mathds{P}(W)}\boxtimes T_{\mathds{P}(W^*)}$.
	\end{proof}

	\section{Derived Category}
	\label{section: derived category}

	In this section we will study the derived category of the 6-dimensional quiver moduli space $Y$, and give a full exceptional sequence.
	
	\begin{definition}
		Let $\mathcal{T}$ be a triangulated category.
		An object $E\in \mathcal{T}$ is called \emph{exceptional} if $\rhom^\bullet(E,E)$ $=\mathds{C}[0]$.
		A collection of objects $\langle E_1,E_2,\dots,E_r\rangle$ is called \emph{exceptional} if each object $E_i$ is exceptional and $\mathrm{RHom}^{\bullet}(E_i,E_j)=0$ for $i>j$.
		An exceptional collection of objects $\langle E_1,E_2,\dots,E_r\rangle$ is \emph{full} if the smallest triangulated subcategory containing $\langle E_1,E_2,\dots,E_r\rangle$ is $\mathcal{T}$ itself.
	\end{definition}
	
	\begin{definition}
		Let $i_*:\mathcal{A}\hookrightarrow \mathcal{T}$ be an admissible subcategory, let $i^*$ and $i^{!}$ denote the left and right adjoint functors of $i_*$, respectively.
		The \emph{left mutation} of an element $\mathcal{G}\in \mathcal{T}$ across $\mathcal A$ is defined to be the cone of the canonical map $i_*i^!\mathcal{G}\rightarrow\mathcal{G} $, and the \emph{right mutation} across $\mathcal A$ is defined to be the cone of the map $\mathcal{G}[-1]\rightarrow i_*i^*\mathcal{G}[-1]$.
	\end{definition}
	
	When $\mathcal A = \langle E \rangle$ for an exceptional object $E$, we have 
	$i_*i^! (\mathcal G) = \rhom^\bullet ( E , \mathcal G ) \otimes E$ and 
	$i_*i^* (\mathcal G) = \rhom^\bullet ( \mathcal G , E )^* \otimes E$.
	
	\begin{lemma}
		Assume $\langle E_1, E_2\rangle$ is an exceptional collection in a triangulated category $\mathcal{C}$. For any object $\mathcal{G}\in \mathcal{C}$, we have $\mathrm{RHom}^{\bullet}(E_2,\mathcal{G})=\mathrm{RHom}^{\bullet}(E_2,\mathrm{R}_{E_1}\mathcal{G})$, where $\mathrm{R}_{E_1}\mathcal{G}$ is the right mutation of $\mathcal{G}$ across $\langle E_1 \rangle$, and similarly $\mathrm{RHom}^{\bullet}(\mathcal{G},E_1)=\mathrm{RHom}^{\bullet}(\mathrm{L}_{E_2}\mathcal{G}, E_1)$.\label{mutation and Rhom}
	\end{lemma}
	\begin{proof}
		The right mutation is defined via the exact triangle:
		$$\mathrm{R}_{E_1}\mathcal{G}\rightarrow \mathcal{G}\rightarrow (\mathrm{RHom}^{\bullet}(\mathcal{G},E_1))^*\otimes E_1\rightarrow \mathrm{R}_{E_1}\mathcal{G}[1].$$
		We only need to apply the functor $\mathrm{RHom}^{\bullet}(E_2,-)$ and notice that $\mathrm{RHom}^{\bullet}(E_2,E_1)=0$ to conclude. \label{Der lemma}
	\end{proof}
	\begin{lemma}\label{res G}
		Assume $Y$ is a smooth projective variety, and $\mathcal{G}^.\in \mathrm{D}^b(Y)$. If the derived pullback $i_y^*(\mathcal{G}^.)=0$ in $\mathrm{D}^b(k(y))$ for every closed point $y$, where $i_y: \mathrm{Spec}(k(y))\hookrightarrow Y$ is the closed embedding, then $\mathcal{G}^.=0$.
	\end{lemma}
	\begin{proof}
		If $\mathcal{G}^.\neq 0$, we can assume $\mathcal{G}^.$ is a bounded complex of locally free sheaves and $\mathcal{G}^j=0$ for $j>q$ and $\mathcal{H}^q(\mathcal{G}^.)=\mathrm{Coker}(\mathcal{G}^{q-1}\rightarrow\mathcal{G}^q)$ is not zero. We pick $y$ in the support of $\mathcal{H}^q(\mathcal{G}^.)$ and apply the derived pullback $i_y^*$. Since ordinary pullback is a right exact functor, $\mathcal{H}^q(i_y^*(\mathcal{G}^.))$ is equal to the fiber of $\mathcal{H}^q(\mathcal{G}^.)$ at $y$, which is nonzero. This contradicts our assumption.
	\end{proof}
	\begin{remark}
		Assume that we have a covering family of closed subvarieties of $Y$, i.e. we have subvarieties $\{Z_{\alpha}\}_{\alpha}$  such that every closed point is contained in some $Z_{\alpha}$. As a Corollary of the last lemma, if the derived restriction to $Z_{\alpha}$ of $\mathcal{G}^.\in\mathrm{D}^b(Y)$ is zero for every $\alpha$, then $\mathcal{G}^.=0$. \label{covering}
	\end{remark}
	
	Following \cite{belmans2023chow}, we consider a tautological morphism whose degenerate locus is the diagonal in $M \times M$, where $M$ is a fine moduli space of stable quiver representations.
	We refer to \cite[\S14]{Fulton1998} for the definition of degenerate loci, and recall our notation for quiver moduli.
	Let $Q=(Q_0, Q_1)$ be an acyclic quiver, $\underline{d}\in \mathds{N}^{Q_0}$ be a dimension vector, and $\theta\in \mathrm{Hom}(\mathds{Z}^{Q_0},\mathds{Z})$ be a stability parameter which is coprime to $\underline{d}$. We let $M=M^{\theta-st}(Q,\underline{d})$ denote the moduli space of (semi)stable representations of dimension vector $\underline{d}$ and $\mathcal{U}=(\mathcal{U}_i)_{i\in Q_0}$ denote the universal representation, which is unique up to tensoring with a line bundle.
	With this notation, we are ready to cite the following result.
	
	\begin{thm}[{\cite[Prop. 4.1]{belmans2023chow}}]
		There exists a tautological morphism $\sigma: \bigoplus_{i\in Q_0}\mathcal{U}_i^{*}\boxtimes\mathcal{U}_i\rightarrow \bigoplus_{a:i\rightarrow j}\mathcal{U}_{i}^*\boxtimes\mathcal{U}_{j}$ over $M\times M$ such that the scheme-theoretical maximal degenerate locus of $\sigma$ is exactly the diagonal $\Delta_M\subset M\times M$.\label{resolution}
	\end{thm}
	
	\begin{proposition}
		Let $M$ be the fine moduli space of stable representations of an acyclic quiver $Q$ of fixed dimension vector $\underline d$ with respect to a stability parameter $\theta$ coprime to $\underline d$.
		Then $\mathrm D^b (M)$ is generated by tensor products of Schur functors of universal bundles and their duals.\label{general dercat}
	\end{proposition}
	
	\begin{proof}
		Using the well-known Eagon-Northcott complex
		(see for example \cite[Appendix 2]{Harris})
		for the resolution of the structure sheaf of maximal degenerate locus of the bundle morphism $\sigma$ and \cref{resolution}, we obtain a resolution for the structure sheaf of diagonal $\mathcal{O}_{\Delta_{M}}$ in terms of tensor products of Schur functors of universal bundles and their duals. Viewing $\mathcal{O}_{\Delta_{M}}$ as the Fourier-Mukai kernel for the identity morphism from $\mathrm{D}^b(M)$ to itself, we see that the derived category is generated by tensor products of Schur functors of universal bundles and their duals.
	\end{proof}
	
	In view of the above Proposition, we aim to find a full exceptional collection for our six-dimensional fine quiver moduli space $Y$, such that the exceptional objects are tensor products of Schur functors of universal bundles and their duals.
	
	\subsection{Exceptional Collection}
	\label{subsection: exceptional collection}

	\begin{lemma}
		\label{lemma: induced Serre duality}
		Recall that $\pi : X \to Y$ denoted our $\mathds{P}^1$-bundle.
		Given any two complexes $\mathcal F$ and $\mathcal G$ in $\mathrm D^b (Y)$, we have
		$\rhom_X (\pi^* \mathcal F , \pi^* \mathcal G) \cong \rhom_X (\pi^* \mathcal G , \pi^* \mathcal F \otimes \mathcal O (-9\tau+3E) [6])^*$.
	\end{lemma}
	
	\begin{proof}
		This follows from the fact that $\mathcal O (-9\tau+3E) \cong \pi^* \omega_Y$, that $\pi^*$ is fully faithful, and from Serre duality on $Y$.
	\end{proof}
	
	\begin{thm}
		\label{theorem: exceptional collection on Y}
		The following collection is strong exceptional in $\mathrm D^b (Y)$:
		\begin{equation} \left\langle 
			\mathfrak{sl}(\mathcal{U}_1), \
			\mathcal{O}_Y, \mathcal{U}_2^*, \mathcal{U}_1^*, \mathcal{U}_2(1), \
			\mathcal{O}_Y(1), \mathcal{U}_2^*(1), \mathcal{U}_1^*(1), \mathcal{U}_2(2),  \
			\mathcal{O}_Y(2), \mathcal{U}_2^*(2), \mathcal{U}_1^*(2), \mathcal{U}_2(3)\right\rangle.
			\label{1 exc}
		\end{equation}
		Moreover, for any $i>0$ and any two objects $\mathcal F$ and $\mathcal G$ from the collection, we have that $\Ext^{>0}_Y(\mathcal F, \mathcal G(i)) = 0$.
	\end{thm}
	
	\begin{proof}
		One-direction vanishing of higher Ext groups between two objects follows from \cref{teleman} and \cref{table: weights of universal representations}.
		In particular, if the collection is exceptional, then it is strong.
		
		
		Since the pullback functor $\pi^* : \operatorname{D}^b (Y) \to \operatorname{D}^b (X)$ is fully faithful by the virtue of $\pi$ being a projective bundle, it is enough to check that the pullback of the collection \eqref{1 exc} is exceptional in $\operatorname{D}^b(X)$.
		By \cref{corollary: determinants of E_1 and E_2 and also tau=rel.O(1)} and by definition of $\mathcal E_i := \pi^* \mathcal{U}_i \otimes \mathcal{O}(\tau)$, we have the following:
		\begin{align*}
			\pi^* \mathcal O_Y(1) &= \pi^* (\det \mathcal U_i)^* = \mathcal O_X (3 \tau - E),
			\\
			\pi^* \mathcal U_2^* &= \mathcal E_2^* (\tau),
			\\
			\pi^* \mathcal U_1^* &= \pi^* ( \det \mathcal U_1^* \otimes \mathcal U_1 ) = \mathcal E_1 (2\tau-E),
			\\
			\pi^* (\mathcal U_2(1)) &= \mathcal E_2 (2\tau-E).
		\end{align*}
		Furthermore, formation of $\mathfrak{sl}$ is independent of line bundle twists, and together with the above equations, it implies that it is enough to check that the following collection is exceptional in $\operatorname{D}^b(X)$:
		
		
		\[
		\begin{matrix*}[l]
			\big\langle 
			\mathfrak{sl}(\mathcal{E}_1),&
			\mathcal{O}_X,& \mathcal{E}_2^*(\tau),& \mathcal{E}_1 (2\tau - E),& \mathcal{E}_2 (2\tau - E),  \\
			& \mathcal{O}_X(3\tau - E),& \mathcal{E}_2^*(4\tau - E),& \mathcal{E}_1(5\tau - 2E),& \mathcal{E}_2(5\tau - 2E),   \\
			& \mathcal{O}_X(6\tau - 2E),& \mathcal{E}_2^*(7\tau - 2E),& \mathcal{E}_1(8\tau - 3E),& \mathcal{E}_2(8\tau - 3E) \big\rangle.
		\end{matrix*}
		\]

		In order to prove semiorthogonality, we need to show that the cohomology of the following sheaves vanish:
		\begin{align*}
			\begin{array}{lll}
				\mathcal O_X (-3 \tau + E), & \mathcal O_X (-6 \tau + 2E), \\
				\mathcal E_1 (- \tau), & \mathcal E_1 (- 4 \tau + E),  & \cancel{\mathcal E_1 (- 7 \tau + 2E)}, \\
				\mathcal E_1 \otimes \mathcal E_1 (- 2\tau), & \cancel{\mathcal E_1 \otimes \mathcal E_1 (- 5 \tau + E)},  \\
				\mathcal E_2 (-\tau),  & \mathcal E_2 (-4\tau+E),  & \cancel{\mathcal E_2 (-7\tau+2E)},  \\
				\mathcal E_1 \otimes \mathcal E_2 (-2\tau),    & \cancel{\mathcal E_1 \otimes \mathcal E_2 (-5\tau+E)}, \\
				\mathcal E_2 \otimes \mathcal E_2 (-2\tau),    & \cancel{\mathcal E_2 \otimes \mathcal E_2 (-5\tau+E)}, \\
				\mathcal E_2^* (-2\tau + E),   & {\mathcal E_2^* (-5\tau + 2E)}, & \cancel{\mathcal E_2^* (-8\tau + 3E)},      \\
				\mathcal E_1 \otimes \mathcal E_2^*,   & \mathcal E_1 \otimes \mathcal E_2^* (-3\tau + E),    & \cancel{\mathcal E_1 \otimes \mathcal E_2^* (-6\tau + 2E)}, \\
				\mathcal E_2 \otimes \mathcal E_2^* (-3\tau+E),    & \cancel{\mathcal E_2 \otimes \mathcal E_2^* (-6\tau+2E)}, \\
				\mathcal E_2^* \otimes \mathcal E_2^* (-\tau + E),  &
				\mathcal E_2^* \otimes \mathcal E_2^* (-4\tau + 2E),  &
				\cancel{\mathcal E_2^* \otimes \mathcal E_2^* (-7\tau + 3E)},  \\
				\mathfrak{sl} (\mathcal E_1),  & \mathfrak{sl} (\mathcal E_1) (-3\tau + E),  & \mathfrak{sl} (\mathcal E_1) (-6\tau + 2E), \\
				\mathfrak{sl} (\mathcal E_1) \otimes \mathcal E_1 (-\tau) ,    & \mathfrak{sl} (\mathcal E_1) \otimes \mathcal E_1 (-4\tau+E) ,    & \mathfrak{sl} (\mathcal E_1) \otimes \mathcal E_1 (-7\tau+2E) , \\
				\mathfrak{sl} (\mathcal E_1) \otimes \mathcal E_2 (-\tau) ,    & \mathfrak{sl} (\mathcal E_1) \otimes \mathcal E_2 (-4\tau+E) ,    & \mathfrak{sl} (\mathcal E_1) \otimes \mathcal E_2 (-7\tau+2E) , \\
				\mathfrak{sl} (\mathcal E_1) \otimes \mathcal E_2^* (-2\tau+E), & \mathfrak{sl} (\mathcal E_1) \otimes \mathcal E_2^* (-5\tau+2E), & \mathfrak{sl} (\mathcal E_1) \otimes \mathcal E_2^* (-8\tau+3E) .\\
			\end{array}
		\end{align*}
		We strike out the sheaves whose cohomology vanishing would follow from the Serre duality simplification (\cref{lemma: induced Serre duality}) and vanishing of the cohomology of the remaining sheaves; although we include proofs of vanishing of the cohomology of all the above sheaves, the reader may decide to only check the proofs for those sheaves which are not struck out.
		The rest of the subsection will be dedicated to proving the above vanishing, and that each object in the collection is exceptional.
		Note that the objects $\mathcal O_Y$, $\mathcal O_Y(1)$ and $\mathcal O_Y(2)$ are exceptional because $\mathrm H^\bullet (Y , \mathcal O_Y) = \mathrm H^\bullet (X, \mathcal O_X) = \mathds C[0]$.
	\end{proof}

	\begin{lemma}
		\label{lemma: H(O(a)) = H(O(a+E)) = H(O(a+2E))}
		Over X, we have $\mathrm{H}^{\bullet}(\mathcal{O}(a\tau))\cong\mathrm{H}^{\bullet}(\mathcal{O}(a\tau+E))\cong\mathrm{H}^{\bullet}(\mathcal{O}(a\tau+2E))$ for all $a\in \mathds{Z}$.    
	\end{lemma}
	
	\begin{proof}
		We consider the standard ideal sequence $0\rightarrow \mathcal{O}(-E)\rightarrow \mathcal{O}\rightarrow \mathcal{O}_E\rightarrow 0$ and tensor it by $\mathcal{O}(a\tau+bE)$, for $b=1,2$, to get $0\rightarrow \mathcal{O}(a\tau+(b-1)E)\rightarrow \mathcal{O}(a\tau+bE)\rightarrow \mathcal{O}_E(a\tau+bE)\rightarrow 0$.
		In order to prove the claim,
		we only need to show $\mathrm{H}^{\bullet}(X,\mathcal{O}_E(a\tau+bE))=0$.
		For that, we use $\mathrm{R}p_*(\mathcal{O}_{E/Z}(-b))=0$, for $b=1,2$, to obtain the following:
		\begin{align*}
			\mathrm{H}^{\bullet}(X,\mathcal{O}_E(a\tau+bE)) &= \mathrm{H}^{\bullet}(E,\mathcal{O}_E(a\tau+bE))=\mathrm{H}^{\bullet}(E, p^* \mathcal{O}_Z(a,a)\otimes \mathcal{O}_{E/Z}(-b)))
			\\ &= \mathrm{H}^{\bullet}(Z,\mathcal{O}_Z(a,a)\otimes \mathrm{R}p_*\mathcal{O}_{E/Z}(-b))=0.
		\end{align*}
	\end{proof}

	\begin{lemma}
		\label{lemma: vanishO}
		\label{O(*)}
		We have the following vanishing properties for the cohomology of twists of $\mathcal O_X$:
		\begin{align*}
			& \mathrm{H}^{\bullet}(\mathcal{O}(-E))=0;
			\\&\mathrm{H}^{\bullet}(\mathcal{O}(a\tau+bE))=0 \text{ for } -7\leq a\leq -1 \text{ and } b=0,1,2;
			\\& \mathrm H^\bullet (\mathcal O (a \tau + 3E)) 
			= \Big( \mathrm{H}^{\bullet}(\mathcal{O}((-a-8)\tau-E)) [7] \Big)^* = 0
			\text{ for } a = -8, -7, -6.
		\end{align*}
		Moreover, we have $\mathrm H^\bullet (\mathcal O (\tau - E)) \cong \mathds C[-1]$.
	\end{lemma}
	
	\begin{proof}
		The first claim follows from the long exact sequence of cohomology associated to the ideal sequence of $E$.
		
		For the second statement, we only need to show $\mathrm{H}^{\bullet}(X,\mathcal{O}(a\tau))=0$ for $-7\leq a\leq -1$. This latter vanishing follows from the fact that $\mathrm{H}^{\bullet}(\mathds{P}^7,\mathcal{O}_{\mathds{P}^7}(a))=0$ for $-7\leq a\leq -1$ and the fact that the derived pushforward of $\mathcal{O}_X$ via the blow up map $X\mapsto \mathds{P}^7$ is equal to $\mathcal{O}_{\mathds{P}^7}$.
		
		
		For the third part, we first use Serre duality, noting that $\omega_X = \mathcal O (-8 \tau + 2E)$, and then consider the twist of the ideal sequence of $E$ by $\mathcal O (a' \tau)$, where $a' = -a-8 = 0, -1, -2$:
		\[
		0 \to \mathcal O_X (a'\tau - E) \to \mathcal O_X (a' \tau) \to \mathcal O_E (a' \tau) \to 0.
		\]
		If $a'=0$, then $\mathrm H^\bullet (X, \mathcal O_X) \cong \mathrm H^\bullet (X, \mathcal O_E) \cong \mathds C[0]$, and therefore $\mathrm H^\bullet (X, \mathcal O_X (- E)) = 0$.
		If $a' = -1, -2$,
		we have $\mathrm H^\bullet (X, \mathcal O_X (a' \tau) ) = 0$ by the previous part of this Lemma.
		We now calculate $\mathrm H^\bullet (X, \mathcal O_E (a'\tau)) =
		\mathrm H^\bullet (E, p^* \mathcal O_Z (a',a')) =
		\mathrm H^\bullet (Z, (p_* \mathcal O_E) \otimes \mathcal O_Z (a',a')) = 
		\mathrm H^\bullet (Z, \mathcal O_Z (a',a')) = 0$ for $a' = -1, -2$.
		This allows us to finish the proof.
		
		For the last part, we consider again the twisted ideal sequence $0 \to \mathcal O(\tau - E) \to \mathcal O(\tau) \to \mathcal O_E(\tau) \to 0$.
		We notice that $\mathcal O(\tau) \to \mathcal O_E(\tau)$ can be obtained as a pullback from $\mathds P^7$ of $\mathcal O_{\mathds{P}^7}(1) \to \mathcal O_Z(1,1)$,
		therefore the induced map on cohomology $\mathrm H^0(\mathcal \tau)=\mathrm H^0 (\mathds{P}^7,\mathcal O_{\mathds{P}^7}(1)) \cong \mathfrak {sl} (W) \to \mathrm H^0 (\mathcal O_E(\tau)) \cong \mathrm H^0 (\mathds P^7 , \mathcal O_Z(1,1)) \cong W^* \otimes W$ is an injection.
		We conclude that $\mathrm H^1 (\mathcal O (\tau - E)) \cong W^* \otimes W / \mathfrak {sl} (W) \cong \mathds C$, while all other cohomology groups vanish.
	\end{proof}
	
	\begin{lemma}
		\label{lemma: vanishE1}
		We have the following vanishing properties for the cohomology of twists of $\mathcal E_1$:
		\begin{align*}
			& \operatorname{H}^\bullet{(\mathcal E_1(a\tau - E))} = 0 
			\text { for } a = 0,1;
			\\& \operatorname{H}^\bullet{(\mathcal E_1(a \tau + b E))} = 0
			\text{ for }
			-6 \leq a \leq -1 \text{ and } b = 0, 1;	
			\\ & \operatorname{H}^\bullet{(\mathcal E_1(a \tau + 2 E))} = 0
			\text{ for } a = -7,-6,-5
			.
		\end{align*}
		Furthermore, $\mathrm H^\bullet (X, \mathcal E_1) \cong \mathds C[0]$.
	\end{lemma}
	
	\begin{proof}
		Tensoring the short exact sequence in \cref{E_1 extension} by $\mathcal O(a\tau +bE)$, we get the following:
		$$	0 \to \mathcal O(a \tau +bE) \to \mathcal E_1(a\tau +bE) \to \mathcal O ((a-1)\tau +(b+1)E) \to 0
		\mbox{.} $$
		We now apply \cref{lemma: vanishO} to deduce that the cohomology of the right-hand and the left-hand terms vanish when $a = 0$ and $b=-1$, or when $-6 \leq a \leq -1$ and $b=0,1$, or when $a = -7,-6,-5$ and $b=2$. 
		
		When $a=1$ and $b=-1$, the short exact sequence above becomes $0 \to \mathcal O (\tau - E) \to \mathcal E_1 (\tau - E) \to \mathcal O \to 0$.
		By \cref{O(*)}, the associated long exact sequence of cohomology reads:
		\[
		0 \to \mathrm H^0 (\mathcal E_1 (\tau - E)) \to \mathds C \xrightarrow{\delta}
		\mathds C \to \mathrm H^1 (\mathcal E_1 (\tau - E)) \to 0.
		\]
		We now note that the extension in \cref{E_1 extension} is non-split (in fact we know $\mathrm{End}(\mathcal E_1)=\mathds{C}$ from
		\cite{belmans2023vector}), therefore the element representing it in $\mathrm{Ext}^1 (\mathcal O , \mathcal O (\tau - E)) \cong \mathrm H^1 (\mathcal O (\tau - E))$ is non-zero; but this element is precisely $\delta$.
		Now the long exact sequence of cohomology implies that $\mathrm H^\bullet ( \mathcal E_1 (\tau - E)) = 0$.
		
		In case $a=b=0$, the cohomology of the right-hand term vanish, so $\mathrm H^\bullet (X, \mathcal E_1) \cong \mathrm H^\bullet (X, \mathcal O) \cong \mathds C[0]$.
	\end{proof}
	
	\begin{lemma}
		\label{vanishE1E1}
		We have the following vanishing properties for the cohomology of twists of $\mathcal E_1 \otimes \mathcal E_1$:
		\begin{align*}
			&\operatorname {H^\bullet} ({\mathcal E_1 \otimes \mathcal E_1(-E)}) = 0;
			\\&\operatorname {H^\bullet} ({\mathcal E_1 \otimes \mathcal E_1(a \tau)}) = 0
			\mbox{ for }
			-5 \leq a \leq -1;
			\\ &\operatorname {H^\bullet} ({\mathcal E_1 \otimes \mathcal E_1(a \tau + E)}) = 0
			\mbox{ for }
			a = -6, -5, -4
			\mbox{.}
		\end{align*}
		Furthermore, $\mathrm H^\bullet (X, \mathcal E_1 \otimes \mathcal E_1) \cong \mathrm H^\bullet (X, \mathcal E_1 \otimes \mathcal E_1 (\tau - E)) \cong \mathds C[0]$, hence $\mathcal E_1 (a\tau + bE)$ is exceptional for any $a, b \in \mathds Z$.
	\end{lemma}
	
	\begin{proof}
		Tensoring \cref{E_1 extension} by $\mathcal E_1(a \tau +b E)$, we obtain:
		$$	0 \to \mathcal E_1(a \tau +b E) \to \mathcal E_1\otimes\mathcal E_1(a \tau +b E) \to \mathcal E_1((a-1)\tau+(b+1)E) \to 0
		\mbox{.} $$
		From what has already been proved in \cref{lemma: vanishE1}, we see that the cohomology of the left-hand and the right-hand terms vanish within the desired bounds.

		When $a=b=0$ (resp. $a=1$ and $b=-1$), the cohomology of the rightmost (resp. leftmost) term vanish by \cref{lemma: vanishE1}, hence $\mathrm H^\bullet (X, \mathcal E_1 \otimes \mathcal E_1) \cong \mathrm H^\bullet (X, \mathcal E_1) \cong \mathds C[0]$ (resp. $\mathrm H^\bullet (X, \mathcal E_1 \otimes \mathcal E_1 (\tau - E)) \cong \mathrm H^\bullet (X, \mathcal E_1) \cong \mathds C[0]$).
		Lastly, we note that $\Ext^\bullet (\mathcal E_1, \mathcal E_1) \cong \mathrm H^\bullet (\mathcal E_1^* \otimes \mathcal E_1) \cong \mathrm H^\bullet (\mathcal E_1 \otimes \mathcal E_1 (\tau - E))$, by \cref{dual of E1}.
	\end{proof}
	
	\begin{lemma}
		\label{vanishE2}
		We have the following vanishing properties for the cohomology of twists of $\mathcal E_2$:
		\begin{align*}
			& \operatorname{H}^\bullet({\mathcal E_2(-E)}) = 0;
			\\& \operatorname{H}^\bullet({\mathcal E_2(a\tau + b E)}) = 0
			\text{ for }
			-7 \leq a \leq -1 \text{ and } b = 0,1 ;
			\\ & \operatorname{H}^\bullet({\mathcal E_2(a \tau + 2 E)})
			= 0
			\text{ for }
			a = -7, -6, -5.
		\end{align*}
	\end{lemma}
	
	\begin{proof}
		The twist of the short exact sequence in \cref{E2} by the invertible sheaf $\mathcal O (a \tau + b E)$ gives:
		$$	0 \to W^* \otimes \mathcal O (a\tau +bE) \to \mathcal E_2(a\tau +bE) \to \mathcal O_E (L) \otimes \mathcal O (a\tau +bE) \to 0
		\mbox{.}$$
		Let us calculate the cohomology of the rightmost term:
		\begin{align*}
			\operatorname{H^\bullet} ({X, \mathcal O_E (L) \otimes \mathcal O (a\tau +b E)}) 						&\cong
			\operatorname{H^\bullet} ({E, L \otimes p^* \mathcal O _Z(a,a) \otimes \mathcal O _E(bE)}) 			\\
			&\cong \operatorname{H^\bullet} ({E, p^* \mathcal O _Z(a,a-1) \otimes \mathcal O _{E/Z}(-b-1)})		\\
			&\cong \operatorname{H^\bullet} ({Z, \mathcal O _Z(a,a-1) \otimes \mathrm Rp_* \mathcal O _{E/Z}(-b-1)})
			\mbox{.}
		\end{align*}
		If $b=0,1$, then $\mathrm Rp_* \mathcal O _{E/Z}(-b-1) = 0$, hence
		$\operatorname{H^\bullet} ({\mathcal O _E (L) \otimes \mathcal O (a \tau + b E)}) =0$ for such $b$.
		We also know, due to \cref{lemma: vanishO}, that for $-7 \leq a \leq -1$ and $0 \leq b \leq 2$,
		as well as for $a=0$ and $b=-1$, the cohomology of the left-hand term vanishes.
		We thus get that $\mathcal E_2(a \tau + b E)$ has zero cohomology for $-7 \leq a \leq -1$ and $b=0,1$.
		
		For the first claim, we substitute $b=-1$ and get:
		\begin{align*}
			\operatorname{H^\bullet} ({X, \mathcal O_E (L) \otimes \mathcal O (a\tau - E)}) 
			&\cong \operatorname{H^\bullet} ({Z, \mathcal O _Z(a,a-1)})
			\\ &= 0
			\text{ for } a = -2, -1, 0.
		\end{align*}
		
		For the last case, we substitute $b=2$ and get:
		\begin{align*}
			\operatorname{H^\bullet} ({X, \mathcal O_E (L) \otimes \mathcal O (a\tau + 2 E)}) 						&\cong \operatorname{H^\bullet} ({Z, \mathcal O _Z(a,a-1) \otimes \mathrm Rp_* \mathcal O _{E/Z}(-3)})
			\\ &\cong \mathrm{H^{\bullet-2}} (Z, \mathcal O_Z(a+5,a+4))^*
			\text{, by \cref{-3}}
			\\ &= 0
			\text{ for } a = -7,-6,-5.
			\qedhere
		\end{align*}
	\end{proof}
	
	\begin{lemma}
		\label{vanishE1E2}
		We have the following vanishing properties for the cohomology of twists of $\mathcal E_1 \otimes \mathcal E_2$:
		\begin{align*}
			&\operatorname{H^\bullet} ({\mathcal E_1 \otimes \mathcal E_2(-E)}) = 0;
			\\&\operatorname{H^\bullet} ({\mathcal E_1 \otimes \mathcal E_2(a \tau)}) = 0
			\mbox{ for }
			-6 \leq a \leq -1 ;
			\\ & \operatorname{H^\bullet} ({\mathcal E_1 \otimes \mathcal E_2(a \tau + E)}) = 0
			\mbox{ for }
			a = -6,-5,-4
			\mbox{.}
		\end{align*}
	\end{lemma}
	
	\begin{proof}
		Tensoring \cref{E_1 extension} by $\mathcal E_2(a \tau +b E)$, we obtain:
		$$	0 \to \mathcal E_2(a \tau +b E) \to \mathcal E_1 \otimes \mathcal E_2(a \tau +b E) \to \mathcal E_2((a-1) \tau + (b+1) E) \to 0
		\mbox{.}$$
		\cref{vanishE2} implies vanishing of cohomology of the left- and right-hand side terms within the desired bounds.
	\end{proof}
	
	\begin{lemma}
		\label{lemma: cohg zero of E2E2}
		We have the following vanishing properties for the cohomology of twists of $\mathcal E_2 \otimes \mathcal E_2$:
		\begin{align*}
			& \operatorname{H}^\bullet({\mathcal E_2 \otimes \mathcal E_2 (a\tau)}) = 0
			\text{ for }
			-7 \leq a \leq -1;
			\\ & \operatorname{H}^\bullet({\mathcal E_2 \otimes \mathcal E_2 (a \tau + E)})
			= 0
			\text{ for }
			-7 \leq a \leq -4.
		\end{align*}
	\end{lemma}
	
	\begin{proof}
		Tensoring the short exact sequence in \cref{E2} by $\mathcal E_2(a \tau +b E)$ we obtain:
		\[ 0\rightarrow W^*\otimes \mathcal E_2(a \tau +b E) \rightarrow \mathcal{E}_2 \otimes \mathcal E_2(a \tau +b E) \rightarrow \mathcal E_2(a \tau +b E) \otimes \mathcal{O}_E(L)\rightarrow 0.
		\]
		By \cref{vanishE2}, cohomology of the leftmost term vanish within the required bounds.
		Using \cref{E2} and projection formula, we rewrite the rightmost term as 
		\[ 
		\mathcal E_2(a \tau + b E) \otimes \mathcal{O}_E(L) \cong j_* \left( j^* \mathcal E_2 \otimes p^* \mathcal O_Z (a,a-1) \otimes \mathcal O_{E/Z} (-b-1) \right). 
		\]
		Applying $Lj^*$ to the sequence in \cref{E2} yields the following exact sequence on $E$:
		\[
		0 \to L \otimes \mathcal O_{E/Z} (1) \to W^* \otimes \mathcal O_E \to j^* \mathcal E_2 \to L \to 0.
		\]
		We notice that $L \otimes \mathcal O_{E/Z} (1) \cong p^* \mathcal O_Z (0,-1)$, hence the Euler exact sequence for $\mathds P W^*$ implies that the following sequence is exact on $Z$:
		\[
		0 \to p^* (\mathcal O_{\mathds PW} \boxtimes T_{\mathds PW^*}(-1) ) \to j^* \mathcal E_2 \to p^* \mathcal O_Z(0,-1) \otimes \mathcal O_{E/Z} (-1) \to 0.
		\]
		Twisting it by $p^* \mathcal O_Z (a,a-1) \otimes \mathcal O_{E/Z} (-b-1)$ yields the short exact sequence
		\begin{align*}
			0 
			& \to p^* (\mathcal O_{\mathds PW}(a) \boxtimes T_{\mathds PW^*}(a-2)) \otimes \mathcal O_{E/Z} (-b-1)
			\\& \to j^* \mathcal E_2 \otimes p^* \mathcal O_Z (a,a-1) \otimes \mathcal O_{E/Z} (-b-1)
			\\& \to p^* \mathcal O_Z(a,a-2) \otimes \mathcal O_{E/Z} (-b-2) \to 0.
		\end{align*}
		A direct calculation (using \cref{-3}) shows that the cohomology of the outer terms vanish when $b=0$, or when $b=1$ and $-7 \leq a \leq -4$.
		Tracing back the short exact sequences completes the proof.
	\end{proof}
	
	\begin{lemma}
		\label{lemma: vanishE2*}
		We have the following vanishing properties for the cohomology of twists of $\mathcal E_2^*$:
		\begin{align*}
			& \mathrm{H}^{\bullet}(\mathcal{E}_2^*(a\tau))=0
			\text{ for } -3\leq a\leq 0;
			\\ & \mathrm{H}^{\bullet}(\mathcal{E}_2^*(a\tau+bE))=0
			\text{ for } -7\leq a\leq -1 \text{ and } b = 1,2;
			\\ & \mathrm{H}^{\bullet}(\mathcal{E}_2^*(a\tau+3E))=0
			\text{ for } a = -8, -7, -6.
		\end{align*}
	\end{lemma}
	
	\begin{proof}
		We tensor the short exact sequence in \cref{2dual} by $\mathcal{O}(a\tau+bE)$ and obtain: $$0\rightarrow \mathcal{E}_2^*(a\tau+bE)\rightarrow W\otimes\mathcal{O}(a\tau+bE)\rightarrow j_*p^*\mathcal{O}_Z(0,1)\otimes\mathcal{O}(a\tau+bE)\rightarrow 0.$$
		We have seen in \cref{O(*)} that the cohomology of $\mathcal{O}(a\tau+bE)$ vanish within the desired bounds, except when $a=b=0$.
		We can show that all the cohomology of $j_*\mathcal{O}_Z(0,1)\otimes\mathcal{O}(a\tau+bE)$ are zero using the same argument as before; namely, we make use of the isomorphism
		$ \mathrm H^\bullet (X, j_*p^* \mathcal O_Z(0,1) \otimes \mathcal O_X(a\tau+bE))
		\cong \mathrm H^\bullet (Z, \mathcal O_Z(a,a+1) \otimes R p_* \mathcal O_{E/Z} (-b))$
		and note that the latter vanishes when $-3\leq a\leq -1$ and $b=0$, or $b=1,2$, or $a=-8, -7, -6$ and $b=3$.
		It remains to prove that $\mathrm{H}^{\bullet}(\mathcal{E}_2^*)=0$. For this, we take the long exact sequence associated to \cref{2dual}:
		$$0\rightarrow \mathrm{H}^0(X,\mathcal{E}_2^*)\rightarrow W \rightarrow W\cong \mathrm{H}^0(X,j_*\mathcal{O}_Z(0,1))\rightarrow \mathrm{H}^1(X,\mathcal{E}_2^*)\rightarrow0. $$
		One can easily see that the middle map $W\rightarrow W$ is an isomorphism. As a result, $\mathrm{H}^{\bullet}(\mathcal{E}_2^*)=0$.
	\end{proof}
	
	\begin{lemma}
		\label{lemma: vanishE1E2*}
		We have the following vanishing properties for the cohomology of twists of $\mathcal E_1 \otimes \mathcal E_2^*$:
		\begin{equation*}
			\begin{split}
				& \operatorname{H^\bullet} ({\mathcal E_1 \otimes \mathcal E_2^* (a \tau)}) = 0
				\text { for } -3 \leq a \leq 0;
				\\& \operatorname{H^\bullet} ({\mathcal E_1 \otimes \mathcal E_2^* (a \tau + E)}) = 0
				\text { for } -6 \leq a \leq -1;
				\\& \operatorname{H^\bullet} ({\mathcal E_1 \otimes \mathcal E_2^* (a \tau + 2E)}) = 0
				\text { for } a = -7, -6, -5.        
			\end{split}
		\end{equation*}
	\end{lemma}
	
	\begin{proof}
		Tensoring the sequence in \cref{E_1 extension} with $\mathcal E_2^*(a \tau +b E)$, we obtain:
		$$	0 \to \mathcal E_2^*(a \tau +b E) \to \mathcal E_1 \otimes \mathcal E_2^* (a \tau +b E) \to \mathcal E_2^*((a-1)\tau + (b+1) E) \to 0
		\mbox{.} $$
		Cohomology of the outer terms vanish within the required bounds by \cref{lemma: vanishE2*}.
	\end{proof}

	\begin{lemma}
		We have the following vanishing properties for the cohomology of twists of $\mathcal E_2 \otimes \mathcal E_2^*$:
		\begin{align*}
			& \mathrm{H}^{\bullet}(\mathcal{E}_2 \otimes \mathcal{E}_2^* (a\tau))=0
			\text { for } a = -2,-1;
			\\& \mathrm{H}^{\bullet}(\mathcal{E}_2 \otimes \mathcal{E}_2^* (a\tau+E))=0
			\text { for } -7 \leq a \leq-1;
			\\& \mathrm{H}^{\bullet}(\mathcal{E}_2 \otimes \mathcal{E}_2^* (a\tau+2E))=0
			\text { for } a = -7,-6.
		\end{align*}
		Moreover, $\mathrm{H}^{\bullet}(\mathcal{E}_2 \otimes \mathcal{E}_2^* (-5\tau+2E))=\mathds{C}[-3]$ and $\mathrm{H}^{\bullet}(\mathcal{E}_2 \otimes \mathcal{E}_2^* )=\mathds{C}[0]$; In particular, objects $\mathcal E_2 (a \tau + b E)$ and $\mathcal E_2^* (a \tau + b E)$ are exceptional for all $a, b \in \mathds Z$.
		\label{E2E2^*}
	\end{lemma}
	
	\begin{proof}
		First we tensor the short exact sequence in \cref{E2} with $\mathcal{E}_2^*(a\tau+bH)$:
		$$0\rightarrow W^*\otimes \mathcal{E}_2^*(a\tau+bE)\rightarrow \mathcal{E}_2\otimes \mathcal{E}_2^*(a\tau+bE)\rightarrow \mathcal{O}_E(L)\otimes \mathcal{E}_2^*(a\tau+bE) \rightarrow 0.$$
		We have seen that all the cohomology of the first term vanishes within the desired bounds. To understand the last term, we first tensor the complex in \cref{2dual} with $\mathcal{O}_E$ 
		and obtain the exact triangle:
		$$ \mathcal{E}_2^*\otimes^{\mathds{L}}\mathcal{O}_E=\mathcal{E}_2^*|_E\rightarrow W\otimes\mathcal{O}\otimes^{\mathds{L}}\mathcal{O}_E=W\otimes \mathcal{O}_E\rightarrow j_*\mathcal{O}_Z(0,1)\otimes^{\mathds{L}}\mathcal{O}_E\rightarrow \mathcal{E}_2^*|_E[1],$$
		where $j_*\mathcal{O}_Z(0,1)\otimes^{\mathds{L}}\mathcal{O}_E=j_*\mathcal{O}_Z(0,1)\otimes[0\rightarrow \mathcal{O}(-E)\rightarrow \mathcal{O}\rightarrow 0]=[0\rightarrow j_*\mathcal{O}_Z(0,1)\otimes \mathcal{O}(-E)\xrightarrow{0} j_*\mathcal{O}_Z(0,1)\rightarrow 0]$ has 0-th cohomology $j_*\mathcal{O}_Z(0,1)$ and $(-1)$-th cohomology $j_*\mathcal{O}_Z(0,1)\otimes \mathcal{O}(-E)$. Now we take the long exact sequence associated to the exact triangle and consider it over $E$:
		$$0\rightarrow \mathcal{O}_Z(0,1)\otimes \mathcal{O}_{E/Z}(1)\rightarrow \mathcal{E}_2^*|_E\rightarrow W\otimes \mathcal{O}_E\rightarrow  \mathcal{O}_Z(0,1)\rightarrow 0.$$
		Using the pull back  via $p: E\rightarrow Z$ of the Euler sequence over $\mathds{P}(W^*)$, we can identify the kernel of the last surjection with $p^* \left( \mathcal O \boxtimes \Omega^1_{\mathds{P}(W^*)}(1) \right)$.
		Thus we have the short exact sequence:
		\begin{equation}
			0\rightarrow \mathcal{O}_Z(0,1)\otimes \mathcal{O}(-E)\rightarrow \mathcal{E}_2^*|_E\rightarrow p^* \left( \mathcal O \boxtimes \Omega^1_{\mathds{P}(W^*)}(1) \right) \rightarrow 0.
			\label{E2^*|H}
		\end{equation}
		We twist the above sequence by the line bundle $L\otimes j^*\mathcal{O}(a\tau+bE)=\mathcal{O}_Z(a,a-1)\otimes \mathcal{O}_{E/Z}(-b-1)$ and get:
		\begin{equation}
			\label{eq: E2*|E}
			0\rightarrow \mathcal{O}_Z(a,a)\otimes \mathcal{O}_{E/Z}(-b)\rightarrow \mathcal{O}_E(L)\otimes \mathcal{E}_2^*(a\tau+bE)
			\rightarrow p^* \left( \mathcal O(a) \boxtimes \Omega^1_{\mathds{P}(W^*)} (a) \right) \otimes \mathcal{O}_{E/Z}(-b-1)\rightarrow 0.
		\end{equation}
		Now a calculation of cohomology of the left- and right-most terms (using \cref{-3}) gives the desired vanishing of cohomology in the Lemma. 
		
		When $(a,b) = (-5,2)$ or $(0,0)$, the vanishing of the cohomology of $\mathcal E_2^* (a\tau+bE)$ still holds by \cref{E2E2^*}, hence
		\begin{align*}
			\mathrm{H}^{\bullet}(\mathcal{E}_2\otimes \mathcal{E}_2^*(a\tau+bE))
			& \cong \mathrm H^\bullet (X,\mathcal{O}_E(L)\otimes \mathcal{E}_2^*(a\tau+bE))
			.
		\end{align*}
		When $a = -5$ and $b=2$, we notice in addition that $\mathrm H^\bullet (E, p^* \mathcal{O}_Z(a,a)\otimes \mathcal{O}_{E/Z}(-b)) = 0$, therefore
		\begin{align*}
			\mathrm{H}^{\bullet}(\mathcal{E}_2 \otimes \mathcal{E}_2^* (-5\tau+2E))
			& \cong \mathrm{H}^{\bullet}(Z,(\mathcal{O}_{\mathds{P}(W)} (-5) \boxtimes\Omega^1_{\mathds{P}(W^*)}(-5)) \otimes \mathrm{R}p_*\mathcal{O}_{E/Z}(-3))
			\text{, from \eqref{eq: E2*|E}}
			\\& \cong \mathrm{H}^{\bullet}(\mathds{P}(W^*), \Omega^1_{\mathds{P}(W^*)}[-2])
			\text{, by \cref{-3} and K\"unneth}
			\\& \cong \mathds{C}[-3].
		\end{align*}
		When $a=b=0$, we have instead $\mathrm H^\bullet (E, p^* (\mathcal O \boxtimes \Omega) \otimes \mathcal O_{E/Z} (-1) = 0$, and with this, we compute
		\begin{align*}
			\mathrm{H}^{\bullet}(\mathcal{E}_2 \otimes \mathcal{E}_2^*)
			& \cong \mathrm{H}^{\bullet}(E, \mathcal O_E)
			\text{, from \eqref{eq: E2*|E}}
			\\& \cong \mathds{C}[0].
			\qedhere
		\end{align*}
	\end{proof}

	\begin{lemma}
		\label{HE2^*E2^*}
		We have the following vanishing properties for the cohomology of twists of $\mathcal E_2^* \otimes \mathcal E_2^*$:
		\begin{align*}
			& \mathrm{H}^{\bullet}(\mathcal{E}_2^* \otimes \mathcal{E}_2^* (a\tau + E))=0
			\text { for } -4 \leq a \leq -1;
			\\& \mathrm{H}^{\bullet}(\mathcal{E}_2^* \otimes \mathcal{E}_2^* (a\tau+2E))=0
			\text { for } -7 \leq a \leq -1;
			\\& \mathrm{H}^{\bullet}(\mathcal{E}_2^* \otimes \mathcal{E}_2^* (a\tau+3E))=0
			\text { for } a = -7,-6.
		\end{align*}
	\end{lemma}
	
	\begin{proof}
		We tensor the sequence in \cref{2dual} with $\mathcal{E}_2^*(a\tau+bE)$ and get:
		$$0\rightarrow \mathcal{E}_2^*\otimes \mathcal{E}_2^*(a\tau+bE)\rightarrow W\otimes\mathcal{E}_2^*(a\tau+bE)\rightarrow j_*\mathcal{O}_Z(0,1)\otimes \mathcal{E}_2^*(a\tau+bE)\rightarrow 0.$$
		By \cref{lemma: vanishE2*}, the cohomology of $\mathcal{E}_2^*(a\tau+bE)$ vanish under our assumption.
		To understand the last term in the above sequence, we tensor \eqref{E2^*|H} with $\mathcal{O}_Z(0,1)\otimes j^*\mathcal{O}(a\tau+bE)=\mathcal{O}_Z(a,a+1)\otimes \mathcal{O}_{E/Z}(-b)$:
		\begin{equation*}
			0\rightarrow \mathcal{O}_Z(a,a+2)\otimes \mathcal{O}_{E/Z}(-b+1)\rightarrow j_*\mathcal{O}_Z(0,1)\otimes \mathcal{E}_2^*(a\tau+bE)\rightarrow p^*(\mathcal O(a) \boxtimes \Omega^1_{\mathds{P}(W^*)}(a+2)) \otimes \mathcal{O}_{E/Z}(-b)\rightarrow 0.
			\label{E2E_2^*|H}
		\end{equation*}
		A direct calculation, using \cref{-3}, shows that the cohomology of $j_*\mathcal{O}_Z(0,1)\otimes \mathcal{E}_2^*(a\tau+bE)$ vanish under our assumption.
	\end{proof}
	
	\begin{lemma}
		\label{lemma: filtration on slE1}
		The sheaf $\mathfrak{sl} (\mathcal E_1)$ is a filtered object, with filtration $\mathfrak{sl} (\mathcal E_1) = \Fil_0 \supset \Fil_1 \supset \Fil_2 \supset 0$ such that $\Fil_1 \cong \mathcal E_1 (\tau-E)$.
		Its associated graded quotients are $\gr_0 \cong \mathcal O(-\tau+E)$, $\gr_1 \cong \mathcal O$, $\gr_2 \cong \mathcal O(\tau-E)$.
		In particular, we have a short exact sequence
		\begin{equation}
			\label{eq: ses for slE1}
			0 \to \mathcal E_1 (\tau-E) \to \mathfrak{sl} (\mathcal E_1) \to \mathcal O(-\tau+E) \to 0.
		\end{equation}
	\end{lemma}
	
	\begin{proof}
		By tensoring the sequence from \cref{E_1 extension} with its dual, $0 \to \mathcal O (\tau - E) \to \mathcal E_1^* \to \mathcal O \to 0$, we obtain a 3-step filtration
		\[
		\mathcal E_1^* \otimes \mathcal E_1 \supset \Fil_1 (\mathcal E_1^* \otimes \mathcal E_1)
		\supset \mathcal E_1^* \supset \mathcal O(\tau - E) \supset 0
		\]
		with $\Fil_1 (\mathcal E_1^* \otimes \mathcal E_1) = \Ker \Big(
		\mathcal E_1^* \otimes \mathcal E_1 \to \mathcal O\otimes \mathcal O(-\tau + E)
		\Big)$, and the subquotients
		$(\mathcal E_1^* \otimes \mathcal E_1) / \Fil_1 (\mathcal E_1^* \otimes \mathcal E_1) \cong \mathcal O (-\tau + E)$ and $\Fil_1 (\mathcal E_1^* \otimes \mathcal E_1) / \mathcal E_1^* \cong \mathcal O$.
		We consider the image of this filtration under the quotient map $\mathcal E_1\otimes \mathcal E^*_1\rightarrow\mathfrak {sl} (\mathcal E_1)$. In local charts we easily see that the image of $\mathcal E^*_1$ and $\mathcal O(\tau-E)$ are the same, which is isomorphic to $\mathcal O(\tau-E)$. The image of $\Fil_1 (\mathcal E_1^* \otimes \mathcal E_1)$ is a rank 2 bundle, which we denote by $\Fil_1$. Obviously $\Fil_1$ fits into a short exact sequence $0 \to \mathcal O (\tau - E) \to \Fil_1 \to \mathcal O \to 0$, which is nonsplit since $\mathfrak{sl}(\mathcal E_1)$ and thus $\Fil_1$ have no global section. Because there can be only one such nontrivial extension of $\mathcal O(\tau-E)$ by $\mathcal O$ in view of \cref{lemma: vanishO}, we conclude that $\Fil_1 \cong \mathcal E_1 (\tau - E)$ by \cref{E_1 extension}. Remembering that $\det \mathfrak {sl} (\mathcal E_1) \cong \mathcal O$, we conclude that we get a 2-step filtration with associated graded factors as in the statement.
	\end{proof}
	
	\begin{lemma}
		\label{muorthogonal}
		We have the following vanishing properties for the cohomology of twists of $\mathfrak{sl} (\mathcal E_1)$:
		\begin{align*}
			& \mathrm{H}^{\bullet}(\mathfrak{sl} (\mathcal E_1) )=0;
			\\& \mathrm{H}^{\bullet}(\mathfrak{sl} (\mathcal E_1) (a\tau+E))=0
			\text { for } -6 \leq a \leq -2;
			\\& \mathrm{H}^{\bullet}(\mathfrak{sl} (\mathcal E_1) (a\tau+2E))=0
			\text { for } a = -7,-6,-5.
		\end{align*}
		In particular, $\mathfrak{sl} (\mathcal U_1)$ and $\mathcal O_Y$ are mutually orthogonal, i.e. $\rhom_Y^{\bullet}(\mathfrak{sl} (\mathcal U_1), \mathcal O_Y) = \rhom_Y^{\bullet} (\mathcal O_Y , \mathfrak{sl} (\mathcal U_1)) = 0$.
	\end{lemma}
	
	\begin{proof}
		The outside terms of the twist of \eqref{eq: ses for slE1} by $\mathcal O (a\tau + bE)$ are $\mathcal E_1 ((a+1)\tau + (b-1)E)$ and $\mathcal O ((a-1)\tau + (b+1)E)$.
		By \cref{O(*)} and \cref{lemma: vanishE1}, their cohomology vanish within the desired bounds.
	\end{proof}
	
	\begin{lemma}
		We have the following vanishing properties for the cohomology of twists of $\mathfrak{sl} (\mathcal E_1) \otimes \mathcal E_1$:
		\begin{align*}
			& \mathrm{H}^{\bullet}(\mathfrak{sl} (\mathcal E_1) \otimes \mathcal E_1 (- \tau) )=0;
			\\& \mathrm{H}^{\bullet}(\mathfrak{sl} (\mathcal E_1) \otimes \mathcal E_1 (a\tau+E))=0
			\text { for } a = -6,-5,-4;
			\\& \mathrm{H}^{\bullet}(\mathfrak{sl} (\mathcal E_1) \otimes \mathcal E_1 (-7\tau+2E))=0 .
		\end{align*}
	\end{lemma}
	
	\begin{proof}
		The outside terms of the twist of \eqref{eq: ses for slE1} by $\mathcal E_1 (a\tau + bE)$ are $\mathcal E_1 \otimes \mathcal E_1 ((a+1)\tau + (b-1)E)$ and $\mathcal E_1 ((a-1)\tau + (b+1)E)$.
		By \cref{lemma: vanishE1} and \cref{vanishE1E1}, their cohomology vanish for 
		$b=0$ and $a=-1$, or
		$b=1$ and $a = -6,-5,-4$.
		
		For the last vanishing, we use \cref{lemma: induced Serre duality} and \cref{dual of E1} to get $\mathrm{H}^{\bullet}(\mathfrak{sl} (\mathcal E_1) \otimes \mathcal E_1 (-7\tau+2E)) \cong
		\mathrm{H}^{6-\bullet}(\mathfrak{sl} (\mathcal E_1) \otimes \mathcal E_1 (-\tau))^*$, which vanishes by what we have just proved.
	\end{proof}
	
	\begin{lemma}
		\label{lemma: vanishslE1E2}
		We have the following vanishing properties for the cohomology of twists of $\mathfrak{sl} (\mathcal E_1) \otimes \mathcal E_2$:
		\begin{align*}
			& \mathrm{H}^{\bullet}(\mathfrak{sl} (\mathcal E_1) \otimes \mathcal E_2 (- \tau) )=0;
			\\& \mathrm{H}^{\bullet}(\mathfrak{sl} (\mathcal E_1) \otimes \mathcal E_2 (a\tau+E))=0
			\text { for } a = -6,-5,-4;
			\\& \mathrm{H}^{\bullet}(\mathfrak{sl} (\mathcal E_1) \otimes \mathcal E_2 (-7\tau+2E))=0 .
		\end{align*}
	\end{lemma}
	
	\begin{proof}
		The outside terms of the twist of \eqref{eq: ses for slE1} by $\mathcal E_2 (a\tau + bE)$ are $\mathcal E_1 \otimes \mathcal E_2 ((a+1)\tau + (b-1)E)$ and $\mathcal E_2 ((a-1)\tau + (b+1)E)$.
		By \cref{vanishE2} and \cref{vanishE1E2}, their cohomology vanish for 
		$b=0$ and $a=-1$, or
		$b=1$ and $a = -6,-5,-4$.
		
		For the last vanishing, we use \cref{lemma: induced Serre duality} to get $\mathrm{H}^{\bullet}(\mathfrak{sl} (\mathcal E_1) \otimes \mathcal E_2 (-7\tau+2E)) \cong
		\mathrm{H}^{6-\bullet}(\mathfrak{sl} (\mathcal E_1) \otimes \mathcal E_2^* (-2\tau+E))^*$, which vanishes by the first claim of \cref{lemma: vanishslE1E2*}, which is independent of this Lemma.
	\end{proof}
	
	\begin{lemma}
		\label{lemma: vanishslE1E2*}
		We have the following vanishing properties for the cohomology of twists of $\mathfrak{sl} (\mathcal E_1) \otimes \mathcal E_2^*$:
		\begin{align*}
			& \mathrm{H}^{\bullet}(\mathfrak{sl} (\mathcal E_1) \otimes \mathcal E_2^* (a\tau+E) )=0
			\text{ for } -4 \leq a \leq -1;
			\\& \mathrm{H}^{\bullet}(\mathfrak{sl} (\mathcal E_1) \otimes \mathcal E_2^* (a\tau+2E))=0
			\text { for } a = -7,-6,-5;
			\\& \mathrm{H}^{\bullet}(\mathfrak{sl} (\mathcal E_1) \otimes \mathcal E_2^* (-8\tau+3E))=0 .
		\end{align*}
	\end{lemma}
	
	\begin{proof}
		The outside terms of the twist of \eqref{eq: ses for slE1} by $\mathcal E_2^* (a\tau + bE)$ are $\mathcal E_1 \otimes \mathcal E_2^* ((a+1)\tau + (b-1)E)$ and $\mathcal E_2^* ((a-1)\tau + (b+1)E)$.
		By \cref{lemma: vanishE2*} and \cref{lemma: vanishE1E2*}, their cohomology vanish for 
		$b=1$ and $-4 \leq a \leq -1$, or
		$b=2$ and $a = -7,-6,-5$.
		
		For the last vanishing, we use \cref{lemma: induced Serre duality} to get $\mathrm{H}^{\bullet}(\mathfrak{sl} (\mathcal E_1) \otimes \mathcal E_2^* (-8\tau+3E)) \cong
		\mathrm{H}^{6-\bullet}(\mathfrak{sl} (\mathcal E_1) \otimes \mathcal E_2 (-\tau))^*$, which vanishes by the first claim of \cref{lemma: vanishslE1E2}, which is independent of this Lemma.
	\end{proof}
	
	\begin{lemma}
		The object $\mathfrak{sl}(\mathcal E_1)$ is exceptional.
	\end{lemma}
	
	\begin{proof}
		We need to prove that $\mathrm H^\bullet (X, \mathfrak{sl}(\mathcal E_1) \otimes \mathfrak{sl}(\mathcal E_1) ) \cong \mathds C [0]$.
		By \cref{remark: higher cohg vanish}, it is enough to check that the Euler characteristic $\chi(\mathfrak{sl}(\mathcal E_1) \otimes \mathfrak{sl}(\mathcal E_1)) = 1$, which follows from the Hirzebruch-Riemann-Roch Theorem using the Todd class and chern characters provided before.
	\end{proof}
	
	\begin{remark}
		A large part of exceptionality follows simply from Teleman Quantization and Hirzebruch-Riemann-Roch type calculations. However, in \cref{symmetry} and \cref{1 unique morphism}, we need to calculate for example $\mathrm{H}^{\bullet}(X,\mathcal{E}_2^*\otimes \mathcal{E}_2(-5\tau+2E))=\mathds{C}[-3] $, where the bundle $\mathcal{E}_2^*\otimes \mathcal{E}_2(-5\tau+2E)$ is not the pullback of any bundle over $Y$ and has nonvanishing cohomology in higher degree.
	\end{remark}
	\subsection{Mutations}
	\label{subsection: mutations}

	By \cref{theorem: exceptional collection on Y}, we have an exceptional collection on $Y$.
	Denoting $\langle \mathcal{O}_Y, \mathcal{U}_2^*, \mathcal{U}_1^*, \mathcal{U}_2(1) \rangle$ by $\mathcal A$, we can rewrite \cref{1 exc} as follows:
	\begin{equation}
		\label{equation: shorter notation for exceptional collection}
		\langle \mathfrak{sl} (\mathcal U_1) ,\ \mathcal A,\ \mathcal A(1),\ \mathcal A(2) \rangle.
	\end{equation}
	In this subsection, we will explore further structures of this collection, such as results of mutating some of the objects.

	\begin{proposition}
		\label{proposition: mutating slU1}
		If one right mutates the bundle $\mathfrak{sl}(\mathcal{U}_1)$ across $\mathcal A$ or $\langle\mathcal A, \mathcal{O}_Y(1)\rangle$,
		one gets $\mathfrak{sl}(\mathcal{U}_1)(1)[-2]$.
	\end{proposition}
	
	\begin{proof}
		\textbf{Step 1.}
		\textit{Right mutation of $\mathfrak{sl}(\mathcal{U}_1)$ across $\mathcal O_Y$ is $\mathfrak{sl}(\mathcal{U}_1)$.}
		By \cref{muorthogonal}, the bundle $\mathfrak{sl}(\mathcal{U}_1)$ is orthogonal to $\mathcal{O}_Y$. As a result, the right mutation of $\mathfrak{sl}(\mathcal{U}_1)$ across $\mathcal{O}_Y$ is still $\mathfrak{sl}(\mathcal{U}_1)$ itself. 
		
		For the same reason, we only need to prove that the right mutation of $\mathfrak{sl}(\mathcal{U}_1)$ across $\mathcal A$ is equal to $\mathfrak{sl}(\mathcal{U}_1)(1)[-2]$ because the further mutation of $\mathfrak{sl}(\mathcal{U}_1)(1)[-2]$ across $\mathcal{O}_Y(1)$ won't change anything.
		
		\textbf{Step 2.}
		\textit{Right mutation of $\mathfrak{sl}(\mathcal{U}_1)$ across $\mathcal U_2^*$ is a shifted vector bundle $Q_1[-1]$.}
		In order to understand the right mutation of $\mathfrak{sl}(\mathcal{U}_1)$ across $\mathcal{U}_2^*$, we need to understand the vector space $\mathrm{RHom}^{\bullet}(\mathfrak{sl}(\mathcal{U}_1),\mathcal{U}_2^*)$.
		By Teleman Quantization (\cref{teleman}), we know that 
		$\mathrm{RHom}^{\bullet}(\mathfrak{sl}(\mathcal{U}_1),\mathcal{U}_2^*)$ is isomorphic to
		$\mathrm H^0 (Y, \mathfrak{sl}(\mathcal{U}_1)\otimes\mathcal{U}_2^*)[0]$.
		By calculation using Riemann-Roch Theorem, we conclude that $\mathrm{RHom}^{\bullet}(\mathfrak{sl}(\mathcal{U}_1),$
		$\mathcal{U}_2^*)\cong\mathds{C}^3[0]$, and we will further exhibit an explicit $\operatorname{SL}(W)$-equivariant isomorphism $\mathrm{RHom}^{\bullet}(\mathfrak{sl}(\mathcal{U}_1),\mathcal{U}_2^*)\cong W[0]$, noting that $W \cong \wedge^2 W^*$ as $\operatorname{SL}(W)$-representations.
		We have the universal representation $\mathcal{U}_2^*\rightarrow \mathcal{U}_1^*\otimes W$.
		Applying the second wedge product and using knowledge about Schur functor of tensor product (see \cite[Lemma 0.5]{kapranov83}), we have $\wedge^2\mathcal{U}_2^*\rightarrow \wedge^2(\mathcal{U}_1^*\otimes W)\rightarrow S^2\mathcal{U}_1^*\otimes \wedge^2 W$, thus a map $\wedge^2W^*\otimes S^2\mathcal{U}_1\rightarrow \wedge^2\mathcal{U}_2$.
		Since $\det(\mathcal{U}_1)=\det(\mathcal{U}_2)=\mathcal{O}_Y(-1)$, we have $S^2\mathcal{U}_1\cong \mathfrak{sl}(\mathcal{U}_1)\otimes\mathcal{O}_Y(-1)$ and $\wedge^2\mathcal{U}_2\cong \mathcal{U}_2^*\otimes \mathcal{O}_Y(-1)$.
		These observations give us a map $\wedge^2 W^*\rightarrow \Hom(\mathfrak{sl}(\mathcal{U}_1),\mathcal{U}_2^*)$, which must be an isomorphism by dimension counting.
		One can check that the map $S^2\mathcal{U}_1\rightarrow \wedge^2W\otimes\wedge^2\mathcal{U}_2$ and thus the map $\mathfrak{sl}(\mathcal{U}_1)\rightarrow \wedge^2W\otimes\mathcal{U}_2^*$ are fiberwise
		injective by checking this fact on the fibers over the representative points $\langle x^2,xy,y^2\rangle$, $\langle x^2,xy,xz\rangle$ in the two smallest $SL(W)$-orbits.
		Consequently, we have an exact sequence
		\begin{equation}
			0\rightarrow \mathfrak{sl}(\mathcal{U}_1)\xrightarrow{i} \wedge^2W\otimes\mathcal{U}_2^*\rightarrow Q_1 \rightarrow 0,
			\label{lexact}
		\end{equation}
		where $Q_1$ is a vector bundle. As a result, $Q_1[-1]=\mathrm{R}_{\mathcal{U}_2^*}(\mathfrak{sl}(\mathcal{U}_1))$.
		
		\textbf{Step 3.}
		\textit{Right mutation of $Q_1^*(1)$ across $\mathcal U_2(1)$ is $\mathfrak{sl}(\mathcal{U}_1)(1)[-1]$.}
		By taking dual of the above sequence and then twisting with $\mathcal{O}_Y(1)$, we have another sequence \begin{equation}
			0\rightarrow Q_1^*(1)\rightarrow \wedge^2W^*\otimes\mathcal{U}_2(1)\rightarrow \mathfrak{sl}(\mathcal{U}_1)(1) \rightarrow 0\label{rexact}
		\end{equation}
		which actually shows that $\mathrm{L}_{\mathcal{U}_2(1)} (\mathfrak{sl}(\mathcal{U}_1)(1)) \cong Q_1^*(1)[1]$.
		
		\textbf{Step 4.}
		\textit{Right mutation of $Q_1$ across $\mathcal U_1^*$ is $Q_1$.}
		We will show that $\operatorname{RHom}(Q_1, \mathcal U_1^*) = 0$, which will imply that $\mathrm{R}_{\mathcal{U}_1^*}(Q_1) \cong Q_1$.
		To compute the derived Hom, we apply $\operatorname{RHom}(\blank, \mathcal U_1^*)$ to \cref{lexact}.
		By Teleman Quantization, both $\operatorname{RHom}(\wedge^2 W \otimes \mathcal U_2^*, \mathcal U_1^*)$ and $\operatorname{RHom}(\mathfrak{sl}(\mathcal U_1), \mathcal U_1^*)$ are concentrated in degree zero.
		By calculations analogous to those in \cref{subsection: exceptional collection}, we have that
		$\operatorname{RHom}(\wedge^2 W \otimes \mathcal U_2^*, \mathcal U_1^*)$ and $\operatorname{RHom}(\mathfrak{sl}(\mathcal U_1), \mathcal U_1^*)$ are both $\operatorname{SL}(W)$-equivariantly isomorphic to $\operatorname{End} (W) \cong \mathbb C \oplus \mathfrak{sl}(W)$.
		One can show that the $\operatorname{SL}(W)$-equivariant map
		\begin{equation}
			\label{equation: iso between RHoms}
			i^* : \operatorname{End}(W) \cong \operatorname{RHom}(\wedge^2 W \otimes \mathcal U_2^*, \mathcal U_1^*) \to \operatorname{RHom}(\mathfrak{sl}(\mathcal U_1), \mathcal U_1^*)
		\end{equation}
		induced by $i$ in \cref{lexact} is an isomorphism by checking that $i^*(\mathrm{id}_W) \neq 0$ and $i^*(\alpha) \neq 0$ for some $\alpha \in \mathfrak{sl}(W) \subset \operatorname{End}(W)$.
		To do that, one needs the formula for $i^*$, which we now describe.
		Let $u : W^* \otimes \mathcal U_2^* \to \mathcal U_1^*$ denote the universal representation.
		We start by noting that, with the identification $W^* \cong \wedge^2 W$, the isomorphism $\operatorname{End}(W) \cong \operatorname{RHom}(\wedge^2 W \otimes \mathcal U_2^*, \mathcal U_1^*)$ sends an element $g \in \operatorname{End}(W)$ to $f_g = u \circ (g^* \otimes 1)$.
		If $g = w \otimes \eta \in W \otimes W^* \cong \operatorname{End}(W)$, then $f_g (\xi \otimes \alpha) = \xi(w) \cdot u(\eta \otimes \alpha)$, for $\xi\otimes \alpha\in W^*\otimes \mathcal{U}_2^*\cong\wedge^2 W \otimes \mathcal U_2^*$.
		We will abuse notation and denote by $\eta$ the dual to the morphism $u_{|\eta \otimes \mathcal U_2^*} : \mathcal U_2^* \to \mathcal U_1^*$, and so we get
		$f_{w\otimes \eta} (\xi \otimes \alpha) = \xi(w) \cdot \alpha \circ \eta : \mathcal U_1 \to \mathcal O_Y$.
		If $\{\eta_1,\eta_2,\eta_3\}$ is some basis of $W^*$, then the construction of $i \otimes \mathcal O_Y(-1)$ in Step 2 can be described by the formula:
		\begin{align*}
			S^2\mathcal{U}_1 \ni a\cdot b \quad \mapsto \, \quad
			& \eta_1 \otimes (\eta_2(a)\wedge \eta_3(b)+\eta_2(b)\wedge \eta_3(a))
			\\+\, & \eta_2 \otimes (\eta_3(a)\wedge \eta_1(b)+\eta_3(b)\wedge \eta_1(a))
			\\+\, & \eta_3 \otimes (\eta_1(a)\wedge \eta_2(b)+\eta_1(b)\wedge \eta_2(a)).
		\end{align*}
		We leave it to the reader to plug in $\mathrm{id}_W$ into $i^*$.
		To construct $\alpha$ as above, we now redefine $\eta_1$, $\eta_2$ to be some elements of $W^*$ such that the span of the images of $\eta_1, \eta_2 : \mathcal U_1 \to \mathcal U_2$ is $3$-dimensional, over some fixed point of $Y$.
		Extend this pair to a basis $\{\eta_1,\eta_2,\eta_3\}$ of $W^*$, and let $\{w_1,w_2,w_3\}$ denote the dual basis.
		By an explicit calculation, one gets that $i^*(\alpha) \neq 0$ for $\alpha = w_3 \otimes \eta_1$.
		
		Now \cref{lexact} and \cref{equation: iso between RHoms} together imply that  $\operatorname{RHom}(Q_1, \mathcal U_1^*) = 0$, and this in turn proves the main claim of this step.
		
		\textbf{Step 5.}
		\textit{$Q_1$ is isomorphic to $Q_1^*(1)$.}
		We need to prove that $Q_1\cong Q_1^*(1)$ to finish the proof of this Proposition and we show this fact by constructing a global exact sequence directly: 
		$$0\rightarrow \mathfrak{sl}(\mathcal{U}_1)\otimes\det(W^*)\rightarrow \wedge^2W\otimes\mathcal{U}_2^*\otimes\det(W^*)\rightarrow\wedge^2W^*\otimes\mathcal{U}_2(1)\otimes\det(W)\rightarrow \mathfrak{sl}(\mathcal{U}_1)(1)\otimes\det(W) \rightarrow 0.$$
		In the above sequence, we twist the terms with $\det(W)$, $\det(W^*)$ respectively, so that the sequence will be $GL(W)$-equivariant.
		The first map comes from \cref{lexact} and the last map comes from \cref{rexact}.
		The middle map, i.e. $W^*\otimes \mathcal{U}_2^*\rightarrow W\otimes \mathcal{U}_2(1)\cong W\otimes \wedge^2\mathcal{U}_2^*$ is given by the adjunction of the composition map $W^*\otimes W^*\rightarrow S^2W^*\rightarrow \mathcal{U}_2^*\rightarrow \mathcal{H}om(\mathcal{U}_2^*,\wedge^2\mathcal{U}_2^*)$. One can easily check that this middle morphism is of constant rank 6 and the above sequence is actually a complex. Then the exactness of the above sequence comes from the fact that both $Q_1$ and $Q_1^*(1)$ are of rank 6 and $\det(Q_1)\cong\det(Q_1^*(1))$.
	\end{proof}
	
	As a corollary, the following collections are also exceptional:
	\begin{equation}
		\label{2 exc}
		\langle \mathcal A, \mathfrak{sl}(\mathcal{U}_1^*)(1), \mathcal A(1),\mathcal A(2)\rangle
		\text{, } \ \
		\langle \mathcal A,  \mathcal A(1),\mathfrak{sl}(\mathcal{U}_1^*)(2),\mathcal A(2)\rangle
		\ \text{ and } \
		\langle \mathcal A,  \mathcal A(1),\mathcal A(2),\mathfrak{sl}(\mathcal{U}_1^*)(3)\rangle;
	\end{equation}
	\begin{equation}
		\langle \mathcal A,  \mathcal{O}_Y(1), \mathfrak{sl}(\mathcal{U}_1^*)(1), \mathcal{U}_2^*(1),\mathcal{U}_1^*(1),\mathcal{U}_2(2),\mathcal A(2)\rangle
		\text{ and }
		\langle \mathcal A,  \mathcal A(1),\mathcal{O}_Y(2),\mathfrak{sl}(\mathcal{U}_1^*)(2),\mathcal{U}_2^*(2),\mathcal{U}_1^*(2),\mathcal{U}_2(3)\rangle.
		\label{3 exc}
	\end{equation}
	
	\begin{Corollary}
		\label{ch eq}
		We have the following relations between Chern characters of some tensor combinations of the universal bundles:
		\begin{align*}
			\mathrm{ch}(\mathfrak{sl}(\mathcal{U}_1^*))&=\mathrm{ch}(\mathfrak{sl}(\mathcal{U}_1^*)(1))+3\mathrm{ch}(\mathcal{U}_2^*)-3\mathrm{ch}(\mathcal{U}_2(1)),\\
			\mathrm{ch}(\mathcal{U}_1^*\otimes \mathcal{U}_2(1))&=-\mathrm{ch}( \mathcal{U}_2)+6\mathrm{ch}(\mathcal{O}_Y)+3\mathrm{ch}(\mathcal{U}_2^*)-9\mathrm{ch}(\mathcal{U}_1^*)+3\mathrm{ch}(\mathfrak{sl}(\mathcal{U}_1^*)(1))+3\mathrm{ch}( \mathcal{O}_Y(1)),\\
			\mathrm{ch}(\mathcal{U}_1^*\otimes \mathcal{U}_2(2))&=-\mathrm{ch}( \mathcal{U}_2(1))+6\mathrm{ch}(\mathcal{O}_Y(1))+3\mathrm{ch}(\mathcal{U}_2^*(1))-9\mathrm{ch}(\mathcal{U}_1^*(1))+3\mathrm{ch}(\mathfrak{sl}(\mathcal{U}_1^*)(2))+3\mathrm{ch}( \mathcal{O}_Y(2)).
		\end{align*}
	\end{Corollary}
	
	\begin{proof}
		The first equality comes from the exact sequence in Step 5 in the proof of \cref{proposition: mutating slU1}.
		The second equality is a combination of the first equality with \cref{chern ch}.
		The last equality is obtained by multiplying the second equality by $\mathrm{ch}(\mathcal{O}_Y(1))$.
	\end{proof}
	
	We want to prove that the exceptional collections \eqref{2 exc} and \eqref{3 exc} are full.
	It turns out that the core to proving the fullness is to prove that $\mathcal{U}_1^*\otimes \mathcal{U}_2^*$ and $\mathcal{U}_1^*\otimes \mathcal{U}_2(2)$ lie in the subcategory of $\mathrm{D}^b(Y)$ generated by the above exceptional objects. Notice that $T(\mathcal{U}_1^*\otimes \mathcal{U}_2^*)=\mathcal{U}_1^*\otimes \mathcal{U}_2(2)$. In the following, we will first prove that the following two collections are exceptional:
	\begin{equation}
		\langle \mathcal{O}_Y, \mathcal{U}_2^*,\mathcal{U}_1^*, \mathcal{O}_Y(1),\mathcal{U}_2^*(1),\mathcal{U}_1^*(1),\mathcal{U}_2(2),\mathfrak{sl}(\mathcal{U}_1^*)(2),\mathcal{O}_Y(2),\mathcal{U}_1^*\otimes \mathcal{U}_2(2),\mathcal{U}_2^*(2),\mathcal{U}_1^*(2),\mathcal{U}_2(3)\rangle,\label{4 exc}
	\end{equation}
	\begin{equation}
		\langle \mathcal{O}_Y, \mathcal{U}_2^*,\mathcal{U}_1^*,\mathcal{U}_2(1), \mathcal{U}_1^*\otimes \mathcal{U}_2^*, \mathcal{O}_Y(1),\mathfrak{sl}(\mathcal{U}_1^*)(1), \mathcal{U}_2^*(1),\mathcal{U}_1^*(1),\mathcal{U}_2(2),\mathcal{O}_Y(2),\mathcal{U}_1^*(2),\mathcal{U}_2(3)\rangle. \label{5 exc}
	\end{equation}
	The sequence \eqref{4 exc} is obtained from \eqref{1 exc} by deleting $\mathcal{U}_2(1)$ and adding $\mathcal{U}_1^*\otimes \mathcal{U}_2(2)$.
	The sequence \eqref{5 exc} is obtained from \eqref{1 exc} by deleting $\mathcal{U}_2^*(2)$ and adding $\mathcal{U}_1^*\otimes \mathcal{U}_2^*$.
	
	\begin{proposition}
		The two collections \eqref{4 exc} and \eqref{5 exc} are exceptional. \label{symmetry}
	\end{proposition}
	
	\begin{proof}
		We define a symmetry functor $T$ by the formula $T(\mathcal{G}^.)=(\mathcal{G}^.)^*\otimes \mathcal{O}_Y(3)$.
		Notice that $T(\mathcal{U}_1^*(1))=\mathcal{U}_1(2)\cong \mathcal{U}_1^*(1)$.
		Therefore the symmetry functor $T(\mathcal{G}^.)=(\mathcal{G}^.)^*\otimes \mathcal{O}_Y(3)$ will map the subsequence of \eqref{4 exc}$$\langle \mathcal{U}_2^*,\mathcal{U}_1^*, \mathcal{O}_Y(1),\mathcal{U}_2^*(1),\mathcal{U}_1^*(1),\mathcal{U}_2(2),\mathfrak{sl}(\mathcal{U}_1^*)(2),\mathcal{O}_Y(2),\mathcal{U}_1^*\otimes \mathcal{U}_2(2),\mathcal{U}_2^*(2),\mathcal{U}_1^*(2),\mathcal{U}_2(3)\rangle$$ to the subsequence of \eqref{5 exc} (to be read from right to left instead) $$\langle \mathcal{U}_2^*,\mathcal{U}_1^*,\mathcal{U}_2(1), \mathcal{U}_1^*\otimes \mathcal{U}_2^*, \mathcal{O}_Y(1),\mathfrak{sl}(\mathcal{U}_1^*)(1), \mathcal{U}_2^*(1),\mathcal{U}_1^*(1),\mathcal{U}_2(2),\mathcal{O}_Y(2),\mathcal{U}_1^*(2),\mathcal{U}_2(3)\rangle.$$
		As a result, we only need to prove that \eqref{4 exc} is exceptional and $\mathrm{RHom}^{\bullet}(\mathcal{U}_1^*\otimes \mathcal{U}_2^*, \mathcal{O}_Y)=0$, because we know that \eqref{1 exc} is exceptional and that \eqref{5 exc} is a minor modification of \eqref{1 exc}.
		
		The object $\mathcal{U}_1^*\otimes \mathcal{U}_2(2)$ is exceptional because $\chi(\mathcal{U}_1^*\otimes\mathcal{U}_1\otimes \mathcal{U}_2^*\otimes \mathcal{U}_2)=1$ by Hirzebruch-Riemann-Roch Theorem (use \cref{appendix: Sage code} for calculation), and we can check that the higher cohomology of $\mathcal{U}_1^*\otimes\mathcal{U}_1\otimes \mathcal{U}_2^*\otimes \mathcal{U}_2$ vanish using Teleman Quantization.
		
		The vanishings of $\mathrm{RHom}^{\bullet}(\mathcal{U}_1^*\otimes \mathcal{U}_2^*, \mathcal{O}_Y)$ and \begin{alignat*}{3}
			\mathrm{RHom}^{\bullet}(\mathcal{U}_2(3), \mathcal{U}_1^*\otimes \mathcal{U}_2(2)), \quad\quad 
			&&\mathrm{RHom}^{\bullet}(\mathcal{U}_1^*(2), \mathcal{U}_1^*\otimes \mathcal{U}_2(2)),  \quad\quad
			&&\mathrm{RHom}^{\bullet}(\mathcal{U}_2^*(2), \mathcal{U}_1^*\otimes \mathcal{U}_2(2)),\\
			\mathrm{RHom}^{\bullet}(\mathcal{U}_1^*\otimes \mathcal{U}_2(2),\mathcal{O}_Y(2)), \quad\quad 
			&&\mathrm{RHom}^{\bullet}(\mathcal{U}_1^*\otimes \mathcal{U}_2(2),\mathfrak{sl}(\mathcal{U}_1^*)(2)),\quad\quad 
			&&\mathrm{RHom}^{\bullet}(\mathcal{U}_1^*\otimes \mathcal{U}_2(2),\mathcal{U}_2(2)),\\
			\mathrm{RHom}^{\bullet}(\mathcal{U}_1^*\otimes \mathcal{U}_2(2),\mathcal{U}_1^*(1)),  \quad\quad
			&&\mathrm{RHom}^{\bullet}(\mathcal{U}_1^*\otimes \mathcal{U}_2(2),\mathcal{U}_2^*(1)), \quad\quad 
			&&\mathrm{RHom}^{\bullet}(\mathcal{U}_1^*\otimes \mathcal{U}_2(2),\mathcal{O}_Y(1))
		\end{alignat*}
		all follow the same pattern. One can show that the Euler characteristics of these complexes are 0 using the Hirzebruch-Riemann-Roch Theorem (use \cref{appendix: Sage code} for calculation), and all the higher cohomology groups vanish by Teleman Quantization and easy weight calculations.
		
		What remains to verify is $\mathrm{RHom}^{\bullet}(\mathcal{U}_1^*\otimes \mathcal{U}_2(2),\mathcal{U}_1^*)=\mathrm{RHom}^{\bullet}(\mathcal{U}_1^*\otimes \mathcal{U}_2(2),\mathcal{O}_Y)=\mathrm{RHom}^{\bullet}(\mathcal{U}_1^*\otimes \mathcal{U}_2(2),\mathcal{U}_2^*)=0$.
		The first and the second vanishing come from the fact that $\mathfrak{sl}(\mathcal{U}_1^*)(2)$, $\mathcal{O}_Y(2)$, $\mathcal{U}_1^*(2)$ are all left orthogonal to $\mathcal{U}_2^*$, which can be read from the known exceptional sequences.
		The third vanishing is equivalent to the vanishing of $\mathrm{H}^{\bullet}(X,\mathcal{E}_1\otimes \mathcal{E}_2^*\otimes \mathcal{E}_2^*(-5\tau+2E))$ via the fully faithful pullback $\mathds{L}\pi^*: \mathrm{D}^b(Y)\rightarrow \mathrm{D}^b(\mathds{P}_Y(\mathcal{U}_1))=\mathrm{D}^b(X)$.
		By \cref{E_1 extension}, the object $\mathcal{E}_1$ is an extension of $\mathcal{O}_X(-\tau+E)$ by $\mathcal{O}_X$.
		Thus it suffices to prove that $\mathrm{H}^{\bullet}(X, \mathcal{E}_2^*\otimes \mathcal{E}_2^*(-5\tau+2E))=\mathrm{H}^{\bullet}(X, \mathcal{E}_2^*\otimes \mathcal{E}_2^*(-6\tau+3E))=0$, which follows from \cref{HE2^*E2^*}.
	\end{proof}
	
	\begin{lemma}
		We have
		$\mathrm{RHom}^{\bullet}(\mathcal{U}_1^*\otimes \mathcal{U}_2(2), \mathcal{U}_2(1))=\mathds{C}[-3]$. \label{1 unique morphism}
	\end{lemma}
	
	\begin{proof}
		The Lemma is equivalent to saying $\mathrm{H}^{\bullet}(X,\mathcal{E}_1\otimes\mathcal{E}_2\otimes\mathcal{E}_2^*(-4\tau+E))=\mathds{C}[-3]$ via the fully faithfully pullback $\mathds{L}\pi^*: \mathrm{D}^b(Y)\rightarrow \mathrm{D}^b(\mathds{P}_Y(\mathcal{U}_1))=\mathrm{D}^b(X)$.
		By \cref{E2E2^*}, we know $\mathrm{H}^{\bullet}(X,\mathcal{E}_2^*\otimes \mathcal{E}_2(-5\tau+2E))=\mathds{C}[-3]$ and $\mathrm{H}^{\bullet}(X,\mathcal{E}_2^*\otimes \mathcal{E}_2(-4\tau+E))=0$. We conclude by  \cref{E_1 extension}.
	\end{proof}
	
	Based on the above \cref{1 unique morphism}, as well as \cref{ch eq} and \cref{symmetry}, it is reasonable to guess that if we left mutate $\mathcal{U}_1^*\otimes\mathcal{U}_2(2)$ across $\langle \mathcal{O}_Y(1),\mathcal{U}_2^*(1),\mathcal{U}_1^*(1),\mathcal{U}_2(2),\mathfrak{sl}(\mathcal{U}_1^*)(2),\mathcal{O}_Y(2)\rangle$, then we should get $\mathcal{U}_2(1)[3]$. We will prove this fact in a slightly indirect way.
	
	Recall that $\mathcal{U}_2$ and $\mathcal{U}_1$ are pullbacks of the universal subbundles via the embeddings $Y\subset \mathrm{Gr}(3,S^2W)$ and $Y\subset \mathrm{Gr}(2,S^{2,1}W)$ respectively. We let $L_6[-3]$ denote $\mathcal{U}_2(1)$ and let $L_5[-2]$ denote the twsiting of the pullback of the universal quotient bundle over $\mathrm{Gr}(3,S^2W)$ by $\mathcal{O}_Y(1)$. We have the standard short exact sequence:
	\begin{equation}
		0\rightarrow \mathcal{U}_2(1)=L_6[-3]\rightarrow S^2W\otimes\mathcal{O}_Y(1)\rightarrow L_5[-2]\rightarrow 0 \label{mu1}.
	\end{equation}
	
	\begin{lemma}
		As $SL(W)$-modules, $\mathrm{H}^{\bullet}(\mathcal{U}_2\otimes \mathcal{U}_1^*)\cong W^*[0]$,
		$\mathrm{H}^{\bullet}(\mathcal{U}_2^*)\cong S^2W^*[0]$, and
		$\mathrm{H}^{\bullet}(\mathcal{U}_1^*)\cong S^{2,1}W^*[0]$.
		Furthermore, $\mathrm{H}^{\bullet}(\mathcal{U}_2^*)\otimes \mathrm{H}^{\bullet}(\mathcal{U}_2^*)$ is injectively mapped into $ \mathrm{H}^{\bullet}(\mathcal{U}_2^*\otimes\mathcal{U}_2^*)= \mathds{C}^{39}[0]$ via the natural map; and $\mathrm{H}^{\bullet}(\mathcal{U}_2^*\otimes\mathcal{U}_1^*)\cong\mathrm{H}^{\bullet}(\mathcal{U}_2^*)\otimes \mathrm{H}^{\bullet}(\mathcal{U}_1^*)$ via the natural tensor product map.
		\label{global section}
	\end{lemma}
	
	\begin{proof}
		The vanishing of the higher cohomology of $\mathcal{U}_2^*$, $\mathcal{U}_1^*$, $\mathcal{U}_2^*\otimes \mathcal{U}_2^*$, $\mathcal{U}_2^*\otimes \mathcal{U}_1^*$, $\mathcal{U}_2\otimes \mathcal{U}_1^*$ all follow from Teleman Quantization directly.
		Calculations using Hirzebruch-Riemann-Roch Theorem and \cref{appendix: Sage code} show that $\chi(\mathcal{U}_2\otimes \mathcal{U}_1^*)=3, \ \chi(\mathcal{U}_2^*)=6, \ \chi(\mathcal{U}_1^*)=8, \ \chi(\mathcal{U}_2^*\otimes \mathcal{U}_2^*)=39, \ \chi(\mathcal{U}_2^*\otimes \mathcal{U}_1^*)=48$.
		By construction, we have nontrivial maps $W^*\rightarrow \mathrm{Hom}(\mathcal{U}_1,\mathcal{U}_2)=\mathrm{H}^0(\mathcal{U}_2\otimes \mathcal{U}_1^*), \  S^2W^*\rightarrow \mathrm{H}^0(\mathcal{U}_2^*), \  S^{2,1}W^*\rightarrow \mathrm{H}^0(\mathcal{U}_1^*)$, and these maps must be isomorphisms because of dimension countings and the fact that $W^*,S^2W^*, S^{2,1}W^*$ are irreducible $SL(W)$-modules.
		For the last two statements, we want to prove that both $S^2W^*\otimes S^2W^*\rightarrow \mathrm{H}^0(\mathcal{U}_2^*\otimes\mathcal{U}_2^*)$ and $S^2W^*\otimes S^{2,1}W^*\rightarrow\mathrm{H}^0(\mathcal{U}_2^*\otimes\mathcal{U}_1^*)$ are injective.
		For a point $r\in Y$, which corresponds to an orbit class of $2\times 3$ matrices under the action of $GL(2)\times GL(3)/{(\alpha.\mathrm{Id_2},\alpha.\mathrm{Id_3})}_{\alpha\in \mathds{C}^*}$, the fiber of $\mathcal{U}_2$ over $r$ is equal to the three dimensional space of maximal minors $H_r\subset S^2W$  and the fiber of $\mathcal{U}_1$ at $r$ is equal to the two dimensional canonical syzygy space $\Syz_r\subset S^{2,1}W$ at $r$.
		When $r$ varies in $Y$, one can easily verify that all those elements in $H_r\otimes H_r\subset S^2W\otimes S^2W$ together linearly generate the whole space $S^2W\otimes S^2W$, and all those $H_r\otimes \Syz_r$ together linearly generate the whole space $S^2W\otimes S^{2,1}W$, from which the desired inclusions of global sections follow immediately.
	\end{proof}
	
	\begin{Corollary}
		\label{corollary: RHom from L5 to universals^*(1)}
		We have $\mathrm{RHom}^{\bullet}(L_5[-2],\mathcal{U}_2^*(1))=\mathds{C}^{\oplus 3}[-1]$ and $\mathrm{RHom}^{\bullet}(L_5[-2],\mathcal{U}_1^*(1))=0$.\label{rhom mu}
	\end{Corollary}
	
	\begin{proof}
		We apply functors $\mathrm{RHom}^{\bullet}(\blank,\mathcal{U}_2^*(1))$ and $\mathrm{RHom}^{\bullet}(\blank,\mathcal{U}_1^*(1))$ to the short exact sequence \eqref{mu1}. Using the above Lemma, we know that the map $\mathrm{H}^{0}(\mathcal{U}_2^*)\otimes \mathrm{H}^{0}(\mathcal{U}_2^*)\cong\mathrm{Hom}(S^2W\otimes\mathcal{O}_Y(1),\mathcal{U}_2^*(1))\rightarrow \mathrm{Hom}(\mathcal{U}_2(1),\mathcal{U}_2^*(1))$
		$=\mathrm{H}^{0}(\mathcal{U}_2^*\otimes\mathcal{U}_2^*)$ and the map $\mathrm{H}^{0}(\mathcal{U}_2^*)\otimes \mathrm{H}^{0}(\mathcal{U}_1^*)\cong\mathrm{Hom}(S^2W\otimes\mathcal{O}_Y(1),\mathcal{U}_1^*(1))\rightarrow \mathrm{Hom}(\mathcal{U}_2(1),\mathcal{U}_1^*(1))=\mathrm{H}^{0}(\mathcal{U}_2^*\otimes\mathcal{U}_1^*)$ are both injective.
		The claims then follow from the associated long exact sequences in cohomology.
	\end{proof}
	
	\begin{notation}
		\label{notation: L4 L5 L6}
		The last \cref{corollary: RHom from L5 to universals^*(1)} gives $\dim \big( \mathrm{Ext}^1 \left( L_5[-2],\mathcal{U}_2^*(1)\right) \big)=3$.
		Thus there is a canonical element in $\mathrm{Ext}^1 \big(L_5[-2]$,
		$\mathcal{U}_2^*(1)\otimes (\mathrm{Ext}^1(L_5[-2],\mathcal{U}_2^*(1)))^* \big)=\mathrm{Ext}^1 \big( L_5[-2],\mathcal{U}_2^*(1)^{\oplus 3} \big)$, which is $SL(W)$-equivariant. As a result, we have a canonical $SL(W)$-equivariant extension of vector bundles:
		\begin{equation}
			0\rightarrow \mathcal{U}_2^*(1)^{\oplus 3}\rightarrow L_4[-2]\rightarrow L_5[-2]\rightarrow 0 \label{mu2}
		\end{equation}
		It is good to keep in mind that $\mathcal{U}_2^*(1)^{\oplus 3}=\mathcal{U}_2^*(1)\otimes (\mathrm{Ext}^1(L_5[-2],\mathcal{U}_2^*(1)))^*$ when considering the $SL(W)$ actions on $\mathcal{U}_2^*(1)^{\oplus 3}$.
		We denote $(L_6[-3])[3]$ by $L_6$, denote $(L_5[-2])[2]$ by $L_5$ and denote $(L_4[-2])[2]$ by $L_4$.
	\end{notation}
	
	\begin{Corollary}
		The object $L_5$ is the right mutation of $L_6$ across $\mathcal{O}_Y(1)$, while $L_4$ is the right mutation of $L_5$ across $\mathcal{U}_2^*(1)$.\label{whatever}
	\end{Corollary}
	
	Now let us perform a left mutation of $\mathcal{U}_1^*\otimes \mathcal{U}_2(2)$ across $\langle \mathfrak{sl}(\mathcal{U}_1^*)(2),\mathcal{O}_Y(2)\rangle$. We denote the result of this mutation by $L_2$.
	
	\begin{lemma}
		\label{lemma: RHoms into U1^*U2(2)}
		We have
		$\mathrm{RHom}^{\bullet}(\mathcal{U}_1^*\otimes \mathcal{U}_1(2),\mathcal{U}_1^*\otimes \mathcal{U}_2(2))=\mathds{C}^{\oplus 6}[0]$ and $\mathrm{RHom}^{\bullet}(\mathfrak{sl}(\mathcal{U}_1^*)(2),\mathcal{U}_1^*\otimes \mathcal{U}_2(2))=\mathds{C}^{\oplus 3}[0]$.
	\end{lemma}
	
	\begin{proof}
		By Teleman Quantization, we know that the higher cohomology of $\mathcal{U}_1^*\otimes \mathcal{U}_1\otimes \mathcal{U}_1^*\otimes \mathcal{U}_2$ is zero.
		Hirzebruch-Riemann-Roch calculation (using \cref{appendix: Sage code} for calculation) shows $\chi(\mathcal{U}_1^*\otimes \mathcal{U}_1\otimes \mathcal{U}_1^*\otimes \mathcal{U}_2)=6$. The second statement comes from the fact that $\mathcal{U}_1^*\otimes \mathcal{U}_1(2)=\mathfrak{sl}(\mathcal{U}_1^*)(2)\oplus \mathcal{O}_Y(2)$ and $\mathrm{H}^{\bullet}(\mathcal{U}_1^*\otimes \mathcal{U}_2)\cong W^*[0]$.
	\end{proof}
	
	\begin{proposition}
		The left mutation $L_2$ obtained by mutating $\mathcal{U}_1^*\otimes \mathcal{U}_2(2)$ across $\langle \mathfrak{sl}(\mathcal{U}_1^*)(2),\mathcal{O}_Y(2)\rangle$ is isomorphic to $\mathcal{U}_1^*(2)\otimes K[1]$, where $K$ is the kernel of the surjection map $\mathcal{U}_1\otimes W^*\rightarrow \mathcal{U}_2$.
		\begin{equation}
			0\rightarrow L_2[-1]\rightarrow \mathcal{U}_1^*(2)\otimes \mathcal{U}_1\otimes W^*\rightarrow \mathcal{U}_1^*(2)\otimes \mathcal{U}_2\rightarrow 0.
		\end{equation}\label{L_2}
	\end{proposition}
	
	\begin{proof}
		It follows from the stability condition of quiver representations that the universal morphism $\mathcal{U}_1\otimes W^*\rightarrow \mathcal{U}_2$ is surjective.
		
		To mutate $\mathcal{U}_1^*\otimes \mathcal{U}_2(2)$ across $\langle \mathfrak{sl}(\mathcal{U}_1^*)(2),\mathcal{O}_Y(2)\rangle$, first we recall that by \cref{muorthogonal}, the objects $\mathfrak{sl}(\mathcal{U}_1^*)(2)$ and $\mathcal{O}(2)$ are mutually orthogonal, so we have the exact triangle:
		$$L_2[-1]\rightarrow (\mathfrak{sl}(\mathcal{U}_1^*)(2)\otimes \mathrm{Hom}_1)\bigoplus(\mathcal{O}_Y(2)\otimes\mathrm{Hom}_2) \rightarrow \mathcal{U}_1^*\otimes \mathcal{U}_2(2)\rightarrow L_2,$$
		where we denote $\mathrm{RHom}^{\bullet}(\mathfrak{sl}(\mathcal{U}_1^*)(2),\mathcal{U}_1^*\otimes \mathcal{U}_2(2))$ by $\mathrm{Hom}_1$ and $\mathrm{RHom}^{\bullet}(\mathcal{O}_Y(2),\mathcal{U}_1^*\otimes \mathcal{U}_2(2))$ by $\mathrm{Hom}_2$.
		We have seen that $\mathrm{RHom}^{\bullet}(\mathcal{O}_Y(2),\mathcal{U}_1^*\otimes \mathcal{U}_2(2))\cong W^*[0]$,
		and this isomorphism is obtained via $\mathcal{O}_Y(2)\otimes W^*\hookrightarrow \mathcal{U}_1^*\otimes \mathcal{U}_1(2)\otimes W^*\rightarrow\mathcal{U}_1^*\otimes \mathcal{U}_2(2)$,
		where the latter morphism is the twist by $\mathcal{U}_1^*(2)$ of the standard morphism $\mathcal{U}_1\otimes W^*\rightarrow \mathcal{U}_2$. 
		
		We also have a family of morphisms from $\mathfrak{sl}(\mathcal{U}_1^*)(2)$ to $\mathcal{U}_1^*\otimes \mathcal{U}_2(2)$ given by the nonzero map $\mathfrak{sl}(\mathcal{U}_1^*)(2)\otimes W^*\hookrightarrow \mathcal{U}_1^*\otimes \mathcal{U}_1(2)\otimes W^*\rightarrow\mathcal{U}_1^*\otimes \mathcal{U}_2(2)$, where the latter morphism is again the twisted standard morphism. Since $\dim(\Hom(\mathfrak{sl}(\mathcal{U}_1^*)(2),\mathcal{U}_1^*\otimes \mathcal{U}_2(2)))=3$ by \cref{lemma: RHoms into U1^*U2(2)} and $W$ is an irreducible $SL(W)$-module, these give us all the morphisms from $\mathfrak{sl}(\mathcal{U}_1^*)(2)$ to $\mathcal{U}_1^*\otimes \mathcal{U}_2(2)$.
		
		As a result, the map $(\mathfrak{sl}(\mathcal{U}_1^*)(2)\otimes \mathrm{Hom}_1)\bigoplus(\mathcal{O}_Y(2)\otimes\mathrm{Hom}_2) \rightarrow \mathcal{U}_1^*\otimes \mathcal{U}_2(2)$ is exactly $\mathcal{U}_1^*\otimes \mathcal{U}_1(2)\otimes W^*\rightarrow\mathcal{U}_1^*\otimes \mathcal{U}_2(2)$, which is surjective and the kernel is $\mathcal{U}_1^*(2)\otimes K$.
	\end{proof}
	
	Still we want to do one more right mutation, i.e. the right mutation of $L_4$ across $\mathcal{U}_1^*(1)$. For this we apply
	$\mathrm{RHom}^{\bullet}(\blank,\mathcal{U}_1^*(1))$ to the short exact sequence \eqref{mu2} and get 
	$\mathrm{RHom}^{\bullet}(L_4[-2],\mathcal{U}_1^*(1))$
	$=\mathrm{RHom}^{\bullet}(\mathcal{U}_2^*(1)^{\oplus 3},\mathcal{U}_1^*(1))$
	$=(W^*)^{\oplus 3}$, because we know that $\mathrm{RHom}^{\bullet}(L_5[-2],\mathcal{U}_1^*(1))=0$ by \cref{rhom mu}.
	As a result, we have a canonical map $\alpha: L_4[-2]\rightarrow (\mathcal{U}_1^*(1)\otimes W)^{\oplus 3}$.
	We denote the cokernel of $\alpha$ by $L_3[-1]$, and thus obtain the following short exact sequence: 
	\begin{equation}
		L_4[-2]\xrightarrow {\alpha}(\mathcal{U}_1^*(1)\otimes W)^{\oplus 3}\rightarrow L_3[-1]\rightarrow 0. \label{mu3}
	\end{equation}
	In the next subsection we will prove that $\alpha$ is a bundle morphism of maximal rank everywhere (note that $\mathrm{rank}(L_4[-2])=12$), which implies in particular that $L_3[-1]$ is a vector bundle and that the right mutation of $L_4$ across $\mathcal{U}_1^*(1)$ is equal to $L_3$. Moreover, in the next subsection, we will prove that $L_2=L_3$, which is not surprising because of the absence of the $\mathrm{ch}(\mathcal{U}_2(2))$ term in the third equality of \cref{ch eq}; in other words, the right mutation of $L_3$ across $\mathcal{U}_2(2)$ is expected to be $L_3$ itself.

	\subsection{Completion of Mutations}
	\label{subsection: Completion of mutations}
	The main aim of this subsection is to prove the vector bundle isomorphism $L_2[-1]\cong L_3[-1]$ in \cref{same tri subcategory}. We will first prove that these two vector bundles have the same restriction on a special subvariety of $Y$ (\cref{same restriction}) and then conclude using the speciality of this subvariety.
	
	Recall that $Y$ is isomorphic to the zero locus of a general global section of $Q^*(1)$ over $\mathrm{Gr}(2,S^{2,1}W)$, where $Q$ is the universal quotient bundle over $\mathrm{Gr}(2,S^{2,1}W)$. Moreover, the restriction to $Y$ of 
	the universal subbundle is isomorphic to
	$\mathcal{U}_1$.
	If we take the zero locus of a general global section of $(\mathcal{U}_1^*)^{\oplus 2}$, we get a (possibly disconnected) smooth Fano surface $S$, with $K_S=\mathcal{O}_Y(-1)|_S$ and $K_S^2=(c_1^2.d_2^2)=6$.
	We don't need such a general construction of 2-dimensional subvarieties.
	Instead, we will construct a surface contained in $Y$ directly,  which is isomorphic to the blow up of $\mathds P^2$ along 3 points.
	
	\begin{construction}
		\label{construction: BlP2}
		We fix a basis of $W$ to be $\{x,y,z\}$ and the dual basis of $W^*$ to be $\{X,Y,Z\}$.
		We consider the rational map $f:\mathds{P}^2\dashrightarrow Y$ sending $(a:b:c)\in \mathds{P}^2$ to $\begin{pmatrix}
			x & y& z\\
			ay &bz &cx
		\end{pmatrix}\in Y$, viewing elements of $Y$ as equivalence classes of $2\times 3$ matrices as in \cref{construction: Y into Gr(3 S^2 W)}.
		This rational map is undefined at three points $(1:0:0),(0:1:0),(0:0:1)$.
		We consider the blow up of $\mathds{P}^2$ along these three points and simply denote this blow up by $Bl(\mathds{P}^2)$.
		One can check that we have an induced regular map $Bl(\mathds{P}^2)\rightarrow Y$ which is an embedding and over the point $(1:0:0)$ sends $[b':c']\in\mathbb P^1$ to $\begin{pmatrix}
			0 & y& z\\
			y &b'z &c'x
		\end{pmatrix}\in Y$, with analogous formulas over the other two points.
		We use $\mathcal{O}(H)$ to denote the pullback of $\mathcal{O}_{\mathds{P}^2}(1)$, and denote the exceptional divisors by $E_1,E_2,E_3$ respectively.
		
		Over the general point of $Bl(\mathds{P}^2)$, the composition $Bl(\mathds{P}^2)\rightarrow Y\hookrightarrow \mathrm{Gr}(3,S^2W)$ sends $(a:b:c)$ to $\langle bz^2-cxy, cx^2-ayz , bxz-ay^2\rangle\subset\langle z^2,xy\rangle\oplus \langle x^2,yz\rangle\oplus \langle xz,y^2\rangle$, which indicates that $\mathcal{U}_2^*|_{Bl(\mathds{P}^2)}\cong \mathcal{O}(H-E_1)\oplus\mathcal{O}(H-E_2)\oplus\mathcal{O}(H-E_3) $.
		
		Similarly, the composition $Bl(\mathds{P}^2)\rightarrow Y\hookrightarrow \mathrm{Gr}(2,S^{2,1}W)$ sends $(a:b:c)$ to $\langle a(yz\otimes y-y^2\otimes z)+b(xz\otimes z-z^2\otimes x)+c(xy\otimes x-x^2\otimes y),ab(yz\otimes z-z^2\otimes y)+bc(xz\otimes x-x^2\otimes z)+ac(xy\otimes y-y^2\otimes x)\rangle$, which indicates that $\mathcal{U}_1^*|_{Bl(\mathds{P}^2)}\cong \mathcal{O}(H)\oplus\mathcal{O}(2H-E_1-E_2-E_3) $.
		Moreover, one can see from this expression that $Bl(\mathds{P}^2)$ is actually contained in $\mathrm{Gr}(2,6)\bigcap Y\subset \mathrm{Gr}(2,8)=\mathrm{Gr}(2,S^{2,1}W)$, which means that it is the zero locus of a global section of $(\mathcal{U}_1^*)^{\oplus 2}$.
	\end{construction}
	
	
	
	Now our main task is to understand the restriction of the mutation process described in the last subsection to the subvariety $Bl(\mathds{P}^2)$.
	
	\begin{lemma}
		The restriction of $L_5[-2]$ to $Bl(\mathds{P}^2)$ is isomorphic to $(\mathcal{O}(H-E_1)\oplus\mathcal{O}(H-E_2)\oplus\mathcal{O}(H-E_3))\otimes\mathcal{O}_Y(1)|_{Bl(\mathds{P}^2)}$. The restriction of $L_2[-1]$ to $Bl(\mathds{P}^2)$ is isomorphic to $[\mathcal{O}(2H-E_1)\oplus\mathcal{O}(2H-E_2)\oplus\mathcal{O}(2H-E_3)]\otimes\mathcal{O}_Y(1)|_{Bl(\mathds{P}^2)}\bigoplus [\mathcal{O}(3H-2E_1-E_2-E_3)\oplus\mathcal{O}(3H-E_1-2E_2-E_3)\oplus\mathcal{O}(3H-E_1-E_2-2E_3)]\otimes\mathcal{O}_Y(1)|_{Bl(\mathds{P}^2)}$.
		The restriction of $K^*(1) = (\mathcal{U}_1^*(1)\otimes W)/\mathcal{U}_2^*(1)$ is isomorphic to $(\mathcal{O}(2H-E_2-E_3)\oplus\mathcal{O}(2H-E_1-E_3)\oplus\mathcal{O}(2H-E_1-E_2))\otimes\mathcal{O}_Y(1)|_{Bl(\mathds{P}^2)}$. \label{res}
	\end{lemma}
	\begin{proof}
		First we deal with the restriction of the map $\mathcal{U}_2(1)\hookrightarrow S^2W\otimes \mathcal{O}_Y(1)$ which is induced from the map $Y\rightarrow\mathrm{Gr}(3,S^2W)$.
		By the argument in \cref{construction: BlP2}, we know that the restriction of this inclusion is equal to the direct sum of three inclusions: $\mathcal{O}(E_1-H)\otimes \mathcal{O}_Y(1)|_{Bl(\mathds{P}^2)}\hookrightarrow \langle z^2,xy\rangle\otimes \mathcal{O}_Y(1)|_{Bl(\mathds{P}^2)}$,  $\mathcal{O}(E_2-H)\otimes \mathcal{O}_Y(1)|_{Bl(\mathds{P}^2)}\hookrightarrow \langle x^2,yz\rangle\otimes \mathcal{O}_Y(1)|_{Bl(\mathds{P}^2)}$, $\mathcal{O}(E_3-H)\otimes \mathcal{O}_Y(1)|_{Bl(\mathds{P}^2)}\hookrightarrow \langle y^2,xz\rangle\otimes \mathcal{O}_Y(1)|_{Bl(\mathds{P}^2)}$.
		As a result, the quotient bundle $L_5[-2]|_{Bl(\mathds{P}^2)}$ of the restriction of the inclusion $\mathcal{U}_2(1)\hookrightarrow S^2W\otimes \mathcal{O}_Y(1)$ is equal to the direct sum of three line bundles, and an easy calculation of the first chern classes shows 
		that the restriction of $L_5[-2]$ to $Bl(\mathds{P}^2)$ is equal to $(\mathcal{O}(H-E_1)\oplus\mathcal{O}(H-E_2)\oplus\mathcal{O}(H-E_3))\otimes\mathcal{O}_Y(1)|_{Bl(\mathds{P}^2)}$.
		
		Now we consider the restriction of the surjection $\mathcal{U}_1\otimes W^*\rightarrow \mathcal{U}_2$, which essentially comes from the dual of the fiberwise inclusion $\Syz_r\subset H_r\otimes W$ (see \cref{construction: Y into Gr(3 S^2 W)} and \cref{construction: Y into Gr(2 S^21 W)} for the notation).
		Similarly, the restriction of this surjection is the direct sum of the following three surjections:
		\begin{align*}
			(\mathcal{O}(-H)\otimes \langle X\rangle)\oplus (\mathcal{O}(-2H+E_1+E_2+E_3)\otimes \langle Y\rangle) &\rightarrow \mathcal{O}(E_1-H),\\
			(\mathcal{O}(-H)\otimes \langle Y\rangle)\oplus (\mathcal{O}(-2H+E_1+E_2+E_3)\otimes \langle Z\rangle) &\rightarrow \mathcal{O}(E_2-H),\\
			(\mathcal{O}(-H)\otimes \langle Z\rangle)\oplus (\mathcal{O}(-2H+E_1+E_2+E_3)\otimes \langle X\rangle) &\rightarrow \mathcal{O}(E_3-H).
		\end{align*}
		As a result, the kernel $K|_{Bl(\mathds{P}^2)}$ of the restriction of the surjection is isomorphic to the direct sum of three line bundles $\mathcal{O}(-2H+E_2+E_3)\oplus\mathcal{O}(-2H+E_1+E_3)\oplus\mathcal{O}(-2H+E_1+E_2)$. The third statement follows by taking dual bundles. The second statement follows from the fact that $L_2[-1]=\mathcal{U}_1^*(2)\otimes K$ (\cref{L_2}) and $\mathcal{O}_Y(1)|_{Bl(\mathds{P}^2)}=\mathrm{det}(\mathcal{U}_1^*)|_{Bl(\mathds{P}^2)}=\mathcal{O}(3H-E_1-E_2-E_3)$.
	\end{proof}
	
	For future use, we record the splitting of $\mathcal{U}_2^*\rightarrow\mathcal{U}_1^*\otimes W$ into direct sums together with the quotient maps in the following diagram:
	\[\scalebox{0.9}
	{\begin{tikzcd}[column sep=3em, row sep=0.5em, ampersand replacement=\&] 
			\mathcal{U}_2^*|_{Bl(\mathds{P}^2)}\ar[rr]\&\&(\mathcal{U}_1^*\otimes W)|_{Bl(\mathds{P}^2)}\\ \\
			\&\&\mathcal{O}(H)\otimes\langle x\rangle\ar[drr,"-a"]\&\&\\
			\mathcal{O}(H-E_1) \ar[urr,"1"]\ar[drr,"a"]\&\&\oplus\&\&\mathcal{O}(2H-E_2-E_3)\\
			\&\&\mathcal{O}(2H-E_1-E_2-E_3)\otimes \langle y\rangle\ar[urr,"1"]\\
			\oplus\&\&\oplus\&\&\oplus\\
			\&\&\mathcal{O}(H)\otimes\langle y\rangle\ar[drr,"-b"]\\
			\mathcal{O}(H-E_2) \ar[urr,"1"]\ar[drr,"b"] \&\&\oplus\&\&\mathcal{O}(2H-E_1-E_3)\&\&(**)\\
			\&\&\mathcal{O}(2H-E_1-E_2-E_3)\otimes \langle z\rangle\ar[urr,"1"]\\
			\oplus\&\&\oplus\&\&\oplus\\
			\&\&\mathcal{O}(H)\otimes\langle z\rangle\ar[drr,"-c"]\\
			\mathcal{O}(H-E_3) \ar[urr,"1"]\ar[drr,"c"] \&\&\oplus\&\&\mathcal{O}(2H-E_1-E_2)\\
			\&\&\mathcal{O}(2H-E_1-E_2-E_3)\otimes \langle x\rangle\ar[urr,"1"]
	\end{tikzcd}}
	\]
	where the symbols on the arrows describe the morphisms between line bundles over $Bl(\mathds{P}^2)-E_1-E_2-E_3$ in terms of the projective coordinates $(a:b:c)$.
	
	\begin{proposition}
		The map $\alpha: L_4[-2]\rightarrow (\mathcal{U}_1^*(1)\otimes W)^{\oplus 3}$ appearing in \eqref{mu3} is a bundle morphism of maximal rank everywhere. In particular, $\mathrm{Coker}(\alpha)=L_3[-1]$ is a vector bundle and the right mutation of $L_4$ across $\mathcal{U}_1^*(1)$ is isomorphic to $L_3$.
		\label{big proof}
	\end{proposition}
	
	\begin{proof}
		Recall that we have the short exact sequence: 
		$$0\rightarrow \mathcal{U}_2^*(1)^{\oplus 3}\rightarrow L_4[-2]\rightarrow L_5[-2]\rightarrow 0. $$
		We view $\mathcal{U}_2^*(1)^{\oplus 3}$ as a subbundle of $L_4[-2]$.
		In the definition of $\alpha: L_4[-2]\rightarrow \mathcal{U}_1^*(1)\otimes W^{\oplus 3}$, we used the fact that $\mathrm{RHom}^{\bullet}(L_4[-2],\mathcal{U}_1^*(1))$ is isomorphic to $\mathrm{RHom}^{\bullet}(\mathcal{U}_2^*(1)^{\oplus 3},\mathcal{U}_1^*(1))= (W^*)^{\oplus 3}$, which is observed via the restriction of morphisms to the subbundle $\mathcal{U}_2^*(1)^{\oplus 3}\subset L_4[-2]$, i.e. via applying $\mathrm{RHom}^{\bullet}(\blank,\mathcal{U}_1^*(1))$ to the short exact sequence above.
		Consequently, the morphism $\alpha$ is the unique morphism such that the restriction of $\alpha$ to the subbundle $\mathcal{U}_2^*(1)^{\oplus 3}\subset L_4[-2]$ is just equal to the direct sum of three copies of the twisted adjunction representation map $\mathcal{U}_2^*(1)\rightarrow\mathcal{U}_1^*(1)\otimes W$.
		It follows from the stability condition (see \cref{remark: stability condition explained}) that $\mathcal{U}_2^*(1)\rightarrow\mathcal{U}_1^*(1)\otimes W$ is of constant maximal rank, thus an injective morphism between vector bundles.
		We take the quotient bundle $(K^*(1))^{\oplus 3} \cong (\mathcal{U}_1^*(1)\otimes W)^{\oplus 3}/\mathcal{U}_2^*(1)^{\oplus 3}$ and consider the induced map
		$$\beta: L_5[-2]=L_4[-2]/\mathcal{U}_2^*(1)^{\oplus 3}\rightarrow 
		K^*(1)^{\oplus 3}.
		$$
		It is enough to prove that $\beta$ is of maximal rank everywhere.
		
		Based on the fact that $\mathrm{RHom}^{\bullet}(L_5[-2],\mathcal{U}_1^*(1))=0$ (\cref{rhom mu}), we know that
		\[
		\mathrm{RHom}^{\bullet}(L_5[-2],
		K^*(1)
		)=\mathrm{RHom}^{\bullet}(L_5[-2],\mathcal{U}_2^*(1)[1])=\mathds{C}^{\oplus 3}[0].
		\]
		again by \cref{rhom mu}. Let $\mathrm{Hom}_3$ denote $\mathrm{Hom}(L_5[-2],
		K^*(1)
		)$.
		We claim that $\beta$ must be equal to the coevaluation map: 
		$$L_5[-2]\rightarrow 
		K^*(1)
		\otimes [\mathrm{Hom}_3]^*=
		K^*(1)^{\oplus 3}
		.$$
		To see this, we use the comparison axiom for the following two exact triangles:
		\[\begin{tikzcd}
			\mathcal{U}_2^*(1)^{\oplus 3}\arrow{r} \arrow[swap]{d}{\mathrm{Id}} & L_4[-2] \arrow{d}{\alpha} \arrow{r}&L_5[-2]\arrow{r}{\gamma}\arrow[dashed]{d}&\mathcal{U}_2^*(1)^{\oplus 3}[1]\arrow[swap]{d}{\mathrm{Id}}\\
			\mathcal{U}_2^*(1)^{\oplus 3}\arrow{r} & (\mathcal{U}_1^*(1)\otimes W)^{\oplus 3}\arrow{r}&
			K^*(1)^{\oplus 3}
			\arrow{r}\arrow{r}{\gamma'}&\mathcal{U}_2^*(1)^{\oplus 3}[1]
		\end{tikzcd}
		\]
		There should be a dashed arrow making the diagram commute. Since  $\mathrm{RHom}^{\bullet}(L_5[-2],\mathcal{U}_1^*(1))=0$, there can be only one dashed arrow making the rightmost square commute, which is equivalent to saying that there can be only one dashed arrow whose composition with $\gamma'$ is equal to $\gamma$.
		Recall that $\gamma$ was defined using the canonical $SL(W)$-equivariant element in the extension group $\mathrm{Ext}^1(L_5[-2],\mathcal{U}_2^*(1)\otimes [\mathrm{Ext}^1(L_5[-2],\mathcal{U}_2^*(1))]^*)=\mathrm{Ext}^1(L_5[-2],\mathcal{U}_2^*(1)^{\oplus 3})$ 
		in \cref{notation: L4 L5 L6}.
		Since $\beta$ obviously makes the above diagram commute and the claimed morphism satisfies the characterizing property of the dashed arrow, we get to prove our claim.
		
		Now we need to understand the three dimensional space $\mathrm{Hom}_3 = \mathrm{Hom}(L_5[-2],
		K^*(1))$ in order to prove that $\beta$ is of maximal rank everywhere. Recall that $L_5[-2]$ can be put in a single short exact sequence \eqref{mu1}:
		\begin{equation}
			0\rightarrow \mathcal{U}_2(1)=L_6[-3]\rightarrow S^2W\otimes\mathcal{O}_Y(1)\rightarrow L_5[-2]\rightarrow 0,
		\end{equation}
		thus a morphism 
		$L_5[-2] \to K^*(1)$
		is a morphism
		$S^2W\otimes \mathcal{O}_Y(1) \to K^*(1)$ such that the restriction to the subbundle $\mathcal{U}_2(1)$ is zero.
		We know that the higher cohomology of $\mathcal{U}_2^*$ and $\mathcal{U}_1^*$ vanish, hence $\mathrm{RHom}^{\bullet}(\mathcal{O}_Y(1),
		K^*(1)
		)=\mathrm{H}^0(
		K^*
		)=(S^{2,1}W^*\otimes W)/S^2W^* [0] \ 
		(\cong S^{3,1}W[0] \oplus W[0]) \ 
		\cong S^2W\otimes W^*[0]$ as $SL(W)$-modules.
		
		The above $SL(W)$-module isomorphism is not unique since this $SL(W)$-module is the direct sum of two different irreducible submodules.
		We will construct below an explicit map $A+2B: S^2W\otimes W^*\rightarrow (S^{2,1}W^*\otimes W)/S^2W^*$, thus a map $W^*\rightarrow \mathrm{Hom}(S^2W,(S^{2,1}W^*\otimes W)/S^2W^*)$,
		such that whenever we fix an element in $W^*$,
		it gives a morphism $S^2W\rightarrow (S^{2,1}W^*\otimes W)/S^2W^*=\mathrm{Hom}(\mathcal{O}_Y(1),
		K^*(1)
		)$,
		which can be viewed as a morphism $S^2W\otimes\mathcal{O}_Y(1)\rightarrow
		K^*(1)
		$ and will actually be (i.e. descend to) a morphism from $L_5[-2]$ to 
		$
		K^*(1)
		$.
		This will imply that $\mathrm{Hom}(L_5[-2],
		K^*(1)
		)\cong W^*$ as a $SL(W)$-module by the dimension counting above.
		
		Recall that $S^{2,1}W\cong \mathfrak{sl}(W)$ as $SL(W)$-representations, with an explicit isomoprhism given in \cref{remark: S21 as slW}.
		In the following, we will also use the dual isomorphism $S^{2,1}W^*\cong \mathfrak{sl}(W)$, noticing that $\mathfrak{sl}(W)$ is self-dual.
		
		We define a first map $A:S^2W\otimes W^*\rightarrow S^{2,1}W\otimes W$ (we use $v,w$ to denote elements in $W$ and $f,g$ to denote elements in $W^*$):
		$$S^2W\otimes W^*\rightarrow W\rightarrow W\otimes W^*\otimes W\rightarrow \mathfrak{sl}(W)\otimes W=S^{2,1}W^*\otimes W$$
		$$(v.w)\otimes f\mapsto f(v).w+f(w).v\mapsto (f(v).w+f(w).v)\otimes (X\otimes x+Y\otimes y+Z\otimes z)\mapsto(\mathrm{projection})$$
		where in the last step the projection $(W\otimes W^*)\otimes W\rightarrow \mathfrak{sl}(W)\otimes W$ is the $SL(W)$-equivariant projection $[W\otimes W^*\rightarrow \mathfrak{sl}(W)]\otimes W$.
		
		We define a second map $B:S^2W\otimes W^*\rightarrow S^{2,1}W\otimes W$:
		$$S^2W\otimes W^*\hookrightarrow W\otimes W\otimes W^*\cong(W\otimes W^*)\otimes W\rightarrow \mathfrak{sl}(W)\otimes W=S^{2,1}W^*\otimes W$$
		$$(v.w)\otimes f \mapsto(v\otimes w+w\otimes v)\otimes f\mapsto v\otimes f \otimes w+w\otimes f \otimes v\mapsto(\mathrm{projection})$$
		
		We claim that the composition of $A+2B:S^2W\otimes W^*\rightarrow S^{2,1}W^*\otimes W$ with the quotient map $S^{2,1}W^*\otimes W\rightarrow S^{2,1}W^*\otimes W/S^2W^*$ gives the desired isomorphism $\mathrm{Hom}(L_5[-2],(\mathcal{U}_1^*(1)\otimes W)/\mathcal{U}_2^*(1))\cong W^*$. To verify this, we only need to verify that the induced morphism $S^2W\otimes\mathcal{O}_Y(1)\otimes W^*\rightarrow(\mathcal{U}_1^*(1)\otimes W)/\mathcal{U}_2^*(1)=K^*(1)$ becomes zero when restricted to the subbundle $\mathcal{U}_2(1)\otimes W^*\subset S^2W\otimes\mathcal{O}_Y(1)\otimes W^*$. We only need to verify the restricted map $\mathcal{U}_2(1)\otimes W^*\rightarrow (\mathcal{U}_1^*(1)\otimes W)/\mathcal{U}_2^*(1)$ is zero at general points, and using $SL(W)$-action we only need to verify that it is zero on the fiber of one arbitrary element in the open $SL(W)$ orbit $\mathcal{O}_6$.
		We will postpone doing this until \cref{comp zero}.
		
		Now that we have a good understanding of the isomorphism $\mathrm{Hom}(L_5[-2],(\mathcal{U}_1^*(1)\otimes W)/\mathcal{U}_2^*(1))\cong W^*$, we can try to prove $\beta:L_5[-2]\rightarrow[(\mathcal{U}_1^*(1)\otimes W)/\mathcal{U}_2^*(1)]\otimes W=(\mathcal{U}_1^*(1)\otimes W)^{\oplus 3}/\mathcal{U}_2^*(1)^{\oplus 3}=K^*(1)^{\oplus 3}$ is of maximal rank everywhere. This is again a $SL(W)$-equivariant map and we only need to verify this fact on the fiber of arbitrary elements in the two smallest closed orbits. Our favourite subvariety $Bl(\mathds{P}^2)$ contains elements in the biggest orbit $\mathcal{O}_6$ and the two smallest orbits $\mathcal{O}_2,\mathcal{O}_2'$, so again we will postpone the calculation until \cref{final cal}.
	\end{proof}
	
	We do some calculations in the following paragraphs:
	
	We first calculate the image of $z^2\otimes W^*\in S^2W\otimes W^*$ in $S^{2,1}W^*\otimes W\cong\mathfrak{sl}(W)\otimes W$ under the map $A+2B$. 
	$$(A+2B):(z^2\otimes X)\rightarrow 4 (z\otimes X)\otimes z=4E_{31}\otimes x$$
	$$(A+2B)(z^2\otimes Y)=4 E_{32}\otimes z$$
	$$(A+2B)(z^2\otimes Z)=2 E_{31}\otimes x+ 2E_{32}\otimes y+2(2E_{33}-E_{11}-E_{22})\otimes z$$
	We also calculate the image of $xy\otimes W^*\in S^2W\otimes W^*$ in $S^{2,1}W^*\otimes W\cong\mathfrak{sl}(W)\otimes W$ under the map $A+2B$. 
	$$(A+2B)(xy\otimes X)=3E_{21}\otimes x+E_{23}\otimes z+\frac{1}{3}(2E_{22}-E_{11}-E_{33})\otimes y+\frac{2}{3}(2E_{11}-E_{22}-E_{33})\otimes y$$
	$$(A+2B)(xy\otimes Y)=3E_{12}\otimes y+E_{13}\otimes z+\frac{1}{3}(2E_{11}-E_{22}-E_{33})\otimes x+\frac{2}{3}(2E_{22}-E_{11}-E_{33})\otimes x$$
	$$(A+2B)(xy\otimes Z)=2 E_{13}\otimes y+ 2E_{23}\otimes x$$
	Let $(a,b,c)\in \mathds{C}^3$. We define $$\mathrm{Syz}_1=a(yz\otimes y-y^2\otimes z)+b(xz\otimes z-z^2\otimes x)+c(xy\otimes x-x^2\otimes y)$$
	$$\mathrm{Syz}_2=ab(yz\otimes z-z^2\otimes y)+bc(xz\otimes x-x^2\otimes z)+ac(xy\otimes y-y^2\otimes x).$$ Under the isomorphism $S^{2,1}W\cong \mathfrak{sl}(W)$ we know that $\mathrm{Syz}_1=-aE_{21}-bE_{32}-cE_{13}$ and $\mathrm{Syz}_2=acE_{23}+bcE_{12}+abE_{31}$. (For further calculation, notice that the $SL(W)$-equivariant dual basis of $E_{i,j},i\neq j$ is equal to $E_{j,i}$.)
	\[\begin{tikzcd}
		0\ar{r}&\mathcal{U}_2(1)\otimes W^*\arrow{r}  & (S^2W\otimes W^*)\otimes \mathcal{O}_Y(1) \arrow{d}{F} \ar{dl}{A+2B}\arrow{r}& L_5[-2]\otimes W^*\arrow[dashed]{ddl}\ar{r}&0\\
		&(S^{2,1}W^*\otimes W)\otimes \mathcal{O}_Y(1)\arrow[r,two heads] \arrow[d,two heads] &\mathcal{U}_1^*(1)\otimes W\arrow[d,two heads,"q"]&&\\
		&((S^{2,1}W^*\otimes W)/S^2W^*)\otimes \mathcal{O}_Y(1)\arrow[r,two heads] &(\mathcal{U}_1^*(1)\otimes W)/\mathcal{U}_2^*(1)=K^*(1) & &
	\end{tikzcd}
	\]
	In the above commutative diagram, the exact sequence on the first row is the twisted version of (\ref{mu1}), and the morphism $A+2B$ is the twisted version of $A+2B: S^2W\otimes W^*\rightarrow S^{2,1}W^*\otimes W$. The horizontal arrow on the second row comes from the fact that the bundle $\mathcal{U}_1^*\otimes W$ is generated by global sections $\mathrm{H}^0(\mathcal{U}_1^*\otimes W)\cong S^{2,1}W^*\otimes W$ and the horizontal arrow on the third row comes from the same reasoning. The dashed arrow is essentially the morphism $\beta$ in its adjunction form. We aim to describe the restriction of the map $F(-1)=F\otimes \mathcal{O}_Y(-1): S^2W\otimes W^*\rightarrow \mathcal{U}_1^*\otimes W$: 
	
	We consider the image of $z^2\otimes X\in S^2W\otimes W^* $ via $F(-1)|_{Bl(\mathds{P}^2)}$ for example. This image of $z^2\otimes X$ via $A+2B$ is an element in $S^{2,1}W^*\otimes W$ and thus can be viewed as a linear form on $\langle\mathrm{Syz}_1\otimes X,\mathrm{Syz}_1\otimes Y,\mathrm{Syz}_1\otimes Z,\mathrm{Syz}_2\otimes X,\mathrm{Syz}_2\otimes Y,\mathrm{Syz}_2\otimes Z\rangle\subset S^{2,1}W\otimes W^*$ i.e. a linear form on the fiber of $\mathcal{U}_1\otimes W^*|_{Bl(\mathds{P}^2)}$ at the point $(a:b:c)\in \mathds{P}^2\backslash3\ points$ and thus can be viewed as a global section of $\mathcal{U}_1^*\otimes W|_{Bl(\mathds{P}^2)}$. 
	With this interpretation, we can calculate the map $F(-1)|_{Bl(\mathds{P}^2)}$ (and thus $F|_{Bl(\mathds{P}^2)}$) explicitly and find that the action of $F|_{Bl(\mathds{P}^2)}$ on the subbundle $(\langle z^2,xy\rangle\otimes \mathcal{O}_Y(1))\otimes W^*|_{Bl(\mathds{P}^2)}\subset (S^2W\otimes W^*)\otimes \mathcal{O}_Y(1)|_{Bl(\mathds{P}^2)}$ is the direct sum of three morphisms. The result of the calculation is summarized in the second and third column of the following diagram: 
	\[
	\adjustbox{max width=\textwidth}{
		\begin{tikzcd}[column sep=4em , row sep=0.5em] 
			&&\langle z^2,xy\rangle\otimes \mathcal{O}_Y(1)\otimes W^*|_{Bl(\mathds{P}^2)} \ar[rr,"F|_{Bl(\mathds{P}^2)}"]&&\mathcal{U}_1^*(1)\otimes W|_{Bl(\mathds{P}^2)}\\
			\\
			&&z^2\mathcal{O}(1)X\ar[rr,"-4c"]\ar[ddrr,"0" near start]&&\mathcal{O}(H)(1)z\\
			\mathcal{O}(E_1-H)(1)X\ar[urr, "b"]\ar[drr, "-c"]&&\oplus &&\oplus\\
			&&xy\mathcal{O}(1)X\ar[rr,"3bc"]\ar[uurr,"-b" near end]&&\mathcal{O}(2H-E_1-E_2-E_3)(1)x\\
			\oplus &&\oplus&&\oplus\\
			&&z^2\mathcal{O}(1)Y\ar[rr,"0"]\ar[ddrr,"4ac" near start]&&\mathcal{O}(H)(1)y\\
			\mathcal{O}(E_1-H)(1)Y\ar[urr, "b"]\ar[drr, "-c"]&&\oplus &&\oplus\\
			&&xy\mathcal{O}(1)Y\ar[rr,"ab"]\ar[uurr,"-3a" near end]&&\mathcal{O}(2H-E_1-E_2-E_3)(1)z\\
			\oplus &&\oplus&&\oplus\\
			&&z^2\mathcal{O}(1)Z\ar[rr,"-2c"]\ar[ddrr,"2ac" near start]&&\mathcal{O}(H)(1)x\\
			\mathcal{O}(E_1-H)(1)Z\ar[urr, "b"]\ar[drr, "-c"]&&\oplus &&\oplus\\
			&&xy\mathcal{O}(1)Z\ar[rr,"2ab"]\ar[uurr,"-2b" near end]&&\mathcal{O}(2H-E_1-E_2-E_3)(1)y
		\end{tikzcd}
	}
	\]
	
	(Notation: we use $\mathcal{O}(E_1-H)(1)X$ to denote $\mathcal{O}(E_1-H)\otimes\mathcal{O}_Y(1)|_{Bl_{\mathds{P}^2}}\otimes \langle X\rangle\subset \mathcal{U}_2(1)|_{Bl_{\mathds{P}^2}}\otimes W^*$, we use $z^2\mathcal{O}(1)X$ to denote $\langle z^2\rangle \otimes\mathcal{O}_Y(1)|_{Bl(\mathds{P}^2)}\otimes\langle X\rangle\subset S^2W\otimes \mathcal{O}_Y(1)\otimes W^*|_{Bl(\mathds{P}^2)}$, we use $\mathcal{O}(H)(1)x$ to denote $\mathcal{O}(H)\otimes \mathcal{O}_Y(1)|_{Bl(\mathds{P}^2)}\otimes\langle x\rangle\subset(\mathcal{U}_1^*(1)\otimes W)|_{Bl(\mathds{P}^2)}$ and we use $\mathcal{O}(2H-E_1-E_2-E_3)(1)x$ to denote $\mathcal{O}(2H-E_1-E_2-E_3)\otimes \mathcal{O}_Y(1)|_{Bl(\mathds{P}^2)}\otimes\langle x\rangle\subset(\mathcal{U}_1^*(1)\otimes W)|_{Bl(\mathds{P}^2)}$ for shortness of notation.)
	
	From the first column of the above diagram, one can read the component morphism $F_1$ of the restriction of ``the composition of morphism $\mathcal{U}_2(1)\otimes W^*\hookrightarrow S^2W\otimes\mathcal{O}_Y(1)\otimes W^*\xrightarrow F \mathcal{U}_1^*(1)\otimes W$'' to the subbundle $\mathcal{O}(E_1-H)\otimes\mathcal{O}_Y(1)\otimes W^*|_{Bl(\mathds{P}^2)}\subset \mathcal{U}_2(1)\otimes W^*|_{Bl(\mathds{P}^2)}$.
	
	\begin{lemma}
		The restriction of the composition map $\mathcal{U}_2(1)\otimes W^*\hookrightarrow S^2W\otimes\mathcal{O}_Y(1)\otimes W^*\rightarrow(\mathcal{U}_1^*(1)\otimes W)/\mathcal{U}_2^*(1)$ to $Bl(\mathds{P}^2)$ is zero, thus inducing a morphism from $L_5[-2]|_{Bl(\mathds{P}^2)}\otimes W^*$ to $(\mathcal{U}_1^*(1)\otimes W)/\mathcal{U}_2^*(1)|_{Bl(\mathds{P}^2)}$ and by adjunction a morphism $L_5[-2]|_{Bl(\mathds{P}^2)}\rightarrow((\mathcal{U}_1^*(1)\otimes W)/\mathcal{U}_2^*(1)|_{Bl(\mathds{P}^2)})\otimes W$ (the latter is the restriction of the morphsim $\beta$). \label{comp zero}
	\end{lemma}
	
	\begin{proof}
		We only need to prove that the component morphism of the map $\mathcal{U}_2(1)\otimes W^*|_{Bl(\mathds{P}^2)}\hookrightarrow S^2W\otimes\mathcal{O}_Y(1)\otimes W^*|_{Bl(\mathds{P}^2)}\rightarrow(\mathcal{U}_1^*(1)\otimes W)/\mathcal{U}_2^*(1)|_{Bl(\mathds{P}^2)}$ starting from $\mathcal{O}(E_1-H)\otimes\mathcal{O}_Y(1)\otimes W^*|_{Bl_{\mathds{P}^2}}\subset \mathcal{U}_2(1)\otimes W^*|_{Bl_{\mathds{P}^2}}$ is zero. This component morphism is the composition of the morphism $F_1:\mathcal{O}(E_1-H)\otimes\mathcal{O}_Y(1)\otimes W^*|_{Bl_{\mathds{P}^2}}\rightarrow \mathcal{U}_1^*(1)\otimes W|_{Bl_{\mathds{P}^2}}$ described above with the quotient map $q:\mathcal{U}_1^*(1)\otimes W|_{Bl_{\mathds{P}^2}}\rightarrow (\mathcal{U}_1^*(1)\otimes W)/\mathcal{U}_2^*(1)|_{Bl_{\mathds{P}^2}}$. Now we need to show that the image of $F_1$ lies in the kernel of $q$, which is the image of the inclusion $\mathcal{U}_2^*(1)|_{Bl(\mathds{P}^2)}\hookrightarrow(\mathcal{U}_1^*(1)\otimes W)|_{Bl(\mathds{P}^2)}$.
		It is trivial to check this last claim, since we have described explicitly a twisted version of the inclusion $\mathcal{U}_2^*|_{Bl(\mathds{P}^2)}\hookrightarrow(\mathcal{U}_1^*\otimes W)|_{Bl(\mathds{P}^2)}$ in the diagram $(**)$ below \cref{res}.
	\end{proof}
	
	\begin{lemma}
		The induced morphism from $L_5[-2]|_{Bl(\mathds{P}^2)}$ to $K^*(1)\otimes W=(\mathcal{U}_1^*(1)\otimes W)/\mathcal{U}_2^*(1)|_{Bl(\mathds{P}^2)}\otimes W$ is of maximal rank everywhere.  Moreover, the cokernel of this morphism 
		is isomorphic to $L_2[-1]|_{Bl(\mathds{P}^2)}$. \label{final cal}
	\end{lemma}
	\begin{proof}
		The morphism from $L_5[-2]|_{Bl(\mathds{P}^2)}$ to $(\mathcal{U}_1^*(1)\otimes W)/\mathcal{U}_2^*(1)|_{Bl(\mathds{P}^2)}\otimes W$ is induced (descended) from the morphism $S^2W\otimes \mathcal{O}_Y(1)|_{Bl(\mathds{P}^2)}\rightarrow(\mathcal{U}_1^*(1)\otimes W)/\mathcal{U}_2^*(1)|_{Bl(\mathds{P}^2)}\otimes W$.
		In the proof of \cref{res}, we have described the quotient map $S^2W\otimes \mathcal{O}_Y(1)|_{Bl(\mathds{P}^2)}\rightarrow L_5[-2]|_{Bl(\mathds{P}^2)}$, which splits into the direct sum of several surjective component maps.
		One of the component maps of that quotient map is exhibited as the first row of the diagram below:
		\[\begin{tikzcd}
			\langle z^2,xy\rangle\otimes\mathcal{O}_Y(1)|_{Bl(\mathds{P}^2)}\arrow[r,twoheadrightarrow] \arrow[d,hook] & \mathcal{O}(H-E_1)\otimes \mathcal{O}_Y(1)|_{Bl(\mathds{P}^2)}\arrow[d,hook]\\
			S^2W\otimes \mathcal{O}_Y(1)|_{Bl(\mathds{P}^2)}\arrow[r,twoheadrightarrow] & L_5[-2]|_{Bl(\mathds{P}^2)}
		\end{tikzcd}
		\]
		
		By the diagram preceding the \cref{comp zero}, we know that the image of $\langle z^2,xy\rangle\otimes\mathcal{O}_Y(1)|_{Bl(\mathds{P}^2)}\subset S^2W\otimes \mathcal{O}_Y(1)|_{Bl(\mathds{P}^2)}$ in $(\mathcal{U}_1^*(1)\otimes W)|_{Bl(\mathds{P}^2)}\otimes W$ actually lies in the subbundle 
		\begin{equation}\label{longbundle}
			\begin{split}
				[&\mathcal{O}(H)\otimes\langle z\rangle\otimes \langle x\rangle\bigoplus \mathcal{O}(2H-E_1-E_2-E_3)\otimes \langle x\rangle\otimes \langle x\rangle \\
				\bigoplus&\mathcal{O}(H)\otimes\langle y\rangle\otimes \langle y\rangle\bigoplus \mathcal{O}(2H-E_1-E_2-E_3)\otimes \langle z\rangle\otimes \langle y\rangle \\
				\bigoplus
				&\mathcal{O}(H)\otimes\langle x\rangle\otimes \langle z\rangle\bigoplus \mathcal{O}(2H-E_1-E_2-E_3)\otimes \langle y\rangle\otimes \langle z\rangle]\otimes \mathcal{O}_Y(1)|_{Bl(\mathds{P}^2)}
			\end{split}
		\end{equation}
		
		Because of the above observations, the image of $\mathcal{O}(H-E_1)\otimes\mathcal{O}_Y(1)|_{Bl(\mathds{P}^2)}\subset L_5[-2]|_{Bl(\mathds{P}^2)}$ in $(\mathcal{U}_1^*(1)\otimes W)/\mathcal{U}_2^*(1)|_{Bl(\mathds{P}^2)}\otimes W$ lie in the subbundle $G$, which denotes the quotient of \eqref{longbundle}
		by the bundle $[(\mathcal{O}(H-E_3)\otimes \langle x\rangle)\bigoplus(\mathcal{O}(H-E_2)\otimes \langle y\rangle)\bigoplus(\mathcal{O}(H-E_1)\otimes \langle z\rangle)]\otimes \mathcal{O}_Y(1)|_{Bl(\mathds{P}^2)}$ embedded into \eqref{longbundle} via the twisted version of the component morphisms of the map described in diagram $(**)$. One can easily see that $G$ again splits and is equal to $[\mathcal{O}(2H-E_1-E_2)\otimes\langle x\rangle\bigoplus\mathcal{O}(2H-E_1-E_3)\otimes\langle y\rangle\bigoplus\mathcal{O}(2H-E_2-E_3)\otimes\langle z\rangle]\otimes \mathcal{O}_Y(1)|_{Bl(\mathds{P}^2)}$ 
		
		By calculation, the map $\mathcal{O}(H-E_1)\otimes\mathcal{O}_Y(1)|_{Bl(\mathds{P}^2)}\rightarrow G$ together with the cokernel map are as follows:
		\[
		\adjustbox{max width=\textwidth}{
			\begin{tikzcd}[column sep=2em , row sep=0.5em] 
				&&\mathcal{O}(2H-E_1-E_2)\otimes\mathcal{O}_Y(1)|_{Bl(\mathds{P}^2)}\ar[drr,"0"]\ar[dddrr,"a" near start]&&\\
				&&\oplus &&\mathcal{O}(2H-E_3)\otimes\mathcal{O}_Y(1)|_{Bl(\mathds{P}^2)}\\
				\mathcal{O}(H-E_1)\otimes\mathcal{O}_Y(1)|_{Bl(\mathds{P}^2)}\ar[uurr,"4c"]\ar[rr,"4a"]\ar[ddrr,"4a"]&&\mathcal{O}(2H-E_1-E_3)\otimes\mathcal{O}_Y(1)|_{Bl(\mathds{P}^2)}\ar[urr,"-1" near end]\ar[drr,"0" near start]&&\oplus\\
				&&\oplus&&\mathcal{O}(3H-E_1-2E_2-E_3)\otimes\mathcal{O}_Y(1)|_{Bl(\mathds{P}^2)}\\
				&&\mathcal{O}(2H-E_2-E_3)\otimes\mathcal{O}_Y(1)|_{Bl(\mathds{P}^2)}\ar[uuurr,"1" near start]\ar[urr,"-c"]&&
			\end{tikzcd}
		}
		\]
		
		This implies that $\mathcal{O}(H-E_1)\otimes\mathcal{O}_Y(1)|_{Bl(\mathds{P}^2)}\rightarrow G$ is of constant rank 1, and the above sequence is a short exact sequence:
		
		To check the injectivity on the left, notice that the degenerate locus of the map $\mathcal{O}(H-E_1)\otimes\mathcal{O}_Y(1)|_{Bl(\mathds{P}^2)}$
		$\xrightarrow{4c}\mathcal{O}(2H-E_1-E_2)\otimes\mathcal{O}_Y(1)|_{Bl(\mathds{P}^2)}$ is equal to the zero locus of the global section $4c\in \mathrm{H}^0(Bl(\mathds{P}^2),\mathcal{O}(H-E_2))$ and is thus equal to the union of the strict transform of $\{c=0\}$ with $E_1$. The degenerate loci of the other two morphisms starting from $\mathcal{O}(H-E_1)\otimes\mathcal{O}_Y(1)|_{Bl(\mathds{P}^2)}$ are equal to the union of the strict transform of $\{a=0\}$ with $E_2$, and the union of the strict transform of $\{a=0\}$ with $E_1$ respectively. The intersection of the three degenerate loci is empty and implies that the above sequence is injective on the left.
		
		The above sequence is obviously surjective on the right. Finally one can check the exactness in the middle term by verifying that the multiplicity of total Chern classes holds.
		
		Consequently, $L_5[-2]|_{Bl(\mathds{P}^2)}\rightarrow(\mathcal{U}_1^*(1)\otimes W)/\mathcal{U}_2^*(1)|_{Bl(\mathds{P}^2)}\otimes W$ is of maximal rank everywhere because this map is the direct sum of several maps of the form $\mathcal{O}(H-E_1)\otimes\mathcal{O}_Y(1)|_{Bl(\mathds{P}^2)}\rightarrow G$, and the cokernel is equal to $[\mathcal{O}(2H-E_1)\oplus\mathcal{O}(2H-E_2)\oplus\mathcal{O}(2H-E_3)]\otimes\mathcal{O}_Y(1)|_{Bl(\mathds{P}^2)}\bigoplus [\mathcal{O}(3H-2E_1-E_2-E_3)\oplus\mathcal{O}(3H-E_1-2E_2-E_3)\oplus\mathcal{O}(3H-E_1-E_2-2E_3)]\otimes\mathcal{O}_Y(1)|_{Bl(\mathds{P}^2)}$, which is equal to the restriction of $L_2[-1]$ to $Bl(\mathds{P}^2)$ (see \cref{res}).
	\end{proof}
	\begin{Corollary}
		$L_3[-1]|_{Bl(\mathds{P}^2)}\cong L_2[-1]|_{Bl(\mathds{P}^2)}$.\label{same restriction}
	\end{Corollary}
	\begin{proof}
		Recall that the morphism from $L_5[-2]$ to $[(\mathcal{U}_1^*(1)\otimes W)/\mathcal{U}_2^*(1)]\otimes W$ is the morphism $\beta$ and share the same cokernel with the morphism $\alpha$ (see the proof of \cref{big proof}). By \cref{big proof}, we know that the cokernel of $\alpha$ is the vector bundle $L_3[-1]$. By the last Lemma, the cokernel of $\beta|_{Bl(\mathds{P}^2)}$ is isomorphic to $L_2[-1]|_{Bl(\mathds{P}^2)}$ and this finishes the proof.
	\end{proof}
	
	\begin{proposition}
		Over $Y$ globally, we have an isomorphism $L_3[-1]\cong L_2[-1]$. As a result, the exceptional collections \eqref{1 exc}, \eqref{4 exc}, \eqref{5 exc} generate the same triangulated subcategory in $\mathrm{D}^b(Y)$. \label{same tri subcategory}
	\end{proposition}
	\begin{proof}
		Recall that in \cref{1 unique morphism}, we have $\mathrm{RHom}^{\bullet}(\mathcal{U}_1^*\otimes \mathcal{U}_2(2), \mathcal{U}_2(1)[3])=\mathds{C}[0]$ and that $L_6=\mathcal{U}_2(1)[3]$ in our notation. Since $L_3$ is obtained from $L_6$ by right mutations and that $L_2$ is obtained from $\mathcal{U}_1^*\otimes \mathcal{U}_2(2)$ by left mutations, we know by \cref{mutation and Rhom} that $\mathrm{RHom}^{\bullet}(L_2[-1], L_3[-1])=\mathds{C}[0]$. This unique nonzero morphism between the vector bundles $L_2[-1]$ and $L_3[-1]$ must be $SL(W)$-equivariant because $SL(W)$ is a semisimple group and has no nontrivial group character. Thus the rank of the bundle map $L_2[-1]\rightarrow L_3[-1]$ must be constant along $SL(W)$-orbits.
		
		We know that $L_3[-1]|_{Bl(\mathds{P}^2)}=L_2[-1]|_{Bl(\mathds{P}^2)}=[\mathcal{O}(2H-E_1)\oplus\mathcal{O}(2H-E_2)\oplus\mathcal{O}(2H-E_3)]\otimes\mathcal{O}_Y(1)|_{Bl(\mathds{P}^2)}\bigoplus $
		$[\mathcal{O}(3H-2E_1-E_2-E_3)\oplus\mathcal{O}(3H-E_1-2E_2-E_3)\oplus\mathcal{O}(3H-E_1-E_2-2E_3)]\otimes\mathcal{O}_Y(1)|_{Bl(\mathds{P}^2)}$. It is easy to see that any endomorphism of this vector bundle (over $Bl(\mathds{P}^2)$) must be the direct sum of the endomorphisms of the six line bundle summands. In particular, the map $L_2[-1]|_{Bl(\mathds{P}^2)}\rightarrow L_3[-1]|_{Bl(\mathds{P}^2)}$ must be of constant rank. Moreover, the rank can only be 3 or 6: the rank cannot be zero because we consider the unique nonzero morphism (up to scalar) between vector bundles $L_2[-1]$ and $L_3[-1]$, the rank cannot be $1,2,4,5$ because we have the alternating group $\mathrm{A}_3=\{(1,2,3),(2,3,1),(3,1,2)\}\subset SL(W)$ action on the subvariety $Bl(\mathds{P}^2)$, which permutes the 6 line bundle summands of $L_3[-1]|_{Bl(\mathds{P}^2)}$ 3 by 3.
		If the rank is 3 or 6 over $Bl(\mathds{P}^2)$, then the map is of constant rank 3 or 6 correspondingly over $Y$ because $Bl(\mathds{P}^2)$ contains elements in the biggest $SL(W)$-orbit and the two smallest $SL(W)$-orbits. 
		
		If the rank is 6, we are done. If the rank is 3, then the image of $L_2[-1]\rightarrow L_3[-1]$ is a rank 3 vector subbundle whose restriction to $Bl(\mathds{P}^2)$ is equal to either $[\mathcal{O}(2H-E_1)\oplus\mathcal{O}(2H-E_2)\oplus\mathcal{O}(2H-E_3)]\otimes\mathcal{O}_Y(1)|_{Bl(\mathds{P}^2)}$ or $ [\mathcal{O}(3H-2E_1-E_2-E_3)\oplus\mathcal{O}(3H-E_1-2E_2-E_3)\oplus\mathcal{O}(3H-E_1-E_2-2E_3)]\otimes\mathcal{O}_Y(1)|_{Bl(\mathds{P}^2)}$. Recall that the Picard rank of $Y$ is one, the determinant of the rank 3 subbundle must be equal to $\mathcal{O}_Y(s)$ for some $s\in \mathds{Z}$. The determinant of the restriction of the rank 3 subbundle is equal to either $\det([\mathcal{O}(2H-E_1)\oplus\mathcal{O}(2H-E_2)\oplus\mathcal{O}(2H-E_3)]\otimes\mathcal{O}_Y(1)|_{Bl(\mathds{P}^2)})=\mathcal{O}(6H-E_1-E_2-E_3)\otimes\mathcal{O}_Y(3)|_{Bl(\mathds{P}^2)}$ or $\det([\mathcal{O}(3H-2E_1-E_2-E_3)\oplus\mathcal{O}(3H-E_1-2E_2-E_3)\oplus\mathcal{O}(3H-E_1-E_2-2E_3)]\otimes\mathcal{O}_Y(1)|_{Bl(\mathds{P}^2)})=\mathcal{O}(9H-4E_1-4E_2-4E_3)\otimes\mathcal{O}_Y(3)|_{Bl(\mathds{P}^2)}$, which are not in the form $\mathcal{O}_Y(s)|_{Bl(\mathds{P}^2)}=\mathcal{O}(3H-E_1-E_2-E_3)^{\otimes s}$! We come to a contradiction and conclude that $L_2[-1]\rightarrow L_3[-1]$ is of constant rank 6 and is thus an isomorphism.
		
		Recall that $L_3$ is obtained from $L_6=\mathcal{U}_2(1)[3]$ by right mutation across $\langle\mathcal{O}_Y(1),\mathcal{U}_2^*(1),\mathcal{U}_1^*(1)\rangle$ and that $L_2$ is obtained from $\mathcal{U}_1^*\otimes \mathcal{U}_2(2)$ by left mutation across $\langle \mathfrak{sl}(\mathcal{U}_1^*)(2),\mathcal{O}_Y(2)\rangle$. Consequently, we know that \eqref{1 exc} and \eqref{4 exc} generate the same triangulated subcategory.
		Another statement now follows easily using the symmetry functor $T$ which maps $\mathcal{U}_1^*\otimes \mathcal{U}_2(2)$ to $\mathcal{U}_1^*\otimes \mathcal{U}_2^*$.
	\end{proof}

	\subsection{Fullness of Exceptional Sequences}
	\label{subsection: proving fullness}

	In this subsection we shall use the covering argument as in \cref{covering} to show that all the exceptional sequences given in the last subsections are full.
	
	We will use the following covering family of $Y$, which is obtained by taking zero locus $J$ of a general element in $\mathrm{H}^0(Y,\mathcal{U}_1^*)=S^{2,1}W^*$. Since $Y\subset \mathrm{Gr}(2,S^{2,1}W)=\mathrm{Gr}(2,8)$ is defined by a general section of $Q^*(1)$ (where $Q$ is the universal quotient bundle over $\mathrm{Gr}(2,8)$), this is equivalent to intersecting $Y$ with a general subvariety $\mathrm{Gr}(2,7)$ inside $\mathrm{Gr}(2,8)$, or equivalent to taking the zero locus of a general global section of $Q^*(1)|_{\mathrm{Gr}(2,7)}$ over $\mathrm{Gr}(2,7)$. If we denote the universal quotient bundle on $\mathrm{Gr}(2,7)$ by $Q'$, we know $Q|_{\mathrm{Gr}(2,7)}=Q'\oplus \mathcal{O}_{\mathrm{Gr}(2,7)}$. We can conclude that the subvariety $J$ is isomorphic to a hyperplane section of the zero locus of a general global section of $(Q')^*(1)$ over $\mathrm{Gr}(2,7)$.
	
	We would like to use such $J$ to cover $Y$ with $J$ being smooth. Of course $J$ is smooth when we take a general section of $\mathcal{U}_1^*$. The remaining part is to use the $SL(W)$ action and prove that the zero locus of a general global section of $\mathcal{U}_1^*$ has nonempty intersections with all the five $SL(W)$-orbits. We only need to check that the intersections of $J$ with the two smallest orbits are nonempty, and this is done by calculating $(d_2.(-3c_2d_2+6d_2^2))=(d_2.(3c_2d_2-3d_2^2))=3\neq 0$ based on our previous calculations of the fundamental classes of the two smallest orbits (the fundamental class of $J$ is equal to $d_2$).
	
	A general global section of $(Q')^*(1)$ over $\mathrm{Gr}(2,7)$ is isomorphic to the $\mathrm{G}_2$-Grassmannian $\mathrm{G}_2\mathrm{Gr}(2,7)$
	and it is proved in the paper \cite{hyperplaneDer} by Kuznetsov that $\mathrm{G}_2\mathrm{Gr}(2,7)$ has a homological projective dual variety, which allows us to describe the derived category of a general hyperplane section $J$.
	\begin{thm}[\cite{hyperplaneDer}]
		The collection $\langle \mathcal{O}_J,\mathcal{U}_1^*|_J,\mathcal{O}_J(1),\mathcal{U}_1^*(1)|_J\rangle$ is an exceptional collection in $\mathrm{D}^b(J)$. Its right orthogonal complement in $\mathrm{D}^b(J)$ is equivalent to the derived category of two isolated points, and is in particular spanned by two exceptional objects.\label{kuz theorem}
	\end{thm}
	
	Over $\mathrm{Gr}(3, S^2W)$, we have two interesting subbundles of trivial bundles, namely $\mathcal{U} \subset S^2W \otimes \mathcal{O}$ and $\mathcal{U}^{\perp} \subset S^2W^* \otimes \mathcal{O}$. Notice that $\mathcal{U}|_J=\mathcal{U}_2|_J$ and that we have an exact sequence $0\rightarrow\mathcal{U}^{\perp}|_Y\rightarrow S^2W^* \otimes \mathcal{O}_Y\rightarrow \mathcal{U}_2^*\rightarrow 0$.
	
	\begin{proposition}
		The following sequences are full exceptional sequences in $\mathrm{D}^b(J)$:
		$$\langle \mathcal{U}_2|_J, \mathcal{U}^{\perp}|_J, \mathcal{O}_J, \mathcal{U}_1^*|_J, \mathcal{O}_J(1),\mathcal{U}_1^*(1)|_J\rangle,$$
		$$\langle \mathcal{U}^{\perp}|_J, \mathcal{O}_J, \mathcal{U}_1^*|_J, \mathcal{O}_J(1),\mathcal{U}_1^*(1)|_J,\mathcal{U}_2(2)|_J \rangle.$$
	\end{proposition}
	\begin{proof}
		We begin with proving that the first sequence above is exceptional. 
		
		We want to prove that in $\mathrm{D}^b(Y)$ we have $\mathrm{RHom}^{\bullet}(\mathcal{O}_J,\mathcal{U}_2|_J)=\mathrm{RHom}^{\bullet}(\mathcal{U}_1^*|_J,\mathcal{U}_2|_J)=\mathrm{RHom}^{\bullet}(\mathcal{O}_J(1),\mathcal{U}_2|_J)=\mathrm{RHom}^{\bullet}(\mathcal{U}_1^*(1)|_J,\mathcal{U}_2|_J)=0$. There is a Koszul complex:
		$$0\rightarrow \wedge^2(\mathcal{U}_1)=\mathcal{O}_Y(-1)\rightarrow\mathcal{U}_1\rightarrow \mathcal{O}_Y\rightarrow\mathcal{O}_J\rightarrow0.$$
		After tensoring with $(\mathcal{O}_Y)^*\otimes\mathcal{U}_2,(\mathcal{U}_1^*)^*\otimes\mathcal{U}_2$,$(\mathcal{O}_Y(1))^*\otimes\mathcal{U}_2,(\mathcal{U}_1^*(1))^*\otimes\mathcal{U}_2$ respectively, we only need to prove $$\mathrm{RHom}^{\bullet}(\mathcal{O}_Y,\mathcal{U}_2)=\mathrm{RHom}^{\bullet}(\mathcal{U}_1^*,\mathcal{U}_2)=\mathrm{RHom}^{\bullet}(\mathcal{O}_Y(1),\mathcal{U}_2)=\mathrm{RHom}^{\bullet}(\mathcal{U}_1^*(1),\mathcal{U}_2)=\mathrm{RHom}^{\bullet}(\mathcal{U}_1^*\otimes \mathcal{U}_1^*,\mathcal{U}_2)=0,$$  $$\mathrm{RHom}^{\bullet}(\mathcal{O}_Y(2),\mathcal{U}_2)=\mathrm{RHom}^{\bullet}(\mathcal{U}_1^*(2),\mathcal{U}_2)=\mathrm{RHom}^{\bullet}(\mathcal{U}_1^*\otimes \mathcal{U}_1^*(1),\mathcal{U}_2)=0.$$ The first set of vanishing follow directly from the orthogonal relations that we have seen from \eqref{1 exc} and \eqref{2 exc}, for example, $\mathrm{RHom}^{\bullet}(\mathcal{U}_1^*\otimes \mathcal{U}_1^*,\mathcal{U}_2)=\mathrm{RHom}^{\bullet}(\mathcal{U}_1^*\otimes \mathcal{U}_1^*(1),\mathcal{U}_2(1))=\mathrm{RHom}^{\bullet}(\mathcal{U}_1\otimes \mathcal{U}_1^*(2),\mathcal{U}_2(1))=\mathrm{RHom}^{\bullet}(\mathcal{O}_Y(2)\oplus\mathfrak{sl}(\mathcal{U}_1)(2),\mathcal{U}_2(1))=0$. The second set of vanishing follow from the orthogonal relations that we have seen after we apply the Grothendieck-Verdier duality, for example $\mathrm{RHom}^{\bullet}(\mathcal{U}_1^*\otimes \mathcal{U}_1^*(1),\mathcal{U}_2)=\mathrm{RHom}^{\bullet}(\mathcal{U}_2,\mathcal{U}_1\otimes \mathcal{U}_1^*(2)\otimes \omega_Y[6])^*=\mathrm{RHom}^{\bullet}(\mathcal{U}_2(1),\mathcal{U}_1\otimes \mathcal{U}_1^*[6])^*=0$.
		
		We also want to prove \begin{equation*}
			\mathrm{RHom}^{\bullet}(\mathcal{O}_J,\mathcal{U}^{\perp}|_J)=\mathrm{RHom}^{\bullet}(\mathcal{U}_1^*|_J,\mathcal{U}^{\perp}|_J)=\mathrm{RHom}^{\bullet}(\mathcal{O}_J(1),\mathcal{U}^{\perp}|_J)=\mathrm{RHom}^{\bullet}(\mathcal{U}_1^*(1)|_J,\mathcal{U}^{\perp}|_J)=0
		\end{equation*}
		Thanks to the Koszul complex, we only need to prove that $\mathrm{RHom}^{\bullet}(\mathcal{O}_Y,\mathcal{U}^{\perp}|_Y)=0$ and $\mathrm{RHom}^{\bullet}(\mathcal{U}_1^*,\mathcal{U}^{\perp}|_Y)=\mathrm{RHom}^{\bullet}(\mathcal{O}_Y(1),\mathcal{U}^{\perp}|_Y)=\mathrm{RHom}^{\bullet}(\mathcal{U}_1^*(1),\mathcal{U}^{\perp}|_Y)=\mathrm{RHom}^{\bullet}(\mathcal{U}_1^*\otimes \mathcal{U}_1^*,\mathcal{U}^{\perp}|_Y)=\mathrm{RHom}^{\bullet}(\mathcal{O}_Y(2),\mathcal{U}^{\perp}|_Y)=$
		$\mathrm{RHom}^{\bullet}$
		$(\mathcal{U}_1^*(2),\mathcal{U}^{\perp}|_Y)=\mathrm{RHom}^{\bullet}(\mathcal{U}_1^*\otimes \mathcal{U}_1^*(1),\mathcal{U}^{\perp}|_Y)=0$. To do this, recall that we have an exact sequence \begin{equation}
			0\rightarrow\mathcal{U}^{\perp}|_Y \rightarrow S^2W^* \otimes \mathcal{O}_Y\rightarrow \mathcal{U}_2^*\rightarrow 0.\label{last new exact sequence}
		\end{equation}
		The vanishing of $\mathrm{RHom}^{\bullet}(\mathcal{O}_Y,\mathcal{U}^{\perp}|_Y)=0$ comes from the fact that the morphism $\mathrm{RHom}^{\bullet}(\mathcal{O}_Y,S^2W^* \otimes \mathcal{O}_Y)\rightarrow$
		$\mathrm{RHom}^{\bullet}(\mathcal{O}_Y, \mathcal{U}_2^*)$ is an isomorphism (see \cref{global section} and recall that $Y$ is a rational variety). To verify the other vanishing of cohomology, we only need to verify that the corresponding vanishing equalities hold when we replace $\mathcal{U}^{\perp}|_Y$ by $\mathcal{O}_Y$ and $\mathcal{U}_2^*$ (because of the above exact sequence). Then these vanishing equalities follow from the orthogonal relations in \eqref{1 exc} and \eqref{2 exc}.
		
		The fact that $\mathrm{RHom}^{\bullet}(\mathcal{U}_2|_J,\mathcal{U}_2|_J)=\mathds{C}[0]$ also follows from a Koszul complex argument: $\mathrm{RHom}^{\bullet}(\mathcal{U}_2,\mathcal{U}_2)=\mathrm{RHom}^{\bullet}(\mathcal{U}_2(1),\mathcal{U}_2(1))=\mathds{C}[0]$, $\mathrm{RHom}^{\bullet}(\mathcal{U}_2(1),\mathcal{U}_2)=\mathrm{RHom}^{\bullet}(\mathcal{U}_2(2),\mathcal{U}_2(1))=0$ and $\mathrm{RHom}^{\bullet}(\mathcal{U}_2\otimes \mathcal{U}_1^*,\mathcal{U}_2)=\mathrm{RHom}^{\bullet}(\mathcal{U}_2\otimes \mathcal{U}_1^*(2),\mathcal{U}_2(2))=0$ because of the orthogonal relations in \eqref{1 exc} and \eqref{4 exc}.
		
		To prove $\mathrm{RHom}^{\bullet}(\mathcal{U}^{\perp}|_J,\mathcal{U}_2|_J)=0$, we only need to prove that $\mathrm{RHom}^{\bullet}(\mathcal{O}_J,\mathcal{U}_2|_J)=\mathrm{RHom}^{\bullet}(\mathcal{U}_2^*|_J,\mathcal{U}_2|_J)=0$, because of the short exact sequence \eqref{last new exact sequence}.
		We have seen in the above paragraphs that $\mathrm{RHom}^{\bullet}(\mathcal{O}_J,\mathcal{U}_2|_J)=0$. We apply again the Koszul complex argument to show that $\mathrm{RHom}^{\bullet}(\mathcal{U}_2^*|_J,\mathcal{U}_2|_J)=0$: for example $\mathrm{RHom}^{\bullet}(\mathcal{U}_2^*\otimes \mathcal{U}_1^*,\mathcal{U}_2)=\mathrm{RHom}^{\bullet}(\mathcal{U}_2^*,\mathcal{U}_1\otimes\mathcal{U}_2)=\mathrm{RHom}^{\bullet}(\mathcal{U}_1\otimes \mathcal{U}_2,\mathcal{U}_2^*(-3)[6])^*=\mathrm{RHom}^{\bullet}(\mathcal{U}_1^*\otimes \mathcal{U}_2(2),\mathcal{U}_2^*[6])^*=0$ by Serre duality and \cref{symmetry}.
		
		It is a similar routine check to see that $\mathrm{RHom}^{\bullet}(\mathcal{U}^{\perp}|_J,\mathcal{U}^{\perp}|_J)=\mathds{C}[0]$. From the above calculations, we see that the first sequence given in the statement is exceptional and that the two exceptional objects $\mathcal{U}_2|_J, \mathcal{U}^{\perp}|_J$ lie in the right orthogonal complement of $\langle \mathcal{O}_J,\mathcal{U}_1^*|_J,\mathcal{O}_J(1),\mathcal{U}_1^*(1)|_J\rangle$ in $\mathrm{D}^b(J)$. It follows from Kuznetsov's \cref{kuz theorem} that the first sequence is a full exceptional sequence for $\mathrm{D}^b(J)$.
		
		As a result, the second sequence is also a full exceptional sequence because of the fullness of the first sequence, Grothendieck-Verdier duality and the fact that $K_J\cong \mathcal{O}_J(-2)$.
	\end{proof}
	
	Now we are ready to prove the main Theorem.
	
	\begin{thm}
		All the exceptional sequences \eqref{1 exc}, \eqref{2 exc}, \eqref{3 exc}, \eqref{4 exc}, \eqref{5 exc} in $\mathrm{D}^b(Y)$ are full.\label{fullnessthm}
	\end{thm}
	
	\begin{proof}
		From \cref{same tri subcategory}, we have seen that all these exceptional sequences generate the same triangulated subcategory of $\mathrm{D}^b(Y)$, which we denote by $\mathcal{C}$.
		We need to prove that if $\mathcal{S}^.\in \mathrm{D}^b(Y)$ belongs to $\mathcal{C}^{\perp}$, then $\mathcal{S}^.$ is the zero object.
		Under this assumption, we only need to prove $\mathcal{S}^.|_J$ is the zero object in $\mathrm{D}^b(J)$ by
		\cref{covering}, which is equivalent to $\mathrm{RHom}^{\bullet}(\mathcal{U}^{\perp}|_J,\mathcal{S}^.|_J)=\mathrm{RHom}^{\bullet}(\mathcal{O}_J,\mathcal{S}^.|_J)=\mathrm{RHom}^{\bullet}(\mathcal{U}_1^*|_J,\mathcal{S}^.|_J)=\mathrm{RHom}^{\bullet}(\mathcal{O}_J(1),\mathcal{S}^.|_J)=\mathrm{RHom}^{\bullet}(\mathcal{U}_1^*(1)|_J,\mathcal{S}^.|_J)=\mathrm{RHom}^{\bullet}(\mathcal{U}_2(2)|_J,\mathcal{S}^.|_J)=0$ by the last Proposition. By the Koszul complex for $\mathcal{O}_J$, we only need to prove that $\mathcal{S}^.$ is right orthogonal to 
		$$  \mathcal{U}^{\perp}|_Y, \mathcal{O}_Y, \mathcal{U}_1^*, \mathcal{O}_Y(1),\mathcal{U}_1^*(1),\mathcal{U}_2(2),$$
		$$  \mathcal{U}^{\perp}|_Y\otimes \mathcal{U}_1^*, \mathcal{U}_1^*, \mathcal{U}_1^*\otimes \mathcal{U}_1^*,  \mathcal{U}_1^*(1),\mathcal{U}_1^*\otimes \mathcal{U}_1^*(1),\mathcal{U}_1^*\otimes \mathcal{U}_2(2),$$
		$$  \mathcal{U}^{\perp}(1)|_Y, \mathcal{O}_Y(1), \mathcal{U}_1^*(1), \mathcal{O}_Y(2),\mathcal{U}_1^*(2),\mathcal{U}_2(3).$$
		We know that all these objects belong to $\mathcal{C}$ thanks to the short exact sequence $0\rightarrow\mathcal{U}^{\perp}|_Y\rightarrow S^2W^* \otimes \mathcal{O}_Y\rightarrow \mathcal{U}_2^*\rightarrow 0$. This shows $\mathcal{S}^.|_J=0$ whenever $\mathcal{S}^.\in \mathcal{C}^{\perp}$ and by \cref{covering}, we know $\mathcal{S}^.=0$.
	\end{proof}
	
	\begin{Corollary}
		The left mutation of $\mathcal{U}_1^*\otimes\mathcal{U}_2(2)$ across $\langle \mathcal{O}_Y(1),\mathcal{U}_2^*(1),\mathcal{U}_1^*(1),\mathcal{U}_2(2),\mathfrak{sl}(\mathcal{U}_1^*)(2),\mathcal{O}_Y(2)\rangle$ is equal to $\mathcal{U}_2(1)[3]$.
	\end{Corollary}
	
	\begin{proof}
		We know that the left mutation of $\mathcal{U}_1^*\otimes\mathcal{U}_2(2)$ across $\langle \mathcal{O}_Y(1),\mathcal{U}_2^*(1),\mathcal{U}_1^*(1),\mathcal{U}_2(2),\mathfrak{sl}(\mathcal{U}_1^*)(2),\mathcal{O}_Y(2)\rangle$ lies in the subcategory generated by $\mathcal{U}_2(1)$ by the comparison of the two full exceptional sequences \eqref{1 exc} and \eqref{4 exc}. Since both the result of this left mutation and $\mathcal{U}_2(1)$ are exceptional objects and $\mathrm{RHom}^{\bullet}(\mathcal{U}_1^*\otimes \mathcal{U}_2(2), \mathcal{U}_2(1))=\mathds{C}[-3]$ by \cref{1 unique morphism}, we arrive at the conclusion.
	\end{proof}
	
	For the reader's convenience, we summarize the process of the mutations as follows:
	\begin{gather*}  
		0\rightarrow \mathcal{U}_2(1)=L_6[-3]\rightarrow S^2W\otimes\mathcal{O}_Y(1)\rightarrow L_5[-2]\rightarrow 0, \\
		0\rightarrow \mathcal{U}_2^*(1)^{\oplus 3}\rightarrow L_4[-2]\rightarrow L_5[-2]\rightarrow 0,    \\
		0\rightarrow L_4[-2]\rightarrow (\mathcal{U}_1^*(1)\otimes W)^{\oplus 3}\rightarrow L_3[-1]\rightarrow 0, \\
		0\rightarrow L_3[-1]\rightarrow L_2[-1]\rightarrow 0,\\
		0\rightarrow L_2[-1]\rightarrow \mathcal{U}_1^*\otimes \mathcal{U}_1(2)\otimes W^*\rightarrow \mathcal{U}_1^*\otimes \mathcal{U}_2(2)\rightarrow 0.
	\end{gather*}

	\appendix

	\section{Julia code for Hirzebruch-Riemann-Roch}
	\label{appendix: Sage code}

	In order to verify the Hirzebruch-Riemann-Roch calculations in the main text, the reader may use the method \texttt{integral} that is defined in the Julia package QuiverTools \cite{BFP:QuiverTools} and computes the Euler characteristic of a sheaf with the given Chern character.
	The reader can do so by copying the code below into a Julia notebook and adjusting the last line when necessary.
	
\begin{lstlisting}[numbersep=0pt]
  using QuiverTools;
  Q = mKronecker_quiver(3); d = [2,3]; Y = QuiverModuliSpace(Q, d);
  U1     = Chern_character_universal_bundle(Y, 1);
  U1dual = dual_Chern_character(Y, U1);
  U2     = Chern_character_universal_bundle(Y, 2);
  U2dual = dual_Chern_character(Y, U2);
  O1     = Chern_character_line_bundle(Y, [3,-2]);
  O1dual = dual_Chern_character(Y, O1);
		
  integral(Y, U2dual*U1dual*O1dual^2)
\end{lstlisting}
	
	Running \texttt{integral(Y, U1\textasciicircum\textbf{a} * U1dual\textasciicircum\textbf{b} * U2\textasciicircum\textbf{c} * U2dual\textasciicircum\textbf{d} * O1\textasciicircum\textbf{e} * O1dual\textasciicircum\textbf{f})} returns the Euler characteristic of $\mathcal U_1^{\otimes a} \otimes \mathcal (U_1^*)^{\otimes b} \otimes \mathcal U_2^{\otimes c} \otimes \mathcal (U_2^*)^{\otimes d} \otimes \mathcal O_Y(e-f)$, where the reader will replace $a,b,c,d,e,f$ with natural numbers.
	

	\printbibliography


	\texttt{Svetlana.Makarova@anu.edu.au} \\
	Mathematical Sciences Institute, Australian National University, Canberra ACT 2600, Australia
	
	\texttt{Junyu.Meng@math.univ-toulouse.fr} \\
	Université Paul Sabatier, Institut de Mathématiques de Toulouse, 118, Route de Narbonne, F-31062 Toulouse Cedex 9, France
\end{document}